\documentclass[12pt,a4paper,amsmath,amssymb]{amsart}

\allowdisplaybreaks

\usepackage[latin1]{inputenc}
\usepackage{amsthm}
\usepackage{latexsym}
\usepackage{amssymb}
\usepackage{amsfonts}
\usepackage{eufrak}
\usepackage{color} 
\usepackage{graphicx}
\usepackage{sidecap}
\usepackage{adjustbox}
\usepackage{subcaption}
\usepackage{hyperref}
\usepackage{enumerate}
\usepackage{amsaddr}
\usepackage[abs]{overpic}
\usepackage{pict2e}
\usepackage{doi}

\newtheorem{theorem}{Theorem}
\newtheorem{proposition}[theorem]{Proposition}
\newtheorem{lemma}[theorem]{Lemma}
\newtheorem{corollary}[theorem]{Corollary}

\theoremstyle{definition}
\newtheorem{definition}[theorem]{Definition}

\newtheorem{remark}[theorem]{Remark}

\definecolor{lighter}{RGB}{10.0,146.0,35.0}
\definecolor{darker}{RGB}{2.0,100.0,20.0}
\definecolor{sininen}{RGB}{89.5,101.8,238.1}
\definecolor{punainen}{RGB}{227.0,12.0,12.0}

\newcommand{\lighter}[1]{\textcolor{lighter}{#1}}
\newcommand{\darker}[1]{\textcolor{darker}{#1}}
\newcommand{\punainen}[1]{\textcolor{punainen}{#1}}
\newcommand{\sininen}[1]{\textcolor{sininen}{#1}}

\newcommand{\mc}[1]{\mathcal{#1}}

\newcommand{\mb}[1]{\mathbb{#1}}

\newcommand{\N}{\mathbb N}
\newcommand{\R}{\mathbb R}

\newcommand{\Q}{\mathbb Q}
\newcommand{\T}{\mathbb T}

\newcommand{\tr}[1]{\mathrm{Tr}\left(#1\right)} 
\def\<{\langle}
\def\>{\rangle}

\newcommand{\fii}{\varphi}

\newcommand{\tj}{\vartheta}

\newcommand{\V}{\mathcal{V}}

\DeclareMathOperator*{\esssup}{ess\,sup}

\makeatletter
\DeclareRobustCommand\bigop[1]{%
	\mathop{\vphantom{\sum}\mathpalette\bigop@{#1}}\slimits@
}
\newcommand{\bigop@}[2]{%
	\vcenter{%
		\sbox\z@{$#1\sum$}%
		\hbox{\resizebox{\ifx#1\displaystyle.9\fi\dimexpr\ht\z@+\dp\z@}{!}{$\m@th#2$}}%
	}%
}
\makeatother
\newcommand{\bigboxplus}{\DOTSB\bigop{\boxplus}}

\newcommand{\rleq}{\preceq}
\newcommand{\rgeq}{\succeq}

\begin{document}
	
	\title[Matrix Majorization in Large Samples with Varying Support Restrictions]{Matrix Majorization in Large Samples\\with Varying Support Restrictions}

    \author{Frits Verhagen}
	\address{Centre for Quantum  Technologies,  National University of Singapore, Singapore}
    \thanks{Corresponding author's email address: {\tt verhagen.frits@gmail.com}.}
    \thanks{The results in this paper were presented in part at the Quantum Resources Workshop, December 2023, Singapore, and the Focused Workshop on Quantum R\'enyi Divergences, July 2024, Erd\H{o}s Center, Budapest.}
    \author{Marco Tomamichel}
	\address{Centre for Quantum  Technologies,  National University of Singapore, Singapore}
	\address{Department of Electrical and Computer Engineering, National University of Singapore, Singapore}
	\author{Erkka Haapasalo}
	\address{Centre for Quantum  Technologies,  National University of Singapore, Singapore}

\begin{abstract}
We say that a matrix $P$ with non-negative entries majorizes another such matrix $Q$ if there is a stochastic matrix $T$ such that $Q=TP$. We study matrix majorization in large samples and in the catalytic regime in the case where the columns of the matrices need not have equal support, as has been assumed in earlier works. We focus on two cases: either there are no support restrictions (except for requiring a non-empty intersection for the supports) or the final column dominates the others. Using real-algebraic methods, we identify sufficient and almost necessary conditions for majorization in large samples or when using catalytic states under these support conditions. These conditions are given in terms of multivariate divergences that generalize the R\'enyi divergences. We notice that varying support conditions dramatically affect the relevant set of divergences. Our results find an application in the theory of catalytic state transformation in quantum thermodynamics. 
\end{abstract}
    
\maketitle
	
\section{Introduction}

{\it Statistical experiments} are a mathematical description of the distributions of outcomes from an experiment on some physical system. They are described by a family $P=\{p^\tj\}_{\tj\in\Theta}$ of probability measures or vectors $p^\tj$ on the space of outcomes of the experiment, where the index set $\Theta$ represents the different states the system can be in. To avoid redundancy, it is important to characterize when one experiment is more informative than another, i.e.,\ we need to be able to compare experiments. Due to results by Blackwell \cite{Blackwell51,Blackwell53}, the information order of experiments can be characterized in an intuitive way: An experiment $P=\{p^\tj\}_{\tj\in\Theta}$ is more informative than an experiment $Q=\{q^\tj\}_{\tj\in \Theta}$, or $P$ {\it majorizes} $Q$, if and only if there is a stochastic map $T$ such that $q^\tj=Tp^\tj$ for all $\tj\in\Theta$. In this case, we write $P\rgeq Q$.

Majorization of one experiment $P=\{p^{\tj}\}_{\tj\in\Theta}$ over another experiment $Q=\{q^\tj\}_{\tj\in\Theta}$ can also be aided by having many copies of the experiments or by catalysis. The former situation means that
\begin{equation}
\left\{\big(p^\tj\big)^{\otimes n}\right\}_{\tj\in\Theta}\rgeq\left\{\big(q^\tj\big)^{\otimes n}\right\}_{\tj\in\Theta}
\end{equation}
for sufficiently large $n\in\N$. Note that $p^{\otimes n}$ denotes the $n$-th Kronecker product power of $p$. We call this {\it majorization in large samples}. The latter situation means that there exists a third experiment $R=\{r^\tj\}_{\tj\in\Theta}$ (a catalyst) such that
\begin{equation}
\big\{p^\tj\otimes r^\tj\big\}_{\tj\in\Theta}\rgeq\big\{q^\tj\otimes r^\tj\big\}_{\tj\in\Theta}.
\end{equation}
This is called {\it catalytic majorization}. Large-sample majorization implies catalytic majorization \cite{Duan2005}, but the reverse implication is false in general \cite{Feng_et_al_2006}. However, the exact difference between large-sample and catalytic majorization is still unclear. In this work, we are particularly interested in large-sample and catalytic majorization of finite experiments \cite{Heyer1982} where $\Theta=\{1,\ldots,d\}$ is a finite set and the probability measures of all experiments are finite probability vectors. Majorization in this case is also called {\it matrix majorization} \cite{dahl99}. When we identify a finite experiment $P=\big(p^{(1)},\ldots,p^{(d)}\big)$ with a matrix with columns $p^{(k)}$, majorization $P\rgeq Q$ is equivalent to the existence of a stochastic matrix $T$ such that $Q=TP$. Note that a stochastic matrix is a matrix with non-negative entries where all the columns are probability vectors.

In the case $\Theta=\{1,2\}$, majorization in large samples (under certain regularity conditions) is characterized by inequalities involving the R\'{e}nyi divergences \cite{Mu_et_al_2020}. In \cite{farooq2024}, this result was extended to the general case $\Theta=\{1,\ldots,d\}$ for any $d\in\N$ in the case of finite experiments. According to \cite{farooq2024}, matrix majorization in large samples is characterized by multivariate generalizations of the R\'{e}nyi divergences: For any $\underline{\alpha}=(\alpha_1,\ldots,\alpha_d)\in\R^d$ such that $\max_{1\leq k\leq d}\alpha_k\neq 1$, we define the quantity $D_{\underline{\alpha}}(P)$,
\begin{equation}\label{eq:multivarD}
D_{\underline{\alpha}}(P)=\frac{1}{\max_{1\leq k\leq d}\alpha_k-1}\log\sum_{i=1}^n \big(p^{(1)}_i\big)^{\alpha_1}\cdots\big(p^{(d)}_i\big)^{\alpha_d}, 
\end{equation}
where $P=\big(p^{(1)},\ldots,p^{(d)}\big)$ and $p^{(k)}=\big(p^{(k)}_1,\ldots,p^{(k)}_n\big)$ for $k=1,\ldots,d$. The relevant quantities are those $\underline{\alpha}\in\R^d$ with 
\begin{equation}\label{eq:A+}
\underline{\alpha}\in A_+:=\{(\alpha_1,\ldots,\alpha_d)\in\R^d\,|\,\alpha_1+\cdots+\alpha_d=1,\ \alpha_k\geq0,\ k=1,\ldots,d\}
\end{equation}
or
\begin{equation}\label{eq:A-}
\underline{\alpha}\in A_-:=\bigcup_{k=1}^d \{(\alpha_1,\ldots,\alpha_d)\in\R^d\,|\,\alpha_1+\cdots+\alpha_d=1,\ \alpha_k>0,\ \alpha_\ell\leq0,\ \ell\neq k\}
\end{equation}
and not equal to one of the basis vectors $e_k$, $k=1,\ldots,d$. Also, one has to consider certain pointwise limits of $D_{\underline{\alpha}}$ in the $\underline{\alpha}$-parameter space. Roughly speaking, if $P$ attains a higher value than $Q$ in all these quantities, $P$ majorizes $Q$ in large samples. For illustration of the quantities in \eqref{eq:multivarD} and their pointwise limits in the case $d=3$, see Figure \ref{fig:Restricted}.

\begin{figure}
\begin{center}
\begin{overpic}[scale=0.45,unit=1mm]{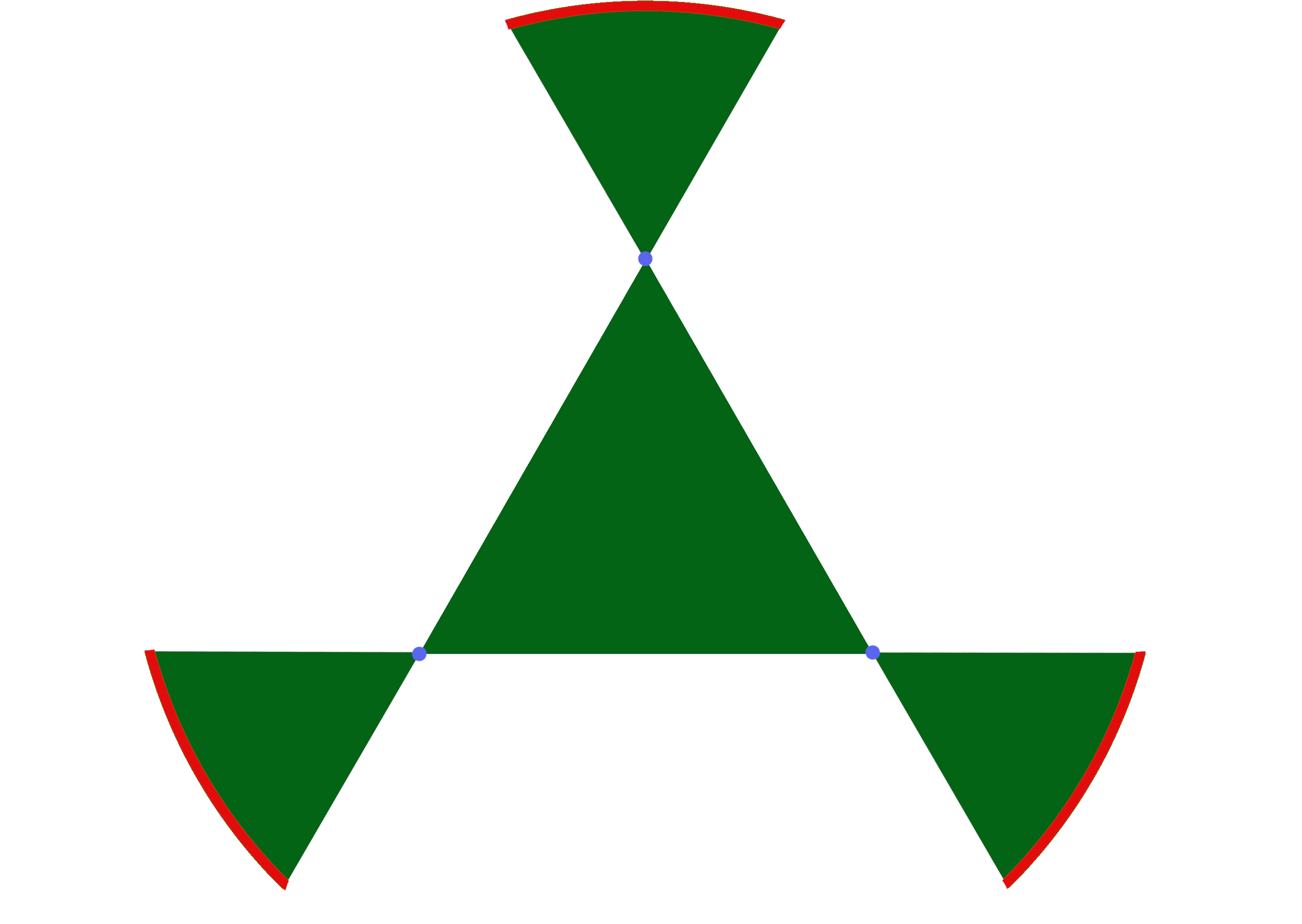}
\put(45,30){\begin{Huge}
${\color{white} A_+}$
\end{Huge}}
\put(17,12){\begin{Large}
${\color{white} A_-}$
\end{Large}}
\put(77,12){\begin{Large}
${\color{white} A_-}$
\end{Large}}
\put(47,62){\begin{Large}
${\color{white} A_-}$
\end{Large}}
\put(31,37){\begin{Large}
\darker{$D_{\underline{\alpha}}$}
\end{Large}}
\put(25,7){\begin{Large}
\darker{$D_{\underline{\alpha}}$}
\end{Large}}
\put(68,7){\begin{Large}
\darker{$D_{\underline{\alpha}}$}
\end{Large}}
\put(37,60){\begin{Large}
\darker{$D_{\underline{\alpha}}$}
\end{Large}}
\put(25,25){\begin{large}
\sininen{$\Delta^{(1)}_{\underline{\gamma}}$}
\end{large}}
\put(68,25){\begin{large}
\sininen{$\Delta^{(2)}_{\underline{\gamma}}$}
\end{large}}
\put(52,52){\begin{large}
\sininen{$\Delta^{(3)}_{\underline{\gamma}}$}
\end{large}}
\put(5,8){\begin{Large}
\punainen{$D^\T_{\underline{\beta}}$}
\end{Large}}
\put(86,8){\begin{Large}
\punainen{$D^\T_{\underline{\beta}}$}
\end{Large}}
\put(62,70){\begin{Large}
\punainen{$D^\T_{\underline{\beta}}$}
\end{Large}}
\end{overpic}
\caption{\label{fig:Restricted} The functions of \eqref{eq:multivarD} together with their pointwise limits depicted in the case $d=3$. The figure is contained in the affine plane of $\R^3$ of those points whose coordinates sum up to 1. The maps of \eqref{eq:multivarD} are associated with the areas $A_+$ and $A_-$ and are depicted in green. Note that, e.g.,\ in the lower-left part of $A_-$, we have $\alpha_1>1$ and $\alpha_2,\alpha_3\leq0$. The red arcs represent points at infinity and they correspond to the functions introduced later in \eqref{eq:DbetaTropic}. The blue points are $e_1=(1,0,0)$, $e_2=(0,1,0)$, and $e_3=(0,0,1)$. The pointwise limits of the maps in \eqref{eq:multivarD} as $\underline{\alpha}\to e_1$ are $\Delta^{(1)}_{\underline{\gamma}}(P):=\gamma_2 D_1\big(p^{(1)}\big\|p^{(2)}\big)+\gamma_3 D_1\big(p^{(1)}\big\|p^{(3)}\big)$ with $\gamma_2,\gamma_3\geq0$ which depend on the angle of approach. Similar limits exist for $\underline{\alpha}\to e_2$ or $e_3$. Here $D_1$ is the Kullback-Leibler divergence. Interestingly, as shown in \cite{Csiszar2003}, we have $\min_{p^{(1)}}\Delta^{(1)}_{\underline{\gamma}}(P)\propto D_{\frac{\gamma_2}{\gamma_2+\gamma_3}}(p^{(2)}\|p^{(3)})$. In \cite{mosonyi2024geometric}, generalizations to larger $d$ were derived and involve the quantities $D_{\underline{\alpha}}$ in \ref{eq:multivarD}. The situation depicted in this figure generalizes to the general $d\in\N$ case in a straightforward manner: the parameter area $A_+$ corresponds to the $d$-dimensional probability simplex and $A_-$ extends from the extreme points of the probability simplex as cones reaching to infinity.}
\end{center}
\end{figure}
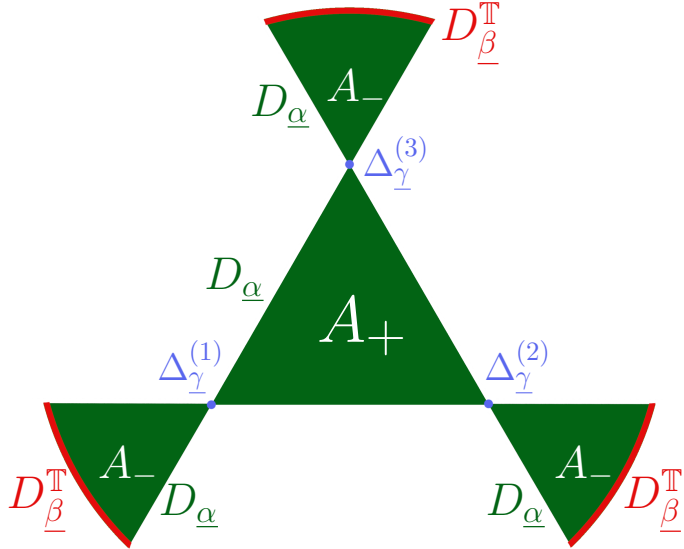

Both in \cite{Mu_et_al_2020} and \cite{farooq2024} it was assumed that, within each statistical experiment, all the probability measures or vectors share the same support. In the finite case this means that, for an experiment $P=\big(p^{(1)},\ldots,p^{(d)}\big)$ consisting of probability vectors $p^{(k)}$, we have
\begin{equation}
{\rm supp}\,p^{(1)}=\cdots={\rm supp}\,p^{(d)}, 
\end{equation}
where ${\rm supp}\,p$ is set of entries where $p$ is non-zero (the support of $p$). In this work, we consider more general support conditions. We first study the case where there are no support conditions except for the baseline
\begin{equation}\label{eq:baselinecond}
\bigcap_{k=1}^d {\rm supp}\,p^{(k)}\neq\emptyset
\end{equation}
within each experiment $P=\big(p^{(1)},\ldots,p^{(d)}\big)$. We call this the \emph{minimal restrictions} case. Note that one could consider settings with even less restrictions by dropping the condition \eqref{eq:baselinecond} and requiring instead that, for instance, the supports of each pair of probability vectors $p^{(k)}$, $p^{(k^\prime)}$ must have a non-empty intersection. However, this would restrict the divergences that characterize large-sample and catalytic majorization to just the bivariate R\'enyi divergences between each pair of probability vectors. Since we want to study their multivariate generalizations of the type \eqref{eq:multivarD}, we impose slightly stronger conditions in order for these divergences to make an appearance. We will return to this issue in Remark \ref{rmk:minimalrestrictions}. 

The minimal restrictions case is analyzed in Section \ref{sec:findings}, which culminates in the main results Theorem \ref{thm:majorization} for exact large-sample or catalytic majorization, and Theorem \ref{thm:genasymptocatalytic} for the asymptotic case where we consider approximate large-sample or catalytic majorization with vanishing error. 

In the second case, we study the case where the support of the final column dominates the others, i.e.,\ in addition to \eqref{eq:baselinecond}, we require
\begin{equation}
{\rm supp}\,p^{(k)}\subseteq{\rm supp}\,p^{(d)}
\end{equation}
for $k=1,\ldots,d-1$. We call this the \emph{dominating column} case and discuss it in Section \ref{sec:Asymmetric}. The main results here are Theorem \ref{thm:majorizationdominatingcolumn} for exact large-sample or catalytic majorization, and Theorem \ref{thm:asymptocatalytic} for asymptotic large-sample or catalytic majorization for $d=2$. 

As mentioned, the conditions for large-sample or catalytic majorization of a matrix $P$ over $Q$ are expressed in terms of inequalities of divergences evaluated on $P$ and $Q$. To be precise, a divergence $D$ is defined to be any real-valued function on the matrices we consider, such that it is additive in the sense that 
\begin{equation}
    D\big(p^{(1)}\otimes q^{(1)},\ldots,p^{(d)}\otimes q^{(d)}\big)=D\big(p^{(1)},\ldots,p^{(d)}\big)+D\big(q^{(1)},\ldots,q^{(d)}\big), 
\end{equation}
and that it satisfies the data processing inequality, i.e. for every stochastic map $T$ we have 
\begin{equation}
    D\big(T\big(p^{(1)}\big),\ldots,T\big(p^{(d)}\big)\big)\leq D\big(p^{(1)},\ldots,p^{(d)}\big). 
\end{equation}
However, our results in the rest of the paper mainly give conditions for majorization in terms of inequalities of monotone homomorphisms and derivations, which we will define precisely later on. A monotone derivation is also a divergence, and a monotone homomorphism $\Phi$ defines the divergence $c\log\circ\,\Phi$, with $c\in\R$ appropriately chosen (see Section \ref{rem:rates} for one convenient choice). Hence, we may associate a family of divergences to a family of monotone homomorphisms and derivations, and conditions for majorization can be expressed using either set of functions. 

In both the minimal restrictions and dominating column cases mentioned above, matrix majorization in large samples is characterized by inequalities involving subsets of the functions $D_{\underline{\alpha}}$. In the first case, only those $D_{\underline{\alpha}}$ where $\underline{\alpha}\in A_+$ appear in the sufficient conditions for majorization in large samples. However, on the boundary of $A_+$ some additional quantities surface. Simply put, in each smaller facet of the simplex $A_+$, a set of additional quantities have to be considered. Note that the facets of $A_+$ are of the form
\begin{align}
A_C:=\{(\alpha_1,\ldots,\alpha_d)\in A_+\,|\,\alpha_k=0,\ k\in \{1,\ldots,d\}\setminus C\}
\end{align}
for some $C\subseteq\{1,\ldots,d\}$. With $C\subsetneq\{1,\ldots,d\}$, the new quantities arising from $A_C$ are essentially the original quantities when we restrict the vectors of the experiment $P=\big(p^{(1)},\ldots,p^{(d)}\big)$ onto $\bigcap_{k\in C}{\rm supp}\,p^{(k)}$. The occurrence of new quantities can be seen as a consequence of the non-commutativity of the limits $x\downarrow0$ and $\alpha\rightarrow0$: 
\begin{equation}
\lim_{x\downarrow0}x^\alpha=0\quad\text{for fixed }\alpha>0,\quad\quad\quad\lim_{\alpha\rightarrow0}x^\alpha=1\quad\text{for fixed }x>0.
\end{equation}
To be exact, the new quantities arise when ${\rm supp}\,\underline{\alpha}\subsetneq\{1,\ldots,d\}$, in which case we consider all ${\rm supp}\,\underline{\alpha}\subseteq C\subseteq\{1,\ldots,d\}$ and the new quantities are given by essentially the same formula as \eqref{eq:multivarD} with the exception that the $i$-sum runs over only the set $\bigcap_{c\in C}{\rm supp}\,p^{(c)}$. In the case of one dominating column, the set of relevant $\underline{\alpha}$ is extended as the one-parameter families
\begin{equation}
D_{\underline{\alpha}(\lambda)},\quad \underline{\alpha}(\lambda)=(0,\ldots,0,\underbrace{1+\lambda}_{k{\rm th\ entry}},0,\ldots,0,-\lambda),\qquad\lambda\geq0,\quad k=1,\ldots,d-1,
\end{equation}
including the points at infinity where $\lambda\to\infty$ also have to be considered.

This work complements earlier work particularly in \cite{farooq2024} by filling in the missing cases that could not be dealt with in \cite{farooq2024} because of the support assumptions imposed there. The simplest complete results we obtain deals with dichotomies. By dichotomy, we mean a pair $\big(p^{(1)},p^{(2)}\big)$ of probability vectors where the second dominates the first, i.e.,\ ${\rm supp}\,p^{(1)}\subseteq{\rm supp}\,p^{(2)}$. Given dichotomies $\big(p^{(1)},p^{(2)}\big)$ and $\big(q^{(1)},q^{(2)}\big)$ with the innocuous assumption that $p^{(1)}\neq p^{(2)}$, we have that $\big(p^{(1)},p^{(2)}\big)$ majorizes $\big(q^{(1)},q^{(2)}\big)$ in large samples approximately with vanishing error if and only if
\begin{equation}
\label{eq:dichotomiesconditionsintroduction}
D_\alpha\big(p^{(1)}\big\|p^{(2)}\big)\geq D_\alpha\big(q^{(1)}\big\|q^{(2)}\big),\qquad D_\alpha\big(p^{(2)}\big\|p^{(1)}\big)\geq D_\alpha\big(q^{(2)}\big\|q^{(1)}\big)
\end{equation}
for all $\alpha\geq 1/2$ where $D_\alpha(\cdot\|\cdot)$ are the R\'{e}nyi divergences (see \eqref{eq:RenyiDivergences} for a definition of the R\'enyi divergences). The same conditions \eqref{eq:dichotomiesconditionsintroduction} apply in the case of approximate catalytic majorization with vanishing error. See Corollary \ref{cor:asymptocatalytic} for an exact formulation of this result. We have similar results for general finite statistical experiments for $d\in\N$ (dichotomies being an example of the $d=2$ case) involving the multivariate generalizations $D_{\underline{\alpha}}$ of the R\'{e}nyi divergences in addition to the new quantities arising in the case where ${\rm supp}\,\underline{\alpha}\subsetneq\{1,\ldots,d\}$ as described earlier. 

The simplest case where both the multivariate divergences and these new quantities come into play is when we consider majorization of tuples of three probability vectors $P=(p^{(1)},p^{(2)},p^{(3)})$. In the following, we will illustrate the conditions for \emph{exact} large-sample or catalytic majorization of $P$ over another tuple $Q=(q^{(1)},q^{(2)},q^{(3)})$ in the setting of one dominating column, say the last, i.e. ${\rm supp}\,p^{(1)},{\rm supp}\,p^{(2)}\subseteq{\rm supp}\,p^{(3)}$. This is the case $d=3$ in Theorem \ref{thm:majorizationdominatingcolumn}. The majorization conditions can be formulated in terms of the divergences $D_{\underline{\alpha}}$ and the R\'enyi divergences $D_\alpha$. Figure \ref{fig:DominatingColumn} gives a depiction of the parameter vectors $\underline{\alpha}$ and parameters $\alpha$ that are applicable. The precise conditions for large-sample and catalytic majorization are: 
\begin{align}
    &D_{\underline{\alpha}}(P)>D_{\underline{\alpha}}(Q)&\forall\underline{\alpha}\in{\rm int}\,A_+,\label{eq:d3condition1}\\
    &D_\alpha\big(p^{(1)}\|p^{(3)}\big)>D_\alpha\big(q^{(1)}\|q^{(3)}\big)&\forall\alpha\in[0,\infty],\label{eq:d3condition2}\\
    &D_\alpha\big(p^{(1)}|_{{\rm supp}\,p^{(2)}}\|p^{(3)}|_{{\rm supp}\,p^{(2)}}\big)>D_\alpha\big(q^{(1)}|_{{\rm supp}\,q^{(2)}}\|q^{(3)}|_{{\rm supp}\,q^{(2)}}\big)&\forall\alpha\in[0,1),\label{eq:d3condition3}\\
    &D_\alpha\big(p^{(2)}\|p^{(3)}\big)>D_\alpha\big(q^{(2)}\|q^{(3)}\big)&\forall\alpha\in[0,\infty],\label{eq:d3condition4}\\
    &D_\alpha\big(p^{(2)}|_{{\rm supp}\,p^{(1)}}\|p^{(3)}|_{{\rm supp}\,p^{(1)}}\big)>D_\alpha\big(q^{(2)}|_{{\rm supp}\,q^{(1)}}\|q^{(3)}|_{{\rm supp}\,q^{(1)}}\big)&\forall\alpha\in[0,1),\label{eq:d3condition5}\\
    &D_\alpha\big(p^{(1)}\|p^{(2)}\big)>D_\alpha\big(q^{(1)}\|q^{(2)}\big),\quad D_\alpha\big(p^{(2)}\|p^{(1)}\big)>D_\alpha\big(q^{(2)}\|q^{(1)}\big)&\forall\alpha\in[0,1/2].\label{eq:d3condition6}
\end{align}

For the first condition \eqref{eq:d3condition1} we consider $D_{\underline{\alpha}}$ for all parameter vectors $\underline{\alpha}$ in ${\rm int}\,A_+$, the interior of $A_+$ (defined in \eqref{eq:A+}), i.e. all $\underline{\alpha}=(\alpha_1,\alpha_2,\alpha_3)\in\R^3$ such that $\alpha_1+\alpha_2+\alpha_3=1$ and $0<\alpha_k<1$ for $k=1,2,3$. Condition \eqref{eq:d3condition2} is in terms of the R\'enyi divergence for $\alpha\in[0,\infty]$ between the first and third vectors. This coincides with the half-line parameterized by $(\alpha,0,1-\alpha)$ in Figure \ref{fig:DominatingColumn} starting from the top vertex of the triangle (corresponding to $\alpha=0$) and passing through the lower left vertex (corresponding to $\alpha=1$, i.e. the Kullback-Leibler divergence). In the limit $\alpha\rightarrow\infty$ this gives the max-divergence $D_\infty$. In condition \eqref{eq:d3condition3}, we denote by $r|_{{\rm supp}\,s}$ the restriction of the vector $r$ to the support of $s$, i.e. $\big(r|_{{\rm supp}\,s}\big)_i=r_i$ if $i\in{\rm supp}\,s$ and $r|_{{\rm supp}\,s}=0$ otherwise. These divergences can be seen as point-wise limits of $D_{(\alpha_1,\alpha_2,\alpha_3)}$ when $\alpha_2\rightarrow0$, similarly as the R\'enyi divergence $D_0$ is the point-wise limit of $D_\alpha$ when $\alpha\rightarrow0$. Conditions \eqref{eq:d3condition4} and \eqref{eq:d3condition5} are similar to the preceding two, and correspond to the other half-line in Figure \ref{fig:DominatingColumn}. Finally, \eqref{eq:d3condition6} corresponds to the lower line segment of the triangle parameterized by $(\alpha,1-\alpha,0)$ for $\alpha\in[0,1]$. 

\begin{figure}
\vspace{5mm}
\begin{center}
\begin{overpic}[scale=0.4,unit=1mm]{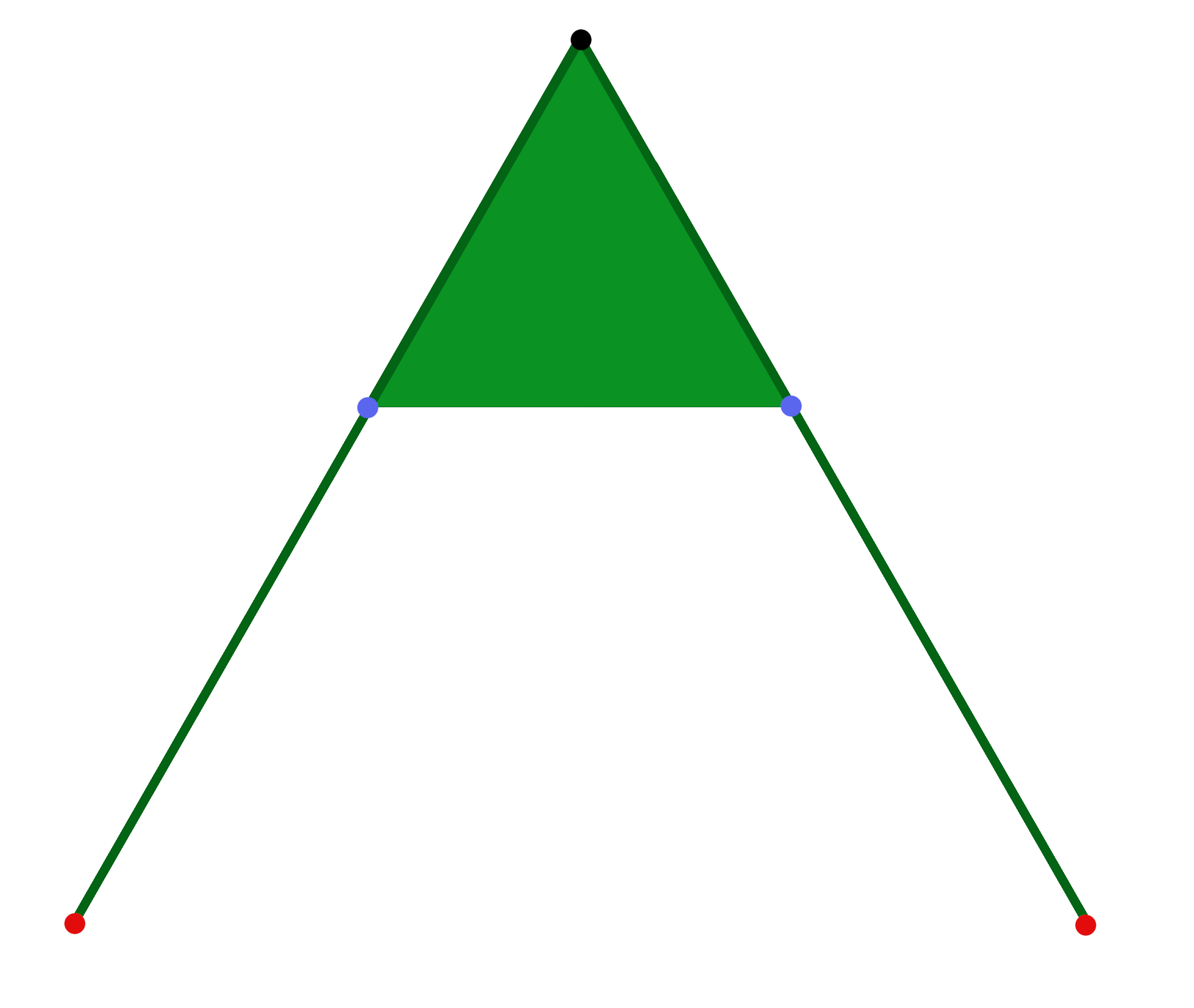}
\put(42,60){\begin{Huge}
$A_+$
\end{Huge}}
\put(-3,86){
\lighter{$D_0\big(p^{(1)}|_{{\rm supp}\,p^{(2)}}\|p^{(3)}|_{{\rm supp}\,p^{(2)}}\big),$} \darker{$D_0\big(p^{(2)}\big\|p^{(3)}\big)$, $D_0\big(p^{(1)}\big\|p^{(3)}\big)$}
}
\put(29,67){
\textbf{(a1)}
}
\put(59,67){
\textbf{(b1)}
}
\put(8,30){
\textbf{(a2)}
}
\put(80,30){
\textbf{(b2)}
}
\put(2,53){
\sininen{$D_1\big(p^{(1)}\big\|p^{(3)}\big)$}
}
\put(2,47){
\darker{$D_0\big(p^{(2)}\big\|p^{(1)}\big)$}
}
\put(70,53){
\sininen{$D_1\big(p^{(2)}\big\|p^{(3)}\big)$}
}
\put(70,47){
\darker{$D_0\big(p^{(1)}\big\|p^{(2)}\big)$}
}
\put(-12,1){
\punainen{$D_\infty\big(p^{(1)}\big\|p^{(3)}\big)$}
}
\put(78,1){
\punainen{$D_\infty\big(p^{(2)}\big\|p^{(3)}\big)$}
}
\put(0,67){
\lighter{$D_{\underline{\alpha},\{1,2,3\}}(P)$}
}
\thicklines
\put(24,68){\vector(1,-0.7){17}}
\put(35,46){
\textbf{(c1)}
}
\put(53,46){
\textbf{(c2)}
}
\end{overpic}
\caption{\label{fig:DominatingColumn} A subset of the set of parameter vectors that are depicted in Figure \ref{fig:Restricted}. These represent the family of divergences that appear in the conditions for exact large-sample and catalytic majorization for matrices $P=(p^{(1)},p^{(2)},p^{(3)})$ where the support of the last column dominates the support of the other two. The interior of the triangle $A_+$ corresponds to the trivariate divergences $D_{\underline{\alpha}}(P)$ given by \eqref{eq:multivarD}. The two half lines beginning at $e_3$, going through $e_1$, respectively $e_2$, and extending to infinity represent the R\'enyi divergences $D_\alpha\big(p^{(1)}\|p^{(3)}\big)$ (\textbf{(a1), (a2)}), respectively $D_\alpha\big(p^{(2)}\|p^{(3)}\big)$ (\textbf{(b1), (b2)}), for $\alpha\in[0,\infty]$. The side of the triangle between $e_1$ and $e_2$ is split in half at $\underline{\alpha}=(1/2,1/2,0)$: the left half \textbf{(c1)} represents $D_\alpha\big(p^{(2)}\|p^{(1)}\big)$ for $\alpha\in[0,1/2]$ and the right half \textbf{(c2)} represents $D_\alpha\big(p^{(1)}\|p^{(2}\big)$ for $\alpha\in[0,1/2]$. Additionally, on the two sides of the triangle between $e_3$ and $e_1$, respectively $e_2$, we have the R\'enyi divergences between two columns when restricted to the support of the remaining column: $D_\alpha\big(p^{(1)}|_{{\rm supp}\,p^{(2)}}\|p^{(3)}|_{{\rm supp}\,p^{(2)}}\big)$ \textbf{(a1)}, respectively $D_\alpha\big(p^{(2)}|_{{\rm supp}\,p^{(1)}}\|p^{(3)}|_{{\rm supp}\,p^{(1)}}\big)$ \textbf{(b1)}. 
}
\end{center}
\end{figure}
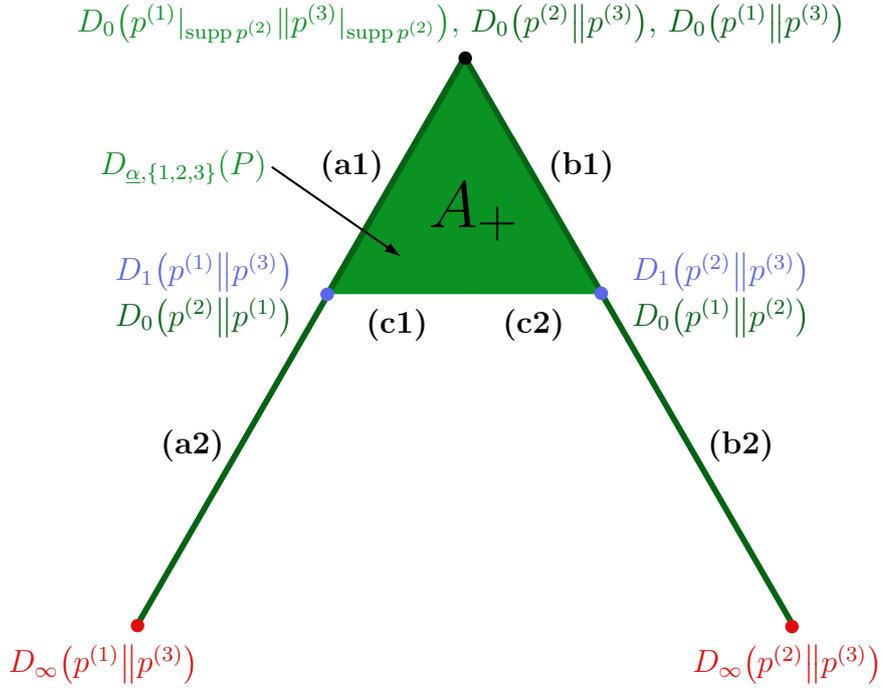

At first glance, Theorems \ref{thm:majorization} and \ref{thm:majorizationdominatingcolumn} for the minimal restrictions and dominating column case, respectively, are very similar and the former seems to apply to a larger set of matrices than the latter. On top of that, the former theorem requires less inequalities to be satisfied. Therefore, one may wonder why we would need the latter theorem at all. The important thing to note here is that the strict inequalities in the conditions that the matrices $P$ and $Q$ have to satisfy force $P$ to be of a very specific form, it has to be \emph{power universal} (more on this later). In the case of Theorem \ref{thm:majorization}, these conditions require the support of $P=(p^{(1)},\ldots,p^{(d)})$ to satisfy the following: 
\begin{equation}
    {\rm supp}\,p^{(k)}\nsubseteq{\rm supp}\,p^{(k^\prime)}\text{ for all }k,k^\prime\in\{1,\ldots,d\}\text{ with }k\neq k^\prime. 
\end{equation}
However, for Theorem \ref{thm:majorizationdominatingcolumn}, $P$ has to satisfy 
\begin{equation}
    {\rm supp}\,p^{(k)}\subsetneq{\rm supp}\,p^{(d)}\text{ for all }k\in\{1,\ldots,k-1\}, 
\end{equation}
and 
\begin{equation}
    {\rm supp}\,p^{(k)}\nsubseteq{\rm supp}\,p^{(k^\prime)}\text{ for all }k,k^\prime\in\{1,\ldots,k-1\}\text{ with }k\neq k^\prime. 
\end{equation}
Therefore, both theorems apply to two disjunct types of matrices, depending on their support structure. In principle, one could derive similar theorems for matrices with different kinds of support structures, where the family of divergences that appear in the conditions of the theorem are different for each type of support structure. We will discuss this point in more detail in the Conclusion and outlook section. To understand why the type of matrices in Theorem \ref{thm:majorization} require less inequalities to be satisfied than those in Theorem \ref{thm:majorizationdominatingcolumn}, one can consider the catalytic case: both theorems guarantee the existence of a catalyst that is taken from the same set of matrices that $P$ and $Q$ are from. Hence, when this set is larger (which it is for Theorem \ref{thm:majorization} compared to Theorem \ref{thm:majorizationdominatingcolumn}), it intuitively makes sense that we require less conditions on $P$ and $Q$. 


Theorems \ref{thm:majorization} and \ref{thm:majorizationdominatingcolumn} also give an expression for the maximal achievable rate at which a tuple of probability distributions $P=(p^{(1)},\ldots,p^{(d)})$ can be transformed into another tuple $Q=(q^{(1)},\ldots,q^{(d)})$ using the same stochastic map for each distribution. Namely, denoting by $\mathcal{D}$ set of divergences associated with the monotone homomorphisms and derivations appearing in the conditions for large-sample and catalytic majorization, the maximal achievable rate can be expressed as 
\begin{equation}
r(P\to Q)=\min_{D\in\mathcal{D}}\,\frac{D(P)}{D(Q)}. 
\end{equation}
The minimum above exists since $\mathcal{D}$ is indeed compact. The result above follows directly from our large-sample results. More details and a full derivation are given in Section \ref{rem:rates}. 

One immediate application of the results obtained in this work is in the field of quantum thermodynamics. In thermodynamical equilibrium, the system is in a stable or thermal state $\gamma_\beta$ which is typically the Gibbs state of inverse temperature $\beta$ (or a KMS-state in general). A quantum channel $\Phi$ (a completely positive trace-preserving affine map on the states) is {\it Gibbs-preserving} if $\Phi(\gamma_\beta)=\gamma_\beta$. We denote $\rho\rgeq_\beta\sigma$ for quantum states $\rho$ and $\sigma$ if there is a Gibbs-preserving channel $\Phi$ such that $\Phi(\rho)=\sigma$. When we require the states $\rho$ and $\sigma$ to be diagonalizable in the eigenbasis of the thermal state $\gamma_\beta$, $\rho\rgeq_\beta\sigma$ means that
\begin{equation}
\big(\lambda^\rho,\lambda^\beta\big)\rgeq\big(\lambda^\sigma,\lambda^\beta\big), 
\end{equation}
where $\lambda^\rho$ and $\lambda^\sigma$ are the vectors of eigenvalues of $\rho$ and $\sigma$ where the eigenvalues appear in the order fixed by the thermal eigenbasis and $\lambda^\beta:=\lambda^{\gamma_\beta}$. This thermal majorization was characterized in the asymptotic catalytic regime in \cite{Brandao2014} using the results on catalytic vector majorization derived by Klimesh in \cite{klimesh2007inequalities}. To apply the vector majorization results, particular embedding methods were used in \cite{Brandao2014} approximating the thermal spectrum $\lambda^\beta$. These approximations introduce some convergence issues which were not completely taken into account in \cite{Brandao2014}, as was also pointed out in \cite{GourTomamichel2021}. Our method bypasses these issues and finally resolves the problem of catalytic thermal majorization also in the case when the states are not of full support. We may also now investigate multiple thermal majorization $(\rho_1,\ldots\rho_m)\rgeq_\beta(\sigma_1,\ldots,\sigma_m)$ i.e.,\ the existence of a Gibbs-preserving channel $\Phi$ such that $\Phi(\rho_i)=\sigma_i$ for $i=1,\ldots,m$.

Our work not only characterizes large-sample and catalytic majorization of finite statistical experiments. In addition to the thermodynamic applications described above, our results also serve as a proof of concept for semiring methods and as a stepping stone to further applications. Similar methods have already been applied, e.g.,\ in comparison of quantum statistical experiments \cite{bunth2021asymptotic,perry2022semiring,bunth2023,bunth2024}. Moreover,  multivariate divergences similar to the ones we study find an operational interpretation in the task of (quantum) hypothesis exclusion, where they appear in the error exponents \cite{Mishra2024}. The methodology of our current work can also be applied to quantum-information-theoretic problems, including quantum majorization questions and resource theory. Just like in the classical case, (multivariate) divergences on quantum states seem to characterize large-sample quantum majorization and rates of transforming i.i.d.\ copies quantum states with limited resources. However, these applications remain a target of future research.

Our results are based on earlier work done in \cite{farooq2024} and on the theory developed in \cite{fritz2023,fritz2022}. The methods used are real algebraic in nature. We build different preordered semirings based on matrix majorization and use the Vergleichsstellens\"{a}tze developed in \cite{fritz2023,fritz2022} which associate ordering in large samples and in the catalytic setting with inequalities involving certain monotone homomorphisms of the semiring. Recall that a semiring is an algebraic structure with an addition and multiplication, but the elements of the semiring are not required to have inverses with respect to either the multiplication or even the addition. Preordered semirings additionally possess a preorder (a reflexive and transitive binary relation) which is required to fit with the algebraic structure of the semiring. 
Volker Strassen famously introduced the first subcubic algorithm for matrix multiplication, a stepping stone into his theory of asymptotic spectra \cite{strassen1986,strassen1987,strassen1988,strassen1991}. These real-algebraic methods have been applied in various fields of probability and information theory and computer science. The main results, {\it Vergleichsstellens\"{a}tze}, of \cite{fritz2023,fritz2022} applied in this work can also be seen as extensions of Strassen's {\it Positivstellensatz}. Another recent extension of Strassen's work has been derived in \cite{Vrana2022}. For an overview of these methods and their varied applications we also refer to \cite{WigZuid}. 

This paper is organized as follows. In Section \ref{sec:background}, we set up basic notation and definitions, review previous results in majorization in its various forms and introduce the theoretical results of \cite{fritz2023} that we will be utilizing in this work. We derive the preordered semiring for matrix majorization with minimal support conditions, its monotone homomorphisms and derivations relevant for large-sample and catalytic majorization, in Section \ref{sec:findings}. We partially build on earlier work in \cite{farooq2024} but see that lifting the assumption of mutually supporting columns in each matrix requires a dramatically different approach. In Section \ref{sec:Asymmetric}, we study the case where the final column supports the others and in subsections \ref{sec:2column} and \ref{sec:thermal} concentrate on the two-column case whose results we apply to the problem of thermal majorization briefly presented above.

\section{Background}\label{sec:background}
	
\subsection{Basic definitions}
	
We follow the convention $\N=\{1,2,\ldots\}$. When $d\in\N$, we define $[d]:=\{1,\ldots,d\}$. For any $n\in\N$, we define the 1-norm in $\R^n$ as
\begin{equation}
\|(x_1,\ldots,x_n)\|_1=|x_1|+\cdots+|x_n|.
\end{equation}
Instead of full $\R$, we usually concentrate on the half-line $\R_+:=[0,\infty)$ or $\R_{>0}:=(0,\infty)$.
	
For any $n\in\N$, we say that $p=(p_1,\ldots,p_n)\in\R_+^n$ is a probability vector when $p_1+\cdots+p_n=1$, i.e.,\ $\|p\|_1=1$. Let us denote the set of all $n$-component probability vectors by $\mc P_n$. This set is convex, i.e.,\ when $p,q\in\mc P_n$ and $t\in[0,1]$, then $tp+(1-t)q\in\mc P_n$. Fix $m,n\in\N$. We say that $T:\mc P_m\to\mc P_n$ is a stochastic map if it is affine, i.e.,\ $T\big(tp+(1-t)q\big)=tTp+(1-t)Tq$ for all $p,q\in\mc P_m$ and $t\in[0,1]$. Equivalently, we may view $T$ as a linear map $T:\R^m\to\R^n$ such that $T\R_+^m\subseteq \R_+^n$ and $\|Tx\|_1=\|x\|_1$ for all $x\in\R^m$. This means that we may view $T$ as an $(n\times m)$-matrix $T=(T_{i,j})_{i,j}$ with entries in $\R_+$ such that $T_{1,j}+\cdots+T_{n,j}=1$ for all $j=1,\ldots,m$. We say that $T$ is bistochastic if also $T_{i,1}+\cdots+T_{i,m}=1$ for all $i=1,\ldots,n$.
	
\subsection{Majorization}
	
Majorization is a word commonly used when an object (often a probability measure or vector or a collection of them) can be reached from another object using a map of a relevant type (often a stochastic map). The arguably simplest notion of majorization deals with finite probability vectors. This is why we call this type of majorization {\it vector majorization}. In this mode of majorization, we say that a probability vector $p=(p_1,\ldots,p_n)$ majorizes another probability vector $q=(q_1,\ldots,q_n)$ if
\begin{equation}
\sum_{i=1}^k p_i^\downarrow\geq\sum_{i=1}^k q_i^\downarrow,\qquad k=1,\ldots,n-1,
\end{equation}
where $p^\downarrow=(p_1^\downarrow,\ldots,p_n^\downarrow)$ and $q^\downarrow=(q_1^\downarrow,\ldots,q_n^\downarrow)$ are $p$ and $q$ with their entries ordered in a non-increasing way. We write $p\rgeq q$ in this situation. This is equivalent with the existence of a bistochastic map $T$ such that $Tp=q$. A stochastic map $T$ is also bistochastic if and only if for the uniform probability vector $u=(1/n,\ldots,1/n)$, we have $Tu=u$. We subsequently generalize this notion of majorization in a straightforward manner.
	
\begin{definition}\label{def:matrixmaj}
Fix $d\in\N$ and consider collections $P=\left(p^{(1)},\ldots,p^{(d)}\right)$ and $Q=\left(q^{(1)},\ldots,q^{(d)}\right)$ of probability vectors. We say that $P$ majorizes $Q$ and denote $P\rgeq Q$ if there is a stochastic map $T$ such that $Tp^{(k)}=q^{(k)}$ for $k=1,\ldots,d$. We call this form of majorization {\it matrix majorization}.
\end{definition}
	
We immediately notice that the preceding vector majorization is an example of this matrix majorization by noting that a probability vector $p=(p_1,\ldots,p_n)$ majorizes another probability vector $q=(q_1,\ldots,q_n)$ if $(p,u)\rgeq(q,u)$ where $u=(1/n,\ldots,1/n)$. Just like matrix majorization, also vector majorization can be studied in the large-sample and catalytic regimes.

Also the special case $d=2$ of matrix majorization is of particular interest. This is called {\it relative majorization}. Relative majorization is often restricted to the case of {\it dichotomies}, i.e.,\ pairs $(p,q)$ of probability measures or vectors where $q$ dominates $p$ (${\rm supp}\,p\subseteq{\rm supp}\,q$), or, symbolically, $p\ll q$.
	
\subsection{Earlier results}

The field of vector majorization is probably the best known mode of majorization. In \cite{klimesh2007inequalities}, necessary and sufficient conditions for catalytic vector majorization were derived. These conditions can be written using the {\it R\'{e}nyi divergences} $D_\alpha(\cdot\|\cdot)$, $0\leq\alpha\leq\infty$, defined on any pair $(p,q)$ of probability measures with supports ${\rm supp}\,p=:I$ and ${\rm supp}\,q=:J$,
\begin{equation}\label{eq:RenyiDivergences}
D_\alpha(p\|q)=\left\{\begin{array}{ll}
-\log{\sum_{i\in I\cap J}q_i}&\text{if}\ \alpha=0,\ {\rm and}\ I\cap J\neq\emptyset\\
\frac{1}{\alpha-1}\log{\sum_{i\in I\cap J}p_i^\alpha q_i^{1-\alpha}}&\text{if}\ \alpha\in(0,1),\ {\rm and}\ I\cap J\neq\emptyset\\
\sum_{i\in I}p_i\log{\frac{p_i}{q_i}}&\text{if}\ \alpha=1,\ {\rm and}\ I\subseteq J\\
\frac{1}{\alpha-1}\log{\sum_{i\in I}p_i^\alpha q_i^{1-\alpha}}&\text{if}\ \alpha\in(1,\infty),\ {\rm and}\ I\subseteq J\\
\max_{i\in I}\log{\frac{p_i}{q_i}}&\text{if}\ \alpha=\infty\ {\rm and}\ I\subseteq J\\
\infty&{\rm otherwise}
\end{array}\right.,
\end{equation}
in the following way: For probability vectors $p$ and $q$ such that $p\neq q$, there is a catalytic probability vector $r$ such that $p\otimes r\rgeq q\otimes r$ if and only if, denoting the uniform distribution on ${\rm supp}\,p\cup{\rm supp}\,q$ by $u$,
\begin{equation}\label{eq:KlimeshCond}
D_\alpha(p\|u)>D_\alpha(q\|u),\quad D_\alpha(u\|p)>D_\alpha(u\|q),\qquad\forall \alpha\in[1/2,\infty).
\end{equation}

Conditions for vector majorization in large samples were derived in \cite{Jensen2019} which were further extended in \cite{farooq2024}. The latter extended result can be stated as follows: given probability vectors $p$ and $q$, there is $n\in\N$ such that $p^{\otimes n}\rgeq q^{\otimes n}$ if conditions \eqref{eq:KlimeshCond} hold including the case $\alpha=\infty$.

In \cite{AubrunNechita2007}, asymptotic large-sample and catalytic vector majorization were studied. It was found that, given again probability vectors $p$ and $q$, the following conditions are equivalent: 
\begin{itemize}
\item[(i)] For all $\varepsilon>0$, there is $n_\varepsilon\in\N$ and a probability vector $q_\varepsilon$ such that $\|q-q_\varepsilon\|_1\leq\varepsilon$ and $p^{\otimes n}\rgeq q_\varepsilon^{\otimes n}$ for all $n\geq n_\varepsilon$.
\item[(ii)] For all $\varepsilon>0$, there are probability vectors $q_\varepsilon$ and $r_\varepsilon$ such that $\|q-q_\varepsilon\|_1\leq\varepsilon$ and $p\otimes r_\varepsilon\rgeq q_\varepsilon\otimes r_\varepsilon$.
\item[(iii)] $\|p\|_\alpha\geq \|q\|_\alpha$ for all $\alpha\geq1$, where
\begin{equation}
\|(p_1,\ldots,p_n)\|_\alpha:=\left(p_1^\alpha+\cdots+p_n^\alpha\right)^{1/\alpha}.
\end{equation}
\end{itemize}
Also this result was derived using real-algebraic methods in \cite{farooq2024}.

The problem of relative majorization in large samples was studied in \cite{Mu_et_al_2020}. This work studied dichotomies $(p,q)$ of probability measures over a standard Borel space $(X,\mc A)$ where $p\sim q$, i.e.,\ $p\ll q$ and $q\ll p$. In this case, the relevant R\'{e}nyi divergences $D_\alpha$, $\alpha\in[1/2,\infty]$ are given by
\begin{equation}
D_\alpha(p\|q)=\left\{\begin{array}{ll}
\frac{1}{\alpha-1}\int_X \frac{dp}{dq}(x)^\alpha\,dq(x)&\text{if}\ \alpha\in[1/2,\infty)\setminus\{1\}\\
\int_X \log{\frac{dp}{dq}(x)}\,dq(x)&\text{if}\ \alpha=1\\
q-\esssup_{x\in X}\log{\frac{dp}{dq}(x)}&\text{if}\ \alpha=\infty
\end{array}\right..
\end{equation}
It was found in \cite{Mu_et_al_2020} that, a dichotomy $\big(p^{(1)},p^{(2)}\big)$, where $p^{(1)}\sim p^{(2)}$, majorizes $\big(q^{(1)},q^{(2)}\big)$, where  and $q^{(1)}\sim q^{(2)}$, if
\begin{equation}
D_\alpha\big(p^{(1)}\big\|p^{(2)}\big)>D_\alpha\big(q^{(1)}\big\|q^{(2)}\big),\quad D_\alpha\big(p^{(2)}\big\|p^{(1)}\big)>D_\alpha\big(q^{(2)}\big\|q^{(1)}\big),\qquad\forall\alpha\in[1/2,\infty].
\end{equation}

Large-sample and catalytic matrix majorization in the case of matrices (finite experiments) $P=\big(p^{(1)},\ldots,p^{(d)}\big)$ where the supports ${\rm supp}\,p^{(k)}$, $k=1,\ldots,d$, coincide was characterized in \cite{farooq2024}. For these results, let us define $A_+$ and $A_-$ as in \eqref{eq:A+} and \eqref{eq:A-}, and
\begin{align}
&B:=\{(\beta_1,\ldots,\beta_d)\in\R^d\,|\,\beta_1+\cdots+\beta_d=0\},\label{eq:B}\\ 
&B_-:=\bigcup_{k=1}^d\{(\beta_1,\ldots,\beta_d)\in B\,|\,\beta_k>0,\ \beta_\ell\leq 0,\ \ell\neq k\}.\label{eq:B-}
\end{align}
We also define the natural basis vectors $e_k:=(\delta_{\ell,k})_{\ell=1}^d$. We define the maps $D_{\underline{\alpha}}$ and $D_{\underline{\beta}}$ through
\begin{align}
D_{\underline{\alpha}}(P)&=\frac{1}{\max_{1\leq k\leq d}\alpha_k-1}\log{\sum_{i=1}^n \prod_{k=1}^d \big(p^{(k)}_i\big)^{\alpha_k}},\label{eq:DalphaTemperate}\\
D_{\underline{\beta}}^\T(P)&=\frac{1}{\max_{1\leq k\leq d}\beta_k}\log{\max_{1\leq i\leq n}\prod_{i=1}^n \big(p^{(k)}_i\big)^{\beta_k}}\label{eq:DbetaTropic},
\end{align}
for all for $P=\big(p^{(1)},\ldots,p^{(d)}\big)$ where $p^{(k)}=\big(p^{(k)}_1,\ldots,p^{(k)}_n\big)$, $k=1,\ldots,d$. The areas $A_+$ and $A_-$ together with the functions $D_{\underline{\alpha}}$ and $D^\T_{\underline{\beta}}$ are depicted in Figure \ref{fig:Restricted}. It follows that, if
\begin{equation}
\begin{array}{rcll}
D_{\underline{\alpha}}(P)&>&D_{\underline{\alpha}}(Q),&\forall\underline{\alpha}\in A_+\cup A_-\setminus\{e_1,\ldots,e_d\},\\
D_{\underline{\beta}}^\T(P)&>&D_{\underline{\beta}}^\T(Q),&\forall\underline{\beta}\in B_-,\\
D_1\big(p^{(k)}\big\|p^{(\ell)}\big)&>&D_1\big(q^{(k)}\big\|q^{(\ell)}\big),&\forall k,\ell\in\{1,\ldots,d\},\ k\neq\ell,
\end{array}
\end{equation}
then $P=\big(p^{(1)},\ldots,p^{(d)}\big)$ majorizes $Q=\big(q^{(1)},\ldots,q^{(d)}\big)$ in large samples. We also recover the earlier result of \cite{Mu_et_al_2020} for dichotomies of probability vectors as a special case. Also asymptotic large-sample and catalytic majorization can be characterized with similar non-strict inequalities involving $D_{\underline{\alpha}}$ with $\underline{\alpha}\in A_+\cup A_-\setminus\{e_1,\ldots,e_d\}$.

The above result crucially assumes that the columns of each matrix share a common support. In this work, we remove this condition and open the possibility for different support restrictions. This requires modifying the methodology used in \cite{farooq2024}. For this, we first have to introduce the real-algebraic machinery required for the Vergleichsstellens\"{a}tze of \cite{fritz2023} that give the sufficient and almost necessary conditions for large-sample and catalytic matrix majorization.

\subsection{Preordered semirings and a Vergleichsstellensatz}
	
We only study commutative algebraic structures in this work. We say that a tuple $(S,+,\cdot,0,1,\rleq)$ is a {\it preordered semiring} if $(S,+,0)$ and $(S,1,\cdot)$ are commutative semigroups where the multiplication distributes over the addition and $\rleq$ is a preorder (a binary relation which is reflexive and transitive) satisfying
\begin{equation}
x\rleq y\ \Rightarrow\ \left\{\begin{array}{rcl}
x+a&\rleq&y+a,\\
xa&\rleq&ya,
\end{array}\right.
\end{equation}
for all $a\in S$. As above, we usually omit the multiplication dot between elements in the semiring. We denote by $\sim$ the equivalence relation generated by $\rleq$, i.e. $x\sim y$ if and only if there are $z_1,\ldots,z_n\in S$ such that
\begin{equation}
x\rleq z_1\rgeq z_2\rleq\cdots\rgeq z_n\rleq y.
\end{equation}
The preordered semiring $S$ is of {\it polynomial growth} if it has a {\it power universal} $u\in S$, i.e.
\begin{equation}
x\rleq y\quad\Rightarrow\quad\exists k\in\N:\ y\rleq xu^k.
\end{equation}
A preordered semiring $S$ is a {\it preordered semidomain} if
\begin{equation}
\begin{array}{rcl}
xy=0&\Rightarrow&x=0\ {\rm or}\ y=0,\\
0\rleq x\rleq 0&\Rightarrow&x=0.
\end{array}
\end{equation}
Moreover, $S$ is {\it zerosumfree} if $x+y=0$ implies $x=0=y$.
	
Given preordered semirings $S$ and $T$, we say that a map $\Phi:S\to T$ is a {\it monotone homomorphism} if
\begin{itemize}
\item $\Phi(x+y)=\Phi(x)+\Phi(y)$ for all $x,y\in S$ (additivity),
\item $\Phi(xy)=\Phi(x)\Phi(y)$ for all $x,y\in S$ (multiplicativity),
\item $x\rleq y$ $\Rightarrow$ $\Phi(x)\rleq\Phi(y)$ (monotonicity), and
\item $\Phi(0)=0$ and $\Phi(1)=1$.
\end{itemize}
We say that such a monotone homomorphism is {\it degenerate} if $x\rleq y$ implies $\Phi(x)=\Phi(y)$. Otherwise $\Phi$ is {\it nondegenerate}. For our results we need monotone homomorphisms with values in special semirings. These are the following:
\begin{itemize}
\item $\R_+$: The half-line $[0,\infty)$ equipped with the natural addition, multiplication, and total order.
\item $\R_+^{\rm op}$: The same as above but with the reversed order. Together $\R_+$ and $\R_+^{\rm op}$ are often called {\it temperate reals}.
\item $\T\R_+$: The half-line $[0,\infty)$ equipped with the natural multiplication, total order, and the tropical sum $x+y=\max\{x,y\}$.
\item $\T\R_+^{\rm op}$: The same as above but with the reversed order. Together $\T\R_+$ and $\T\R_+^{\rm op}$ are often called {\it tropical reals}.
\end{itemize}
Suppose that $S$ is a preordered semiring and that $\Phi:S\to\R_+$ is a monotone homomorphism. We say that an additive map $\Delta:S\to\R$ is a {\it derivation at $\Phi$} or a {\it $\Phi$-derivation} if it satisfies the {\it Leibniz rule}
\begin{equation}
\Delta(xy)=\Delta(x)\Phi(y)+\Phi(x)\Delta(y)
\end{equation}
for all $x,y\in S$. We are mainly only interested in derivations at degenerate homomorphisms that are also monotone, i.e. satisfy 
\begin{equation}
    x\rleq y\quad\Longrightarrow\quad\Delta(x)\leq\Delta(y). 
\end{equation}

Not all derivations are relevant though, because of the following results. Given a degenerate homomorphism $\Phi:S\to\R_+$, we only need to consider derivations at $\Phi$ modulo interchangeability: we say that derivations $\Delta,\Delta':S\to\R$ at $\Phi$ are {\it interchangeable} if there exists a derivation $\delta:\R_+\to\R$ at the identity function, i.e.,\ $\delta(x+y)=\delta(x)+\delta(y)$ and $\delta(xy)=\delta(x)y+x\delta(y)$ for all $x,y\in\R$, such that
\begin{equation}\label{eq:interchangeable}
\Delta'(x)=\Delta(x)+\delta\big(\Phi(x)\big)
\end{equation}
for all $x\in S$. On the other hand, assume that $\Delta:S\to\R$ is  monotone derivation at a degenerate homomorphism $\Phi:S\to\R_+$ and that $\delta:\R_+\to\R$ is a derivation. Then \eqref{eq:interchangeable} defines a monotone derivation $\Delta':S\to\R$ at $\Phi$. Additivity is shown through a quick calculation. For monotonicity, consider $x,y\in S$, $x\rgeq y$:
\begin{align}
\Delta'(y)&=\underbrace{\Delta(y)}_{\leq\Delta(x)}+\delta\big(\underbrace{\Phi(y)}_{=\Phi(x)}\big)\leq \Delta(x)+\delta\big(\Phi(x)\big)=\Delta'(x).
\end{align}
For the Leibniz property, consider $x,y\in S$:
\begin{align}
\Delta'(xy)&=\Delta(xy)+\delta\big(\underbrace{\Phi(xy)}_{=\Phi(x)\Phi(y)}\big)=\Delta(x)\Phi(y)+\Phi(x)\Delta(y)+\delta\big(\Phi(x)\big)\Phi(y)+\Phi(x)\delta\big(\Phi(y)\big)\\
&=\Big(\Delta(x)+\delta\big(\Phi(x)\big)\Big)\Phi(y)+\Phi(x)\Big(\Delta(y)+\delta\big(\Phi(y)\big)\Big)=\Delta'(x)\Phi(y)+\Phi(x)\Delta'(y).
\end{align}

Several results collectively called the `Vergleichsstellens\"{a}tze' have been derived in \cite{fritz2022}. Of all of them, we will need the following version:
	
\begin{theorem}[Theorem 8.6 in \cite{fritz2022}]\label{thm:Fritz2022}
Let $S$ be a zerosumfree preordered semidomain with a power universal element $u$. Assume that for some $d\in\N$ there is a surjective homomorphism $\|\cdot\|:S\to\R_{>0}^d\cup\{(0,\ldots,0)\}$ with trivial kernel and such that
\begin{equation}
a\succeq b\ \Rightarrow\ \|a\|=\|b\|\quad {\rm and} \quad \|a\|=\|b\|\ \Rightarrow\ a\sim b.
\end{equation}
Denote the component homomorphisms of $\|\cdot\|$ by $\|\cdot\|_{(j)}$, $j=1,\ldots,d$. Let $x,y\in S\setminus\{0\}$ with $\|x\|=\|y\|$ and where $y$ is power universal. If,
\begin{itemize}
\item[(i)] for every $\mb K\in\{\R_+,\R_+^{\rm op},\T\R_+,\T\R_+^{\rm op}\}$ and every nondegenerate monotone homomorphism $\Phi : S \to \mb K$ with trivial kernel, we have $\Phi(x) > \Phi(y)$, and
\item[(ii)] $\Delta(x) > \Delta(y)$ for every monotone $\|\cdot\|_{(j)}$-derivation $\Delta : S \to \R$ with $\Delta(u) = 1$ (modulo interchangeability) for all component indices $j = 1,\ldots,d$,
\end{itemize}
then
\begin{itemize}
\item[(a)] $x^n \succeq y^n$ for all sufficiently large $n\in\N$, and
\item[(b)] $ax\succeq ay$. 
\end{itemize}
Conversely if either of these properties holds for any $n$ or $a$, then the above inequalities hold non-strictly.
\end{theorem}

Large-sample ordering as in (a) in the above Theorem implies catalytic ordering as in (b), with a catalyst of the form $a=\sum_{\ell=0}^{n-1} x^\ell y^{n-1-\ell}$ for a sufficiently large $n\in\N$. This was shown in \cite{Duan2005} when $x,y$ are probability vectors, but the proof extends directly to the more abstract setting considered here. The converse implication typically does not hold; Theorem 3 of \cite{Feng_et_al_2006} provides a recipe for deriving counter examples in a particular setting.

\section{The minimal restrictions case}\label{sec:findings}

In order to derive conditions for large-sample majorization with varying support conditions for the columns of the matrices, we apply Theorem \ref{thm:Fritz2022}. To this end, we need to define semirings preordered by the extension of the matrix majorization ordering $\rgeq$ defined in Definition \ref{def:matrixmaj} and find their monotone homomorphisms and derivations into each of $\R_+$, $\R_+^{\rm op}$, $\T\R_+$, and $\T\R_+^{\rm op}$. We look at two particular cases: the case where the only condition for the columns of the matrices is that either they are all zero columns or they share some common support and the case where the final column dominates all the other columns. We start with the first case.
	
\subsection{Minimal restrictions semiring}

We now define the semiring to which we will apply the Vergleichsstellensatz in Theorem \ref{thm:Fritz2022} to in order to obtain large-sample majorization results for matrices whose columns need not be mutually dominating. Fix an integer $d\geq2$, which is the number of probability distributions in a tuple. For any integer $n\geq1$, write
\begin{align}
\V_n^d&:=\left\{\,\left(p^{(1)},\ldots,p^{(d)}\right)\in\left(\R_+^n\right)^d\;\middle|\;\bigcap_{k=1}^d{\rm supp}\,p^{(k)}\neq\emptyset\textrm{ or }p^{(k)}=(0,\ldots,0)\;\forall k\in[d]\,\right\},\\
\V^d&:=\bigcup_{n=1}^\infty \V_n^d. 
\end{align}
For any $P=(p^{(1)},\ldots,p^{(d)})\in \V_n^d$ we write $p^{(k)}=(p_1^{(k)},\ldots,p_n^{(k)})$ for $k\in[d]$ and view $P$ as an $n\times d$-matrix. 

We define two binary operations $\boxplus,\boxtimes$ on matrices in $\V^d$. Let $P,Q\in\V^d$. Then 
\begin{equation}
P\boxplus Q:=(p^{(1)}\oplus q^{(1)},\ldots,p^{(d)}\oplus q^{(d)}), 
\end{equation}
where $\oplus$ denotes the concatenation of two vectors. One can view this as stacking the matrices $P$ and $Q$: 
\begin{equation}
P\boxplus Q=\left(\frac{P}{Q}\right). 
\end{equation}
Also, we define 
\begin{equation}
P\boxtimes Q:=(p^{(1)}\otimes q^{(1)},\ldots,p^{(d)}\otimes q^{(d)}), 
\end{equation}
where $\otimes$ is the Kronecker product of two vectors. 

Next, we define a preorder $\rgeq$ on $\V$ as follows: for $P,Q\in\V^d$, we declare $P\rgeq Q$ if there exists a stochastic matrix $T$ such that $TP=Q$. 

The set $\V^d$ together with the operations $\boxplus$, $\boxtimes$ and the preorder $\rgeq$ do not yet form a preordered semiring. For instance, $P\boxplus Q$ and $Q\boxplus P$ have the same rows, but possibly in a different order. We now define an equivalence relation $\approx$ on $\V^d$ such that it becomes a preordered semiring together with the induced operations and preorder. Write $P\approx Q$ when one matrix can be obtained from the other by permuting the rows and adding rows of zeros $(0,\ldots,0)$. Denote by $[P]$ the equivalence class of $P\in\V^d$. The operations $\boxplus$ and $\boxtimes$ induce well-defined operations on $\V^d/\approx$ that are also commutative and transitive. The preorder $\rgeq$ on $\V^d$ induces a well-defined preorder on $\V^d/\approx$ which we still denote by $\rgeq$. Defining $0:=[(0,\ldots,0)]$, $1:=[(1,\ldots,1)]$, one checks that $S^d:=(\V^d/\approx,\boxplus,\boxtimes,0,1,\rgeq)$ is a preordered semiring, which we call the \emph{minimal restrictions matrix majorization semiring of length} $d$. 

In the rest of this paper, we will mostly denote any element $[P]\in\V^d/\approx$ by just one representative matrix $P\in\V^d$, and do calculations with this matrix while freely changing the ordering of the rows and adding or removing rows of zeros. 

Let $P=(p^{(1)},\ldots,p^{(d)}),Q=(q^{(1)},\ldots,q^{(d)})\in\V_n^d$. Since a stochastic map preserves the $1$-norm of vectors, one sees that $P\rleq Q$ implies $\|p^{(k)}\|_1=\|q^{(k)}\|_1$ for all $k\in[d]$. Next, assume $\|p^{(k)}\|_1=\|q^{(k)}\|_1$ for all $k\in[d]$, and let $T$ be the stochastic map that takes the $1$-norm of every column. Then $TP=(\|p^{(1)}\|_1,\ldots,\|p^{(d)}\|_1)=TQ$, hence $P\sim Q$. We conclude that the conditions on $\|\cdot\|:S^d\rightarrow\R_{>0}^d\cup\{0\}$ in Theorem \ref{thm:Fritz2022} are met when we choose $\|P\|:=(\|p^{(1)}\|_1,\ldots,\|p^{(d)}\|_1)$ for $P\in\V^d$, which is clearly a surjective homomorphism. Later, in subsection \ref{subsec:poweruniversalnorestr}, we will show that $S^d$ is of polynomial growth by characterizing the power universals of $S^d$.
	
\subsection{Monotone homomorphisms}
\label{subsec:monotonehomomorphisms}
	
In this section, we characterize all monotone homomorphisms $\Phi:S^d\to\mathbb{K}$, where $\mathbb{K}\in\{\R_+,\R_+^{\rm op},\T\R_+,\T\R_+^{\rm op}\}$. It will turn out that only for $\mathbb{K}=\R_+^{\rm op}$ there exist non-degenerate $\Phi$. 

Define the one-row function $\fii:\R_{>0}^d\cup\{(0,\ldots,0)\}\to\mathbb{K}$, where $\fii(x_1,\ldots,x_d):=\Phi(x_1\,\cdots\,x_d)$. Note that $\fii$ is only defined on those rows of a matrix $P\in S^d$ that contain only positive entries, and the row containing only zeros. Exactly as in the proof of Proposition 13 in \cite{farooq2024}, we can show that 
\begin{equation}
\label{eq:singlerowfunction}
\fii(x_1,\ldots,x_d)=x_1^{\alpha_1}\cdots x_d^{\alpha_d}, 
\end{equation}
where
\begin{equation}
\underline{\alpha}=(\alpha_1,\ldots,\alpha_d)\in\left\{\begin{array}{ll}
A_-&\text{if}\quad\mathbb{K}=\R_+\\
A_+&\text{if}\quad\mathbb{K}=\R_+^{\rm op}\\
B_-&\text{if}\quad\mathbb{K}=\T\R_+\\
B&\text{if}\quad\mathbb{K}=\T\R_+^{\rm op}
\end{array}\right..
\end{equation}
Here, the sets $A_+,A_-,B,B_-$ are as in \eqref{eq:A+}, \eqref{eq:A-}, \eqref{eq:B}, \eqref{eq:B-}, respectively. Note that any $P\in S^d$, $P\neq0$, has at least one row with all positive entries, since the columns must have some common support. We will say that $\Phi$ \emph{is associated with} $\underline{\alpha}$. 

In \cite{farooq2024}, where only matrices $P$ with common support were considered, it was possible to break down $P$ into a $\boxplus$-sum of its rows which can all be viewed as elements of the semiring because each row contains only positive entries, or contains only zeros. In that case, denoting the rows of a matrix $P$ by $p_i$, we may write, for any monotone homomorphism $\Phi$ with codomain $\mathbb{K}$,
\begin{equation}
\Phi(P)=\left\{\begin{array}{ll}
\sum_i \fii(p_i)&\text{if}\quad\mathbb{K}\in\{\R_+,\R_+^{\rm op}\}\\
\max_i \fii(p_i)&\text{if}\quad\mathbb{K}\in\{\T\R_+,\T\R_+^{\rm op}\}
\end{array}\right..
\end{equation}
Thus, it is enough to characterize this one-row function $\fii$ in the common-support case. In the semiring we are considering now, this no longer works because a row with both vanishing and positive entries cannot be seen as an element of our semiring; recall that the columns of any matrix in our semiring must have some common support. This is why we have to consider the value of a homomorphism at matrices with a full-support row and different types of rows with vanishing entries. However, it turns out that we may characterize the monotone homomorphisms considering its values at only two-row matrices where the first row has full support and the second has some vanishing entries. 


We first prove a useful result for the monotone homomorphisms. For any $B\subseteq[d]$ we denote by $(\;B\;)$ the row of length $d$ whose $k$-th entry is 1 if $k\in B$ and otherwise the entry is 0. For any monotone homomorphism $\Phi:S^d\to\mathbb{K}$ for $\mathbb{K}\in\{\R_+,\R_+^{\rm op},\T\R_+,\T\R_+^{\rm op}\}$ denote
\begin{equation}
\Phi[B]:=\Phi\left(\begin{array}{ccc}
1&\cdots&1\\
&B&
\end{array}\right). 
\end{equation}

\begin{lemma}
\label{lem:1or2}
Let $B\subseteq[d]$, and $\Phi:S^d\to\mathbb{K}$ a monotone homomorphism for some $\mathbb{K}\in\{\R_+,\R_+^{\rm op},\T\R_+,\T\R_+^{\rm op}\}$. Then the following holds: 
\begin{enumerate}[(i)]
\item We have 
\begin{equation}
\label{eq:1or2}
\Phi[B]=\left\{\begin{array}{ll}
1\text{ or }2&\text{if}\quad\mathbb{K}\in\{\R_+,\R_+^{\rm op}\}\\
1&\text{if}\quad\mathbb{K}\in\{\T\R_+,\T\R_+^{\rm op}\}
\end{array}\right.. 
\end{equation}
\item Suppose that $\mathbb{K}\in\{\R_+,\R_+^{\rm op}\}$ and that $\Phi$ is associated with $\underline{\alpha}=(\alpha_1,\ldots,\alpha_d)\in A_+\cup A_-$. If $\Phi[B]$ equals $2$, then $\alpha_k=0$ for all $k\notin B$. 
\end{enumerate}		
\end{lemma}
	
\begin{proof}
To prove (i), notice 
\begin{equation}
\label{eq:basicmatrixequation}
\left(\begin{array}{ccc}
1&\cdots&1\\
&B&
\end{array}\right)^{\boxtimes2}\boxplus\left(\begin{array}{ccc}
1&\cdots&1\\
1&\cdots&1
\end{array}\right)=\left(\begin{array}{ccc}
1&\cdots&1\\
&B&
\end{array}\right)\boxplus\left(\begin{array}{ccc}
1&\cdots&1\\
&B&
\end{array}\right)\boxplus\left(\begin{array}{ccc}
1&\cdots&1\\
&B&
\end{array}\right). 
\end{equation}
Using additivity and multiplicativity of $\Phi$, and the fact that we know that $\Phi$ evaluated on single rows with only positive entries is given by \eqref{eq:singlerowfunction}, the above equality, for $\mathbb{K}\in\{\R_+,\R_+^{\rm op}\}$, yields the equation 
\begin{equation}
\left(\Phi[B]\right)^2+2=3\Phi[B], 
\end{equation}
and for $\mathbb{K}\in\{\T\R_+,\T\R_+^{\rm op}\}$, yields the equation 
\begin{equation}
\max\left\{\left(\Phi[B]\right)^2,1\right\}=\Phi[B]. 
\end{equation}
The solutions to the first equation are $1$ and $2$, and to the second the unique solution is $1$, proving \eqref{eq:1or2}. 
		
To prove (ii), define the numbers $x_k$ coinciding with $2$ when $k\in B$ and otherwise $x_k=1$ and $y_k$ coinciding with $2$ when $k\in B$ and otherwise $y_k=0$. Denoting by $\substack{\rgeq\\\rleq}$ majorization in both directions simultaneously, we have
\begin{equation}
\label{eq:1or2a}
\left(\begin{array}{ccc}
1&\cdots&1\\
&B&
\end{array}\right)\boxplus\left(\begin{array}{ccc}
1&\cdots&1\\
&B&
\end{array}\right)\boxplus (x_1\,\cdots\,x_d)\,\substack{\rgeq\\\rleq}\left(\begin{array}{ccc}
x_1&\cdots&x_d\\
y_1&\cdots&y_d
\end{array}\right)\boxplus\left(\begin{array}{ccc}
1&\cdots&1\\
1&\cdots&1
\end{array}\right), 
\end{equation}
where the LHS majorizes the RHS because there is a stochastic matrix that adds the two rows $(\;B\;)$ to obtain the row $(y_1\,\cdots\,y_d)$, and the RHS majorizes the LHS because there is a stochastic matrix that splits the row $(y_1\,\cdots\,y_d)$ into two copies of the row $(\;B\;)$. Next, notice that the RHS is equal to
\begin{equation}
\label{eq:1or2b}
    (x_1\,\cdots\,x_d)\boxtimes\left(\begin{array}{ccc}
        1&\cdots&1\\
        &B&
    \end{array}\right)\boxplus\left(\begin{array}{ccc}
        1&\cdots&1\\
        1&\cdots&1
    \end{array}\right). 
\end{equation}
Using the monotonicity, additivity and multiplicativity of $\Phi$, and the single-row function given by \eqref{eq:singlerowfunction}, the two equations \eqref{eq:1or2a} and \eqref{eq:1or2b} above yield the equation 
\begin{equation}
2+2+2^{\sum_{k\in B}\alpha_k}=2^{\sum_{k\in B}\alpha_k}\cdot2+2, 
\end{equation}
hence $\sum_{k\in B}\alpha_k=1$. This implies $\alpha_k=0$ for all $k\notin B$, where for the case $\mathbb{K}=\R_+$ we use the fact that $\underline{\alpha}\in A_-$ and therefore only one of the coordinates $\alpha_k$ is positive. 
\end{proof}

Consider a monotone homomorphism $\Phi:S^d\to\mathbb{K}$ for $\mathbb{K}\in\{\R_+,\R_+^{\rm op},\T\R_+,\T\R_+^{\rm op}\}$. We call $B\subseteq[d]$ \emph{$\Phi$-trivial} if $\Phi[B]=1$ and \emph{$\Phi$-nontrivial} if $\Phi[B]=2$. Note that if $\mathbb{K}\in\{\T\R_+,\T\R_+^{\rm op}\}$, all sets $B$ are $\Phi$-trivial. Furthermore, when $\mathbb{K}\in\{\R_+,\R_+^{\rm op}\}$, the empty set is $\Phi$-trivial and $[d]$ is $\Phi$-nontrivial. 

\begin{lemma}
\label{clm:nontrivialB}
Fix a monotone homomorphism $\Phi:S^d\to\mathbb{K}$ for $\mathbb{K}\in\{\R_+,\R_+^{\rm op}\}$. Let $B_1,B_2\subseteq[d]$. The following holds: 
\begin{enumerate}[(i)]
\item If $B_1$ and $B_2$ are $\Phi$-nontrivial, then $B_1\cap B_2$ is also $\Phi$-nontrivial. 
\item If $B_1\subseteq B_2$ and $B_1$ is $\Phi$-nontrivial, then $B_2$ is also $\Phi$-nontrivial. 
\end{enumerate}
\end{lemma}

Note that this Lemma implies that when $\mathbb{K}\in\{\R_+,\R_+^{\rm op}\}$, the set of all $\Phi$-nontrivial sets is a filter of the power set of $[d]$. 

\begin{proof}
For the proof of (i), assume that $B_1$ and $B_2$ are $\Phi$-nontrivial. Notice 
\begin{equation}
\left(\begin{array}{ccc}
1&\cdots&1\\
1&\cdots&1
\end{array}\right)\boxplus\left(\begin{array}{ccc}
1&\cdots&1\\
&B_1&
\end{array}\right)\boxtimes\left(\begin{array}{ccc}
1&\cdots&1\\
&B_2&
\end{array}\right)=\left(\begin{array}{ccc}
1&\cdots&1\\
1&\cdots&1\\
1&\cdots&1\\
&B_1&\\
&B_2&\\
&B_1\cap B_2&
\end{array}\right). 
\end{equation}
This yields 
\begin{equation}
2+\Phi[B_1]\Phi[B_2]=\Phi[B_1]+\Phi[B_2]+\Phi[B_1\cap B_2]. 
\end{equation}
Since $\Phi[B_1]=2=\Phi[B_2]$, we conclude that $\Phi[B_1\cap B_2]=2$, hence $B_1\cap B_2$ is $\Phi$-nontrivial. 
	
Next, for the proof of (ii), assume that $B_1$ is $\Phi$-nontrivial. Notice 
\begin{equation}
\left(\begin{array}{ccc}
1&\cdots&1\\
1&\cdots&1
\end{array}\right)\boxplus\left(\begin{array}{ccc}
1&\cdots&1\\
&B_1&
\end{array}\right)\boxtimes\left(\begin{array}{ccc}
1&\cdots&1\\
&B_2&
\end{array}\right)=\left(\begin{array}{ccc}
1&\cdots&1\\
1&\cdots&1\\
1&\cdots&1\\
&B_1&\\
&B_2&\\
&B_1&
\end{array}\right). 
\end{equation}
This yields 
\begin{equation}
2+\Phi[B_1]\Phi[B_2]=2\Phi[B_1]+\Phi[B_2]. 
\end{equation}
Since $\Phi[B_1]=2$, we conclude that $\Phi[B_2]=2$, hence $B_2$ is $\Phi$-nontrivial.
\end{proof}

For any homomorphism $\Phi:S^d\to\mathbb{K}\in\{\R_+,\R_+^{\rm op}\}$, we define the {\it character} $C\subseteq[d]$ as the intersection of all $\Phi$-nontrivial subsets of $[d]$. $C$ is well-defined since $[d]$ itself is always $\Phi$-nontrivial. According to Lemma \ref{clm:nontrivialB}, $C$ is $\Phi$-nontrivial and, by construction, it is the smallest $\Phi$-nontrivial subset of $[d]$. Moreover, $C$ is non-empty, since the empty set is always $\Phi$-trivial. Additionally, observe that $B\subseteq[d]$ is $\Phi$-nontrivial if and only if $C\subseteq B$. 

Let $\underline{\alpha}=(\alpha_1,\ldots,\alpha_d)\in\R_+^d$ with $\alpha_1+\ldots+\alpha_d=1$, and $C=\{c_1,\ldots,c_\ell\}\subseteq[d]$, $C\neq\emptyset$, such that $C$ contains (at least) all columns $c$ with $\alpha_c\neq0$. Denote by $C^\prime:=\{c^\prime_1,\ldots,c^\prime_{\ell^\prime}\}\subseteq C$ the subset of columns $c\in C$ that satisfy $\alpha_c\neq0$. Define $\varphi_{\underline{\alpha},C}:\R_+^{\ell^\prime}\to\R$ by 
\begin{equation}
\varphi_{\underline{\alpha},C}(x_{c^\prime_1},\ldots,x_{c^\prime_{\ell^\prime}}):=x_{c^\prime_1}^{\alpha_{c^\prime_1}}\cdots x_{c^\prime_{\ell^\prime}}^{\alpha_{c^\prime_{\ell^\prime}}}. 
\end{equation}
One can show that $\varphi_{\underline{\alpha},C}$ is concave on $\R_{>0}^{\ell^\prime}$ (see the proof of Proposition 13 in \cite{farooq2024}). Then, by continuity of $\varphi_{\underline{\alpha},C}$ and the fact that all powers $\alpha_{c^\prime}>0$ for $c^\prime\in C^\prime$, using that $\lim_{x\downarrow0}x^\alpha=0$ for $\alpha>0$, one concludes that $\varphi_{\underline{\alpha},C}$ is concave on its entire domain $\R_+^{\ell^\prime}$. In the following, we will abuse notation by writing $\varphi_{\underline{\alpha},C}(x)$ for $x\in\R_+^d$ instead of $x\in\R_+^{\ell^\prime}$, where it is understood that we ignore the coordinates not in $C^\prime$. 
	
We now define the homomorphism $\Phi_{\underline{\alpha},C}:S^d\to\R_+^{\rm op}$ by 
\begin{equation}\label{eq:homomorphismsum}
\Phi_{\underline{\alpha},C}(P):=\sum_{i\in I(P,C)}\varphi_{\underline{\alpha},C}\left(p_i^{(1)},\ldots,p_i^{(d)}\right),\quad\quad P=(p^{(1)},\ldots,p^{(d)})\in \V_{n}^d, 
\end{equation}
where 
\begin{equation}\label{eq:Iset}
I(P,C):=\left\{\,i\in[n]\,\middle|\,p_i^{(c)}>0\quad\forall c\in C\,\right\}=\bigcap_{c\in C}{\rm supp}\,p^{(c)}.
\end{equation}
One can easily check that these are indeed homomorphisms and we will later see that they are also monotone. By considering the rows $I(P,C)$ of $P$ that $\Phi_{\underline{\alpha},C}$ sums over, one sees that $C$ is the smallest $\Phi_{\underline{\alpha},C}$-nontrivial subset of $[d]$, and therefore coincides with the character of the homomorphism $\Phi_{\underline{\alpha},C}$ as defined above. The degenerate ones among these homomorphisms are exactly all $\Phi_{\underline{\alpha},C}=:\Phi^{(k)}$ with $\alpha_k=1$ for some $k$ and $\alpha_{k^\prime}=0$ for all $k^\prime\neq k$, and $C=\{k\}$; indeed, these give the same value for $P$ and $Q$ whenever $P\sim Q$. Denote all the non-degenerate homomorphisms $\Phi_{\underline{\alpha},C}$ by $\mathcal{H}^{\rm n.d.}(S^d,\R_+^{\rm op})$, i.e., denoting by ${\rm supp}\,\underline{\alpha}$ for all $\underline{\alpha}=(\alpha_1,\ldots,\alpha_d)\in A_+$ (where $A_+\subset\R^d$ is given in \eqref{eq:A+}) the subset of those $k\in[d]$ such that $\alpha_k>0$,
\begin{equation}\label{eq:temperateopposite}
\mathcal{H}^{\rm n.d.}(S^d,\R_+^{\rm op})=\left\{\Phi_{\underline{\alpha},C}\,\middle|\,\underline{\alpha}\in A_+,\ {\rm supp}\,\underline{\alpha}\subseteq C\subseteq[d]\right\}\setminus\big\{\Phi^{(k)}\big\}_{k=1}^d.
\end{equation}
	
Note that if $C$ consists of all $c\in[d]$ that satisfy $\alpha_c>0$, and no other $c\in[d]$, then $\Phi_{\underline{\alpha},C}$ essentially sums over all rows of any matrix in $S^d$. Hence a homomorphism with this property is continuous in its argument. From now on, we simply call elements of $\mathcal{H}^{\rm n.d.}(S^d,\R_+^{\rm op})$ with this property continuous, i.e. 
\begin{equation}\label{eq:continuity}
\Phi_{\underline{\alpha},C}\ {\rm continuous}\quad \iff \quad C={\rm supp}\,\underline{\alpha}.
\end{equation}
A homomorphism not satisfying this property is always discontinuous. For example in the case $d=2$ and for the homomorphism $\Phi_{(0,1),\{1,2\}}$, this can be seen by considering 
\begin{equation}
P_\varepsilon:=\left(\begin{array}{ccc}
1-\varepsilon&1/2\\
\varepsilon&1/2
\end{array}\right), 
\end{equation}
where $\varepsilon\geq0$, and noting that there is a gap between $\Phi_{(0,1),\{1,2\}}(P_0)=1/2$ and $\Phi_{(0,1),\{1,2\}}(P_\varepsilon)=1$ for $\varepsilon>0$, although $\lim_{\varepsilon\rightarrow0}P_\varepsilon=P_0$. 
	
We first show that all homomorphisms in $\mathcal{H}^{\rm n.d.}(S^d,\R_+^{\rm op})$ are monotone, and then prove that these are in fact all non-degenerate monotone homomorphisms. 
	
\begin{proposition}
All homomorphisms in $\mathcal{H}^{\rm n.d.}(S^d,\R_+^{\rm op})$ are monotone. 
\end{proposition}
	
\begin{proof}
Let $\underline{\alpha}\in\R_+^d$ such that $\alpha_1+\ldots+\alpha_d=1$, and $C\in\mathcal{P}([d])\setminus\emptyset$ such that $C$ contains all columns $c$ with $\alpha_c\neq0$. Consider $P\in\V_n^d$ with rows $p_1,\ldots,p_n\in\R_+^d$ and a stochastic $m\times n$-matrix $T$, and write $Q:=TP\in\V_m^d$ with rows $q_1,\ldots,q_m\in\R_+^d$, and $T_i:=\sum_{j=1}^n T_{i,j}>0$. Using the concavity of $\varphi_{\underline{\alpha},C}$ in the first inequality, we find 
\begin{align}
\Phi_{\underline{\alpha},C}(Q)&=\sum_{i\in I(Q,C)}\varphi_{\underline{\alpha},C}(q_i)=\sum_{i\in I(Q,C)}\varphi_{\underline{\alpha},C}\left(\sum_{j=1}^n T_{i,j}p_j\right)=\sum_{i\in I(Q,C)}T_i\varphi_{\underline{\alpha},C}\left(\sum_{j=1}^n \frac{T_{i,j}}{T_i}p_j\right)\\
&\geq\sum_{i\in I(Q,C)}T_i\sum_{j=1}^n \frac{T_{i,j}}{T_i}\varphi_{\underline{\alpha},C}(p_j)=\sum_{j=1}^n \varphi_{\underline{\alpha},C}(p_j)\sum_{i\in I(Q,C)}T_{i,j}\\
&=\sum_{j\in I(P,C)} \varphi_{\underline{\alpha},C}(p_j)\underbrace{\sum_{i\in I(Q,C)}T_{i,j}}_{=1}+\sum_{j\notin I(P,C)} \varphi_{\underline{\alpha},C}(p_j)\sum_{i\in I(Q,C)}T_{i,j}\\
&\geq\sum_{j\in I(P,C)} \varphi_{\underline{\alpha},C}(p_j)=\Phi_{\underline{\alpha},C}(P). 
\end{align}
For the second inequality, consider any $j\in I(P,C)$ and $i\notin I(Q,C)$. Then there exists a column $k\in C$ such that $\sum_{j^\prime=1}^n T_{i,j^\prime}(p_{j^\prime})_k=0$, hence $T_{i,j}(p_j)_k=0$. Since $j\in I(P,C)$ implies that $(p_j)_k>0$, we conclude that $T_{i,j}=0$. This means that $\sum_{i\in I(Q,C)}T_{i,j}=\sum_{i=1}^m T_{i,j}=1$ for all $j\in I(P,C)$. 
\end{proof}
	
\begin{proposition}\label{pro:monhommin}
The non-degenerate monotone homomorphisms $S^d\rightarrow\R_+^{\rm op}$ are exactly those in $\mathcal{H}^{\rm n.d.}(S^d,\R_+^{\rm op})$. For $\mathbb{K}$ equal to $\R_+$, $\T\R_+$ or $\T\R_+^{\rm op}$, there are no non-degenerate monotone homomorphisms $S^d\rightarrow\mathbb{K}$. 
\end{proposition}
	
\begin{proof}
Let us first consider the case $\mathbb{K}\in\{\R_+,\R_+^{\rm op}\}$ and fix a monotone homomorphism $\Phi:S^d\to\mathbb{K}$. Let $\underline{\alpha}=(\alpha_1,\ldots,\alpha_d)\in A_+\cup A_-$ be the parameter vector which $\Phi$ is associated with and $C\subseteq[d]$ be the character of $\Phi$. 
For any $x\in\R_+^d$, denote by $\tilde{x}\in\R_{>0}^d$ the vector obtained from $x$ by replacing all vanishing entries with $1$. We now extend the domain $\R_{>0}^d\cup\{(0,\ldots,0)\}$ of the map $\fii$ defined at the beginning of subsection \ref{subsec:monotonehomomorphisms} to $\R_+^d$ through
\begin{equation}
\fii(x):=\Phi\big(\tilde{x}^T\boxplus x^T\big)-\fii(\tilde{x})
\end{equation}
for all $x\in\R_+^d$. Note that for $x=(x_1,\ldots,x_d)\in\R_+^d$, the vector $\tilde{x}$ has no vanishing entries, and hence 
\begin{equation}
\fii(\tilde{x})=\prod_{k\in{\rm supp}\,x}x_k^{\alpha_k}. 
\end{equation}
We see that $\fii$ coincides with its earlier definition on vectors in $\R_{>0}^d\cup\{(0,\ldots,0)\}$. Now we have, for all $x=(x_1,\ldots,x_k)\in\R_+^d$,
\begin{align}
\fii(x)&=\Phi\big(\tilde{x}^T\boxplus x^T\big)-\fii(\tilde{x})=\Phi\left(\tilde{x}^T\boxtimes\left(\begin{array}{ccc}
1&\cdots&1\\
&{\rm supp}\,x&
\end{array}\right)\right)-\fii(\tilde{x})\\
&=\fii(\tilde{x})\left(\Phi\left(\begin{array}{ccc}
1&\cdots&1\\
&{\rm supp}\,x&
\end{array}\right)-1\right)=\fii(\tilde{x})\big(\Phi[{\rm supp}\,x]-1\big)=\left\{\begin{array}{ll}
\fii(\tilde{x}),&C\subseteq {\rm supp}\,x,\\
0&{\rm otherwise}
\end{array}\right.\label{eq:Phihajotus}
\end{align}
where the last equality is due to the fact that $\Phi[B]=2$ if and only if $C\subseteq B$ and $\Phi[B]=1$ otherwise.

Let $P=\big(p^{(1)},\ldots,p^{(d)}\big)\in\mc V^d$, $p^{(k)}=\big(p^{(k)}_1,\ldots,p^{(k)}_n\big)$ for $k=1,\ldots,d$. Define $p_i:=\big(p^{(1)}_i,\ldots,p^{(d)}_i\big)$ for $i=1,\ldots,n$. We may write
\begin{align}
\Phi(P)&=\Phi\left(\bigboxplus_{i=1}^n p_i^T\right)=\Phi\left(\bigboxplus_{i=1}^n \big(\tilde{p}_i^T\boxplus p_i^T\big)\right)-\sum_{i=1}^n\fii(\tilde{p}_i)\label{eq:PhisumsCrowsa}\\
&=\sum_{i=1}^n\left(\Phi\big(\tilde{p}_i^T\boxplus p_i^T\big)-\fii(\tilde{p}_i)\right)=\sum_{i=1}^n \fii(p_i)=\sum_{i:\,C\subseteq{\rm supp}\,p_i}\prod_{k\in{\rm supp}\,p_i}\big(p^{(k)}_i\big)^{\alpha_k}.\label{eq:PhisumsCrowsb}
\end{align}
It follows from Lemma \ref{lem:1or2}(ii) that if $k\notin C$, then $\alpha_k=0$. 

Thus, the form of $\Phi$ is the same as those in $\mathcal{H}^{\rm n.d.}(S^d,\R^{\rm op}_+)$ but with $\underline{\alpha}$ not necessarily in $A_+$. We already know that those within $\mathcal{H}^{\rm n.d.}(S^d,\R_+^{\rm op})$ are monotone in the case $\mathbb{K}=\R_+^{\rm op}$. We know from \cite{farooq2024} that, if $\underline{\alpha}\in A_-\setminus\{e_1,\ldots,e_d\}$, then $\Phi$ is monotone in the case $\mathbb{K}=\R_+$ when restricted on matrices whose columns have common support. Thus, in the case $\mathbb{K}=\R_+^{\rm op}$, the nondegenerate monotone homomorphisms are exactly those in $\mathcal{H}^{\rm n.d.}(S^d,\R^{\rm op})$. We next show that there are actually no nondegenerate monotone homomorphisms in the case $\mathbb{K}=\R_+$.

Assume $\mathbb{K}=\R_+$. It follows from \eqref{eq:PhisumsCrowsa}, \eqref{eq:PhisumsCrowsb} that 
\begin{equation}\label{eq:Phisinglex}
\Phi\left(\begin{array}{ccccccc}
1&\cdots&1&1&1&\cdots&1\\
0&\cdots&0&\smash{\underbrace{x}_{k\text{-th entry}}}&0&\cdots&0
\end{array}\right) 
\vspace{0.5cm}
\end{equation}
equals either $1+x$ or $1$ for all $x\geq0$ depending on whether $C=\{k\}$, respectively $C\neq\{k\}$. Furthermore, in the former case, \eqref{eq:Phisinglex} equals $1$ for all other $k^\prime\neq k$, and since $k^\prime\notin C$, $\alpha_{k^\prime}=0$ for all $k^\prime\neq k$, and hence $\alpha_k=1$. 


Let $P=(p^{(1)},\ldots,p^{(d)})\in\mc V_n^d$. Notice the following inequalities, which follow from monotonicity, and for the second inequality by inserting a row of $1$'s for every non-zero entry in $P$ and putting every entry on its own row: 
\begin{equation}\label{eq:Phiis0}
\Phi(\|p^{(1)}\|_1\;\cdots\;\|p^{(d)}\|_1)\leq\Phi(P)\leq\sum_{i=1}^n\sum_{k=1}^d\left(\Phi\left(\begin{array}{ccccccc}
1&\cdots&1&1&1&\cdots&1\\
0&\cdots&0&\smash{\underbrace{p_i^{(k)}}_{k\text{-th entry}}}&0&\cdots&0 
\end{array}\right)-1\right). 
\vspace{0.5cm}
\end{equation}
Assume \eqref{eq:Phisinglex} equals $1+x$ for some $k$. Then it equals $1$ for all $k^\prime\neq k$, and the RHS equals $\|p^{(k)}\|_1$. Also, since in this case $\alpha_k=1$ and $\alpha_{k^\prime}=0$ for all $k^\prime\neq k$, the LHS is equal to $\|p^{(k)}\|_1$. Hence $\Phi(P)=\|p^{(k)}\|_1$ for all $P$, meaning that $\Phi$ is degenerate. Assume \eqref{eq:Phisinglex} equals $1$ for all $k$. Then it follows from the right inequality in \eqref{eq:Phiis0} that $\Phi(P)=0$ for all $P$, which is not possible. In conclusion, no non-degenerate monotone homomorphism $S^d\to\R_+$ exists. 
		
Next, we consider the cases $\mathbb{K}$ equal to $\T\R_+$ or $\T\R_+^{\rm op}$. Let $\Phi:S^d\to\mathbb{K}$, where $\mathbb{K}\in\{\T\R_+,\T\R_+^{\rm op}\}$, be a non-degenerate monotone homomorphism, and $\underline{\beta}=(\beta_1,\ldots,\beta_d)\in B$ the parameter vector which $\Phi$ is associated with.

Let $\ell\in[d]$ and define $B_\ell:=\{\ell\}$ and $x^{(\ell)}:=(0,\ldots,x,\ldots,0)$ for some $x>0$. Then it follows from Lemma \ref{lem:1or2}(i) that 
\begin{equation}
\Phi\big(\big(\widetilde{x^{(\ell)}}\big)^T\boxplus \big(x^{(\ell)}\big)^T\big)=\Phi\left(\big(\widetilde{x^{(\ell)}}\big)^T\boxtimes\left(\begin{array}{ccc}
1&\cdots&1\\
&B_\ell&
\end{array}\right)\right)=\Phi\big((\widetilde{x^{(\ell)}}\big)^T)\Phi[B_\ell]=\Phi\big((\widetilde{x^{(\ell)}}\big)^T). 
\end{equation}

Consider now the case $\mathbb{K}=\T\R_+$. Let $P\in\V_n^d$, $P\neq0$. By inserting rows of $1$'s and placing each non-zero entry in $P$ on its own row, we have 
\begin{equation}
P\boxplus\bigboxplus_{i,k\,:\,p_i^{(k)}>0}(1\,\cdots\,1\,p_i^{(k)}\,1\,\cdots\,1)\preceq\bigboxplus_{i,k\,:\,p_i^{(k)}>0}\left(\begin{array}{ccccccc}
1&\cdots&1&p_i^{(k)}&1&\cdots&1\\[0.1cm]
0&\cdots&0&\smash{\underbrace{p_i^{(k)}}_{k\text{-th entry}}}&0&\cdots&0 
\end{array}\right). 
\vspace{0.5cm}
\end{equation}
Assume all $\beta_k=0$.	Then from the majorization above and the fact that we know the value of $\Phi$ on single rows \eqref{eq:singlerowfunction} it follows that 
\begin{equation}
\Phi(P)\leq\max\{\Phi(P),1\}\leq1. 
\end{equation}
Since also $1=\Phi(\|p^{(1)}\|_1\;\cdots\;\|p^{(d)}\|_1)\leq\Phi(P)$, we find that $\Phi(P)=1$. Hence $\Phi$ is degenerate, which leads to a contradiction. Thus there exists $k$ such that $\beta_k>0$. Note that 
\begin{equation}
\left(\begin{array}{cc}
1&\frac{1}{2}\\[0.1cm]
0&\frac{1}{2}
\end{array}\right)\left(\begin{array}{ccccccc}
\frac{3}{2}&\cdots&\frac{3}{2}&1&\frac{3}{2}&\cdots&\frac{3}{2}\\[0.1cm]
0&\cdots&0&\smash{\underbrace{1}_{k\text{-th entry}}}&0&\cdots&0
\end{array}\right)=\left(\begin{array}{ccccccc}
\frac{3}{2}&\cdots&\frac{3}{2}&\frac{3}{2}&\frac{3}{2}&\cdots&\frac{3}{2}\\[0.1cm]
0&\cdots&0&\smash{\underbrace{\textstyle\frac{1}{2}}_{k\text{-th entry}}}&0&\cdots&0
\end{array}\right). 
\vspace{0.5cm}
\end{equation}
Using this and the fact that $\sum_{k^\prime\neq k}\beta_{k^\prime}<0$ it follows that
\begin{align}\label{eq:tropicalhomomorphismcontradiction}
1>\left(\frac{3}{2}\right)^{\sum_{k^\prime\neq k}\beta_{k^\prime}}&=\Phi\left(\left(\frac{3}{2}\;\cdots\;\frac{3}{2}\;\smash{\underbrace{1}_{k\text{-th entry}}}\;\frac{3}{2}\;\cdots\;\frac{3}{2}\right)\boxtimes\left(\begin{array}{ccccccc}
1&\cdots&1&1&1&\cdots&1\\[0.1cm]
0&\cdots&0&\smash{\underbrace{1}_{k\text{-th entry}}}&0&\cdots&0
\end{array}\right)\right)\\\nonumber\\
&=\Phi\left(\begin{array}{ccccccc}
\frac{3}{2}&\cdots&\frac{3}{2}&1&\frac{3}{2}&\cdots&\frac{3}{2}\\[0.1cm]
0&\cdots&0&\smash{\underbrace{1}_{k\text{-th entry}}}&0&\cdots&0
\end{array}\right)\\\nonumber\\
&\geq\Phi\left(\begin{array}{ccccccc}
\frac{3}{2}&\cdots&\frac{3}{2}&\frac{3}{2}&\frac{3}{2}&\cdots&\frac{3}{2}\\[0.1cm]
0&\cdots&0&\smash{\underbrace{\textstyle\frac{1}{2}}_{k\text{-th entry}}}&0&\cdots&0
\end{array}\right)\\\nonumber\\
&=\Phi\left(\left(\frac{3}{2}\;\cdots\;\frac{3}{2}\right)\boxtimes\left(\begin{array}{ccccccc}
1&\cdots&1&1&1&\cdots&1\\[0.1cm]
0&\cdots&0&\smash{\underbrace{\textstyle\frac{1}{3}}_{k\text{-th entry}}}&0&\cdots&0
\end{array}\right)\right)\\\nonumber\\
&\geq\Phi\left(\frac{3}{2}\;\cdots\;\frac{3}{2}\right)\Phi\left(1\;\cdots\;1\;\smash{\underbrace{\textstyle\frac{4}{3}}_{k\text{-th entry}}}\;1\;\cdots\;1\right)=\left(\frac{4}{3}\right)^{\beta_k}>1, \\\nonumber
\end{align}
which leads to a contradiction. We conclude that no non-degenerate monotone homomorphism $S^d\to\T\R_+$ exists. 

Finally, consider the case $\mathbb{K}=\T\R_+^{\rm op}$. Let $P\in\V_n^d$, $P\neq0$. $P$ has at least one row not containing any $0$. Assume WLOG it is the first row. Then 
\begin{equation}
P\preceq\left(\begin{array}{ccc}
p_1^{(1)}/2&\cdots&p_1^{(d)}/2\\
p_1^{(1)}/2&\cdots&p_1^{(d)}/2\\
p_2^{(1)}&\cdots&p_2^{(d)}\\
\vdots&\vdots&\vdots\\
p_n^{(1)}&\cdots&p_n^{(d)}
\end{array}\right), 
\end{equation}
hence
\begin{equation}
\Phi(P)\geq\max\left\{\Phi\left(p_1^{(1)}/2\;\cdots\;p_1^{(d)}/2\right),\Phi\left(\begin{array}{ccc}
p_1^{(1)}/2&\cdots&p_1^{(d)}/2\\
p_2^{(1)}&\cdots&p_2^{(d)}\\
\vdots&\vdots&\vdots\\
p_n^{(1)}&\cdots&p_n^{(d)}
\end{array}\right)\right\}. 
\end{equation}	
Assume all $\beta_k=0$. Then it follows from the above and \eqref{eq:singlerowfunction} that $\Phi(P)\geq1$. Since also $1=\Phi(\|p^{(1)}\|_1\;\cdots\;\|p^{(d)}\|_1)\geq\Phi(P)$, we find that $\Phi(P)=1$. Hence $\Phi$ is degenerate, which leads to a contradiction. Thus there exists $k$ such that $\beta_k<0$. Then we can perform the same calculations as in \eqref{eq:tropicalhomomorphismcontradiction} but now with $\beta_k<0$ and all the inequalities reversed. This leads to a contradiction. Hence there exists no non-degenerate monotone homomorphism $S^d\to\T\R_+^{\rm op}$.
\end{proof}
	
\subsection{Monotone derivations}
	
In this section, we identify the monotone derivations of the degenerate monotone homomorphisms on $S^d$, up to \emph{interchangeability}: if the difference between two $\Phi$-derivations factors through $S^d/\sim$, then they are called \emph{interchangeable} and we only need to consider one of these derivations in the Vergleichsstellensatz (see Definition 8.1 and the subsequent discussion in \cite{fritz2022}). 

Recall that the degenerate monotone homomorphisms are exactly the $d$ $1$-norms of the columns $P\mapsto \|p^{(k)}\|_1$ for every $k\in[d]$. It will turn out that up to interchangeability, the only monotone derivation $\Delta$ is $\Delta(P)=0$ for all $P\in S^d$. 

\begin{proposition}
\label{pro:derivations}
There is a unique monotone derivation $\Delta$, up to interchangeability, for any degenerate monotone homomorphism $S^d\rightarrow\R_+$, and satisfies $\Delta(P)=0$ for all $P\in S^d$. 
\end{proposition}
 
Before proving this Proposition, we will use interchangeability to show that we only need to consider a certain type of derivations, namely those derivations $\Delta$ satisfying the vanishing property $\Delta(x,\ldots,x)=0$ for all $x\geq0$, as stated by the Lemma below. For the derivations of the semiring of matrices with common support studied in \cite{farooq2024}, the vanishing property was only shown for $x\in\Q$, $x\geq0$. However, one could show that it does in fact hold for all $x\geq0$ analogously to the proof of the Lemma below. 

\begin{lemma}
\label{lem:interchangeability}
Up to interchangeability, we may assume that any monotone derivation $\Delta$ of a degenerate monotone homomorphism satisfies 
\begin{equation}
\Delta(x\,\cdots\,x)=0 
\end{equation}
for any $x\geq0$. 
\end{lemma}
	
\begin{proof}
Let $\Delta$ be a monotone derivation of the degenerate monotone homomorphism $\Phi^{(k)}:P\mapsto \|p^{(k)}\|_1$ for some $k\in[d]$. Define the function $\delta:\R_+\to\R$ through $\delta(x)=-\Delta(x\,\cdots\,x)$. This is a derivation at the identity function. Let us prove this. For the additivity and the Leibniz property, consider $x,y\in\R_+$. Note that
\begin{equation}
\left(\begin{array}{ccc}
x&\cdots&x\\
y&\cdots&y
\end{array}\right)\rgeq \big(x+y\,\cdots\,x+y\big)\rgeq \left(\begin{array}{ccc}
x&\cdots&x\\
y&\cdots&y
\end{array}\right).
\end{equation}
Thus,
\begin{align}
\delta(x+y)&=-\Delta\left(\begin{array}{ccc}
x&\cdots&x\\
y&\cdots&y
\end{array}
\right)=-\Delta(x\,\cdots\,x)-\Delta(y\,\cdots\,y)=\delta(x)+\delta(y).
\end{align}
Moreover,
\begin{align}
\delta(xy)&=-\Delta\big((x\,\cdots\,x)\boxtimes(y\,\cdots\,y)\big)=-\Delta(x\,\cdots\,x)y-x\Delta(y\,\cdots\,y)=\delta(x)y+x\delta(y).
\end{align}

Let us define the monotone derivation $\Delta'$ at $\Phi^{(k)}$ as in \ref{eq:interchangeable}, i.e.,
\begin{equation}
\Delta'(P)=\Delta(P)+\delta\big(\|p^{(k)}\|_1\big).
\end{equation}
Since $\Delta$ and $\Delta'$ are interchangeable, we may replace $\Delta$ with $\Delta'$. We now have, for all $x\geq0$,
\begin{equation}
\Delta'(x\,\cdots\,x)=\Delta(x\,\cdots\,x)+\delta(x)=\Delta(x\,\cdots\,x)-\Delta(x\,\cdots\,x)=0.
\end{equation}
\end{proof}

\begin{remark}
According to Lemma \ref{lem:interchangeability}, up to interchangeability, any monotone derivation $\Delta:S^d\to\R$ has the vanishing property $\Delta(p,\ldots,p)=0$ for any $p\in\R_+^n$, $n\in\N$. Indeed, denoting $p=(p_1,\ldots,p_n)$, we have
\begin{equation}
\Delta(p,\ldots,p)=\sum_{i=1}^n \Delta(p_i\,\cdots\,p_i)=0.
\end{equation}
This same result also applies to the results of \cite{farooq2024} where the assumption that $\|p\|_1\in\Q_+$ was made to guarantee this vanishing property.
\end{remark}

Recall that the degenerate monotone homomorphisms are all $\Phi^{(k)}$ for $k\in[d]$. Let $\Delta:S^d\to\mathbb{\R}$ be a monotone $\Phi^{(k)}$-derivation for some $k\in[d]$. For any $B\subseteq[d]$ we write 
\begin{equation}
    \Delta[B]:=\Delta\left(\begin{array}{ccc}
            1&\cdots&1\\
            &B&
        \end{array}\right). 
\end{equation}
The following Lemma will be used in the proof of Proposition \ref{pro:derivations}. 
\begin{lemma}
    \label{lem:derivationszero}
    Let $B\subseteq[d]$ and $\Delta:S^d\to\mathbb{\R}$ be a monotone $\Phi^{(k)}$-derivation for some $k\in[d]$. Then $\Delta[B]=0$. 
\end{lemma}

\begin{proof}
Recall the equality of matrices given in \eqref{eq:basicmatrixequation}. Using the Leibniz rule and additivity of $\Delta$, and Lemma \ref{lem:interchangeability}, we find 
    \begin{equation}
        2\Phi^{(k)}[B]\Delta[B]=3\Delta[B], 
    \end{equation}
    where $\Phi^{(k)}[B]=2$, respectively $1$, depending on whether $k\in B$, respectively $k\notin B$. In both cases, the only solution is $\Delta[B]=0$. 
\end{proof}

We now prove Proposition \ref{pro:derivations}. 
	
\begin{proof}[Proof of Proposition \ref{pro:derivations}]
Fix $k\in[d]$ and assume $\Delta:S^d\to\R$ is a monotone $\Phi^{(k)}$-derivation. For $\ell=1,\ldots,d$ define
\begin{equation}
    C_\ell:=\left(\begin{array}{ccccccc}
        1&\cdots&1&1&1&\cdots&1\\[0.1cm]
        0&\cdots&0&\smash{\underbrace{1}_{\ell\text{-th entry}}}&0&\cdots&0
    \end{array}\right). 
    \vspace{0.5cm}
\end{equation}
According to Lemma \ref{lem:derivationszero}, $\Delta(C_\ell)=0$ for $\ell=1,\ldots,d$. Let us consider $P=\big(p^{(1)},\ldots,p^{(d)}\big)\in\mc V^d$ with $p^{(\ell)}=\big(p^{(\ell)}_1,\ldots,p^{(\ell)}_n\big)$. Define the matrix
\begin{equation}
P':=\bigboxplus_{i=1}^n\bigboxplus_{\ell=1}^d \big(p^{(\ell)}_i\,\cdots\,p^{(\ell)}_i\big)\boxtimes C_\ell.
\end{equation}
One easily sees that
\begin{equation}
P'\rgeq P'':=P\boxplus \bigboxplus_{i=1}^n\bigboxplus_{\ell=1}^d \big(p^{(\ell)}_i\,\cdots\,p^{(\ell)}_i\big).
\end{equation}
Since $\Delta(x\,\cdots\,x)=0$ for all $x\geq0$, we notice that $\Delta(P'')=\Delta(P)$ and, hence
\begin{align}
\Delta(P)&=\Delta(P'')\leq\Delta(P')=\sum_{i=1}^n\sum_{\ell=1}^d\Delta\left(\big(p^{(\ell)}_i\,\cdots\,p^{(\ell)}_i\big)\boxtimes C_\ell\right)\label{eq:Deltais0a}\\
&=\sum_{i=1}^n\sum_{\ell=1}^d\Big(\underbrace{\Delta\big(p^{(\ell)}_i\,\cdots\,p^{(\ell)}_i\big)}_{=0}\Phi^{(k)}(C_\ell)+\Phi^{(k)}\big(p^{(\ell)}_i\,\cdots\,p^{(\ell)}_i\big)\underbrace{\Delta(C_\ell)}_{=0}\Big)=0.\label{eq:Deltais0b}
\end{align}
Thus, $\Delta$ is nonpositive.
  
We finish the proof by showing that $\Delta$ is also nonnegative. Define $M:=\max_{1\leq\ell\leq d}\|p^{(\ell)}\|_1$, $m_\ell:=M-\|p^{(\ell)}\|_1$ for $\ell=1,\ldots,d$, and $m:=m_1+\cdots+m_d$. Consider
\begin{equation}
\tilde{P}:=P\boxplus\bigboxplus_{\ell=1}^d\big(m_\ell\,\cdots\,m_\ell\big)\boxtimes C_\ell, 
\end{equation}
which has columns that have the same $1$-norm $M+m$ as one easily verifies. By monotonicity, we find 
\begin{align}
0&=\Delta(M+m\,\cdots\,M+m)\leq \Delta(\tilde{P})=\Delta(P)+\sum_{\ell=1}^d\underbrace{\Delta\Big(\big(m_\ell\,\cdots\,m_\ell\big)\boxtimes C_\ell\Big)}_{=0}=\Delta(P),
\end{align}
where the terms in the sum vanish according to a calculation analogous to \eqref{eq:Deltais0a}, \eqref{eq:Deltais0b}.
\end{proof}

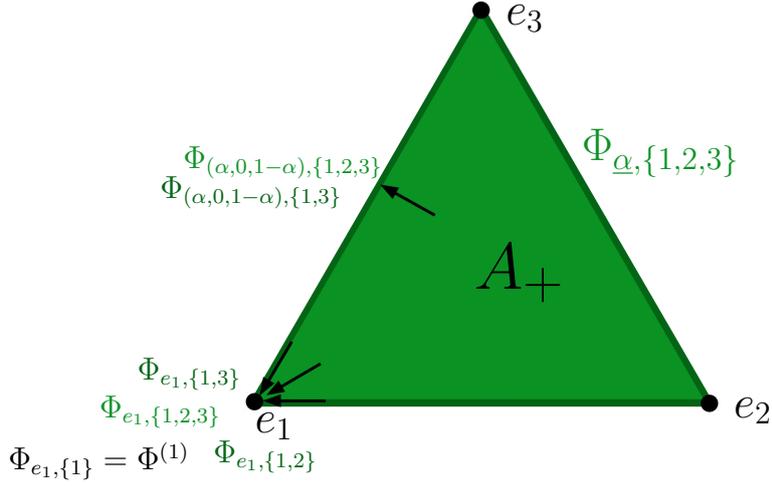
\begin{figure}
\begin{center}
\begin{overpic}[scale=0.43,unit=1mm]{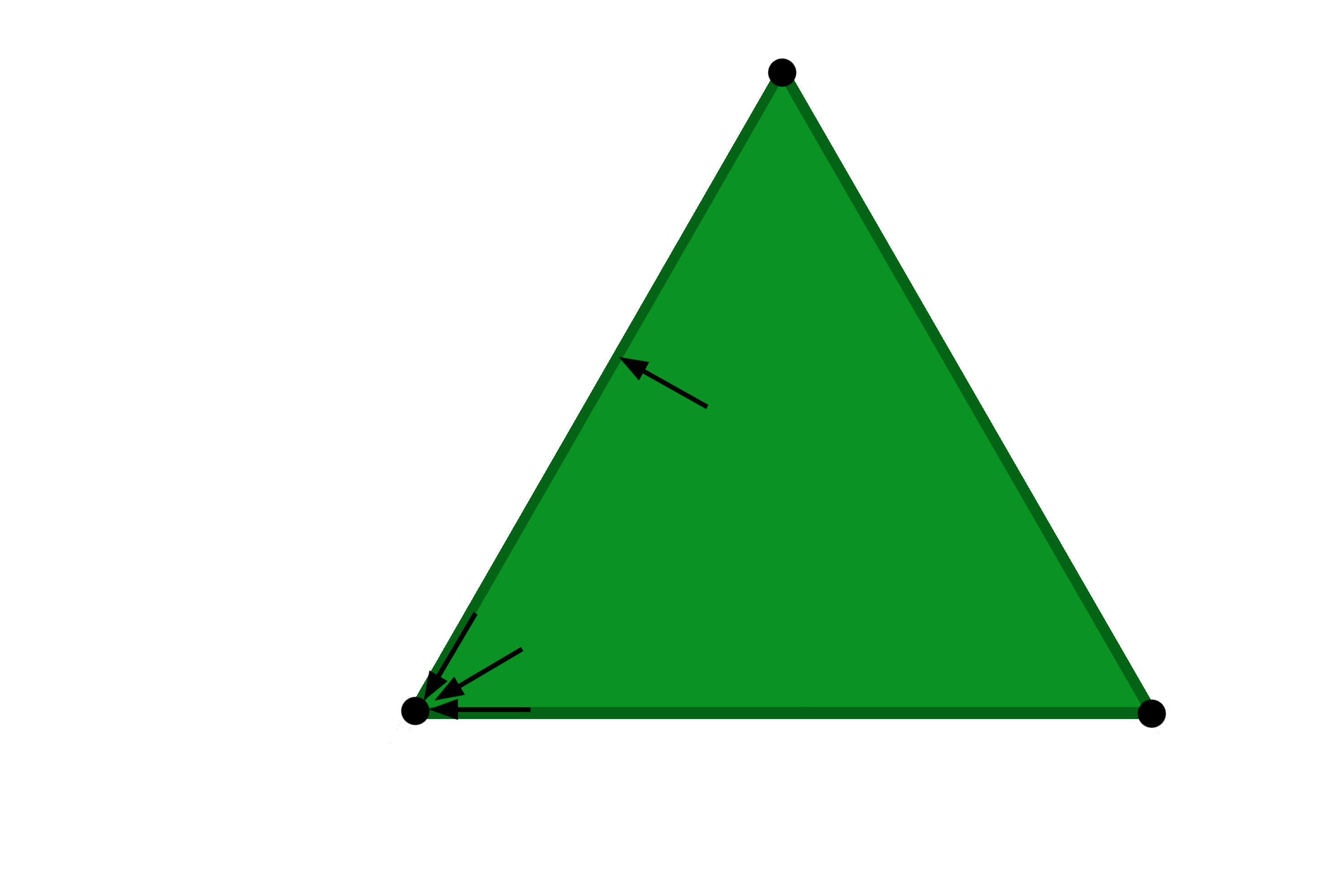}
\put(60,30){\begin{Huge}
$A_+$
\end{Huge}}
\put(32,11){\begin{Large}
$e_1$
\end{Large}}
\put(65,65){\begin{Large}
$e_3$
\end{Large}}
\put(95,13){\begin{Large}
$e_2$
\end{Large}}
\put(75,47){\begin{Large}
${\color{lighter} \Phi_{\underline{\alpha},\{1,2,3\}}}$
\end{Large}}
\put(23,46){
${\color{lighter} \Phi_{(\alpha,0,1-\alpha),\{1,2,3\}}}$
}
\put(12,13){
${\color{lighter} \Phi_{e_1,\{1,2,3\}}}$
}
\put(20,42){
${\color{darker} \Phi_{(\alpha,0,1-\alpha),\{1,3\}}}$
}
\put(17,18){
${\color{darker} \Phi_{e_1,\{1,3\}}}$
}
\put(27,7){
${\color{darker} \Phi_{e_1,\{1,2\}}}$
}
\put(0,6){
$\Phi_{e_1,\{1\}}=\Phi^{(1)}$
}
\end{overpic}
\caption{\label{fig:NoRestrictions} Illustration of the monotone homomorphisms of $S^3$, where the figure concentrates on the line segment from $e_1$ to $e_3$ and the vertex $e_1$: The monotone homomorphisms associated with $\underline{\alpha}$ in the interior of $A_+$ are $\Phi_{\underline{\alpha},\{1,2,3\}}$, depicted in lighter green. The smaller facets of $A_+$, i.e.,\ the three line segments and the singletons $\{e_1\}$, $\{e_2\}$, and $\{e_3\}$ inherit similar homomorphisms still with $C=\{1,2,3\}$, also depicted in lighter green. These homomorphisms can be seen as pointwise limits from the interior of $A_+$, when $\alpha_2\to0$. Additionally, the extreme point $e_1$ inherits the homomorphisms $\Phi_{e_1,\{1,3\}}$ and $\Phi_{e_1,\{1,2\}}$, depicted in darker green, as limits from the two line segments it connects, and similarly for $e_2$ and $e_3$. We also have the homomorphisms $\Phi_{(\alpha,0,1-\alpha),\{1,3\}}$, $\alpha\in[0,1]$, on the line segment between $e_1$ and $e_3$, depicted in darker green, and similarly for the other two line segments. At $e_1$, the degenerate homomorphism with $C=\{1\}$, i.e.,\ $\Phi^{(1)}$, is also present, depicted in black, and similarly for $e_2$ and $e_3$. There are no non-zero derivations at the three extreme points.}
\end{center}
\end{figure}

	\subsection{Power universal elements}\label{subsec:poweruniversalnorestr}
	
	We determine which matrices in $\V^d$ with columns that sum to $1$ are power universal. Recall that for large-sample or catalytic majorization of a matrix $P$ over $Q$, it is a requirement in the Vergleichsstellensatz (Theorem \ref{thm:Fritz2022}) that $P$ is power universal. 
	\begin{proposition}\label{lem:poweruniversalconditions}
		Let $U=(u^{(1)},\ldots,u^{(d)})\in \V^d$ with $u^{(k)}$ for $k\in[d]$ probability distributions. Then the following are equivalent: 
		\begin{enumerate}[(i)]
			\item $U$ is power universal. 
			\item There exists $m\in\N$ such that for all $k\in[d]$, $U^{\boxtimes m}$ has a row $(0\,\cdots\,0\underbrace{x_k}_{k\text{-th entry}}0\,\cdots\,0)$ for some $x_k>0$. 
			\item It holds that ${\rm supp}\,u^{(k)}\nsubseteq{\rm supp}\,u^{(k^\prime)}$ for all $k,k^\prime\in[d]$, $k\neq k^\prime$. 
			\item Any nondegenerate homomorphism $\Phi\in\mathcal{H}^{\rm n.d.}(S^d,\R_+^{\rm op})$ satisfies $\Phi(U)<1$. 
		\end{enumerate}
	\end{proposition}
	
	\begin{proof}
		We show the implications (iii) $\Rightarrow$ (ii) $\Rightarrow$ (i) $\Rightarrow$ (iv) $\Rightarrow$ (iii). 
		
		(iii) $\Rightarrow$ (ii): Let $k\in[d]$. The matrix $U^{\boxtimes(d-1)}$ contains a row which is the element-wise product of the $d-1$ rows $(u^{(1)}_{i_{k^\prime}}\;\cdots\;u^{(d)}_{i_{k^\prime}})$ with $u^{(k)}_{i_{k^\prime}}>0,u^{(k^\prime)}_{i_{k^\prime}}=0$ for $k^\prime\in[d]\setminus\{k\}$, which exist according to (iii). This row in $U^{\boxtimes(d-1)}$ is clearly of the form $(0\,\cdots\,0\underbrace{x_k}_{k\text{-th entry}}0\,\cdots\,0)$ for some $x_k>0$. We conclude that (ii) holds for $m=d-1$. 
		
		(ii) $\Rightarrow$ (i): Assume $U\in \V_n^d$ satisfies (ii), and let $P,Q\in \V_n^d$, $Q\preceq P$. We may assume WLOG that $P,Q$ have columns of unit $1$-norm. By (ii) there exists $m_0\in\N$ such that $\tilde{U}:=U^{\boxtimes m_0}$ has a row $(0\,\cdots\,0\underbrace{x_k}_{k\text{-th entry}}0\,\cdots\,0)$ with $x_k>0$ for every $k\in[d]$. We then have 
		\begin{equation}
			\tilde{U}=\left(\begin{array}{ccc}
				\tilde{u}^{(1)}_0&\cdots&\tilde{u}^{(d)}_0\\
				1-\|\tilde{u}^{(1)}_0\|_1&\cdots&0\\
				0&\ddots&0\\
				0&\cdots&1-\|\tilde{u}^{(d)}_0\|_1
			\end{array}\right)\succeq\left(\begin{array}{ccc}
				\|\tilde{u}^{(1)}_0\|_1&\cdots&\|\tilde{u}^{(d)}_0\|_1\\
				1-\|\tilde{u}^{(1)}_0\|_1&\cdots&0\\
				0&\ddots&0\\
				0&\cdots&1-\|\tilde{u}^{(d)}_0\|_1
			\end{array}\right), 
		\end{equation}
		for some vectors $\tilde{u}_0^{(k)}\in\R_+^l$, $k\in[d]$, where $l\in\N$, and such that $1-\|\tilde{u}_0^{(k)}\|_1>0$ for all $k\in[d]$. Then for any $m\geq1$ 
		\begin{align}
			\tilde{U}^{\boxtimes m}&\succeq\left(\begin{array}{ccc}
				\|\tilde{u}^{(1)}_0\|_1&\cdots&\|\tilde{u}^{(d)}_0\|_1\\
				1-\|\tilde{u}^{(1)}_0\|_1&\cdots&0\\
				0&\ddots&0\\
				0&\cdots&1-\|\tilde{u}^{(d)}_0\|_1
			\end{array}\right)^{\boxtimes m}\\
			&\succeq\left(\begin{array}{ccc}
				\|\tilde{u}^{(1)}_0\|_1^m&\cdots&\|\tilde{u}^{(d)}_0\|_1^m\\
				1-\|\tilde{u}^{(1)}_0\|_1^m&\cdots&0\\
				0&\ddots&0\\
				0&\cdots&1-\|\tilde{u}^{(d)}_0\|_1^m
			\end{array}\right)=:\tilde{U}_m. 
		\end{align}
		Denote the stochastic matrix that maps $\tilde{U}^{\boxtimes m}$ to $\tilde{U}_m$ as $T_m$. Write
		\begin{equation}
			P=\left(\begin{array}{ccc}
				p_0^{(1)}&\cdots&p_0^{(d)}\\
				p_1^{(1)}&\cdots&p_1^{(d)}
			\end{array}\right), 
		\end{equation}
		where for $k\in[d]$ the vectors $p_0^{(k)}\in\R_{>0}^{n_0}$ form the part of common support of $P$ and $p_1^{(k)}\in\R_+^{n_1}$, and $n_0+n_1=n$. Also, define $v_0:=(1/n_0,\ldots,1/n_0)\in\R_+^{n_0}$. Next, consider 
		\begin{equation}
			S_m:=\left(\begin{array}{cccc}
				v_0&\frac{p_0^{(1)}}{1-\|\tilde{u}_0^{(1)}\|_1^m}-\frac{\|\tilde{u}_0^{(1)}\|_1^m}{1-\|\tilde{u}_0^{(1)}\|_1^m}v_0&\cdots&\frac{p_0^{(d)}}{1-\|\tilde{u}_0^{(d)}\|_1^m}-\frac{\|\tilde{u}_0^{(d)}\|_1^m}{1-\|\tilde{u}_0^{(d)}\|_1^m}v_0\\
				0&\frac{1-\|p_0^{(1)}\|_1}{1-\|\tilde{u}_0^{(1)}\|_1^m}&\cdots&0\\
				\vdots&0&\ddots&0\\
				0&0&\cdots&\frac{1-\|p_0^{(d)}\|_1}{1-\|\tilde{u}_0^{(d)}\|_1^m}
			\end{array}\right), 
		\end{equation}
		which has columns that sum to $1$ and non-negative entries for sufficiently large $m$ because $\|\tilde{u}_0^{(k)}\|_1<1$ and $p_0^{(k)}$ has positive entries for all $k\in[d]$. Then 
		\begin{align}
			P&\preceq\left(\begin{array}{ccc}
				p_0^{(1)}&\cdots&p_0^{(d)}\\
				1-\|p_0^{(1)}\|_1&\cdots&0\\
				0&\ddots&0\\
				0&\cdots&1-\|p_0^{(d)}\|_1
			\end{array}\right)=S_m\tilde{U}_m=S_mT_m\tilde{U}^{\boxtimes m}\\
			&\preceq \tilde{U}^{\boxtimes m}=(1\,\cdots\,1)\boxtimes U^{\boxtimes(m_0m)}\preceq Q\boxtimes U^{\boxtimes(m_0m)}, 
		\end{align}
		which is what we aimed to prove. 
		
		(i) $\Rightarrow$ (iv): Let $\Phi_{\underline{\alpha},C}\in\mathcal{H}^{\rm n.d.}(S^d,\R_+^{\rm op})$, and $P,Q\in \V^d$ such that $Q\preceq P$. Since $U$ is power universal, there exists $m\in\N$ such that $P\preceq Q\boxtimes U^{\boxtimes m}$. Using monotonicity and multiplicativity: 
		\begin{equation}
			\Phi_{\underline{\alpha},C}(P)\geq \Phi_{\underline{\alpha},C}(Q\boxtimes U^{\boxtimes m})=\Phi_{\underline{\alpha},C}(Q)\Phi_{\underline{\alpha},C}(U)^m\geq \Phi_{\underline{\alpha},C}(P)\Phi_{\underline{\alpha},C}(U)^m. 
		\end{equation}
		If $\Phi_{\underline{\alpha},C}(U)=1$, then $\Phi_{\underline{\alpha},C}(P)=\Phi_{\underline{\alpha},C}(Q)$. This implies that $\Phi_{\underline{\alpha},C}$ is degenerate, leading to a contradiction. Therefore, $\Phi_{\underline{\alpha},C}(U)<1$. 
		
		(iv) $\Rightarrow$ (iii): Let $k,k^\prime\in[d]$, $k\neq k^\prime$. Since $\Phi_{e_k,\{k,k^\prime\}}(U)<1$ and the columns in $U$ sum to $1$, $I(U,\{k,k^\prime\})$ cannot contain all rows $i$ in $U$ where $u^{(k)}_i>0$. Hence there exists a row $i$ such that $u^{(k)}_i>0$, but $u^{(k^\prime)}_i=0$, proving the claim. 
	\end{proof}
	
	\begin{remark}
		Note that it follows from the proof of (iv) $\Rightarrow$ (iii) above that we only have to require that monotones from a certain subset of $\mathcal{H}^{\rm n.d.}(S^d,\R_+^{\rm op})$ are strictly less than $1$, namely all those of the type $\Phi_{e_k,\{k,k^\prime\}}$, where $k,k^\prime\in[d]$, $k\neq k^\prime$. 		
	\end{remark}
	
	\subsection{Results: applying the Vergleichsstellensatz}
    \label{subsec:applyingvssnorestrictions}
	
	In the previous sections we have characterized all monotone homomorphisms, the monotone derivations at degenerate monotone homomorphisms, and the power universal elements in the matrix majorization semiring $S^d$. Now, we will apply the Vergleichsstellensatz in the form of Theorem \ref{thm:Fritz2022} to $S^d$ to obtain results on the large-sample and catalytic majorization of tuples of probability distributions. We obtain sufficient (and almost necessary) conditions for {\it exact} majorization in large samples and catalytically in Theorem \ref{thm:majorization} below. Furthermore, we derive necessary and sufficient conditions for majorization in large samples and catalytically where the majorization only happens {\it approximately} with asymptotically vanishing error in Theorem \ref{thm:genasymptocatalytic}.
	
	The conditions, in terms of the monotone homomorphisms, for a matrix $P\in \V^d$  to majorize another $Q\in \V^d$ in the large-sample or catalytic setting, can be expressed as 
	\begin{equation}\label{eq:majorizationconditions}
		\Phi_{\underline{\alpha},C}(P)<\Phi_{\underline{\alpha},C}(Q)\quad\quad\forall\; \Phi_{\underline{\alpha},C}\in\mathcal{H}^{\rm n.d.}(S^d,\R_+^{\rm op}). 
	\end{equation}
    These homomorphisms $\Phi_{\underline{\alpha},C}$ are defined in \eqref{eq:homomorphismsum} and if we adopt the convention $0^0=1$, we can rewrite these more conveniently as follows: 
    \begin{equation}
        \label{eq:PhiR+op}
        \Phi_{\underline{\alpha},C}(P)=\sum_{i\in I(P,C)}\big(p^{(1)}_i\big)^{\alpha_1}\cdots\big(p^{(d)}_i\big)^{\alpha_d}, 
    \end{equation}
    where $\underline{\alpha}\in A_+$ (defined in \eqref{eq:A+}), $C\subseteq[d]$ such that ${\rm supp}\,\underline{\alpha}\subseteq C$, and the sum runs only over the rows in $I(P,C)$ (defined in \eqref{eq:Iset}). Furthermore, we exclude the degenerate homomorphisms $\Phi^{(k)}(P)=\|p^{(k)}\|_1$, which equals the sum of the entries in column $k$, for $k\in[d]$. 
	
	We note that if $P,Q\in\V^d$ have columns that sum to $1$ and satisfy \eqref{eq:majorizationconditions}, then $P$ is automatically power universal. To see this, observe that for any $\Phi_{\underline{\alpha},C}\in\mathcal{H}^{\rm n.d.}(S^d,\R_+^{\rm op})$ we have $\Phi_{\underline{\alpha},C}(P)<\Phi_{\underline{\alpha},C}(Q)\leq1$, hence (iv) in Proposition \ref{lem:poweruniversalconditions} holds for $U=P$, which means $P$ is power universal. 
	
	Putting everything together, Theorem \ref{thm:Fritz2022} yields the following result. 
	\begin{theorem}\label{thm:majorization}
		Let $P=(p^{(1)},\ldots,p^{(d)})$, $Q=(q^{(1)},\ldots,q^{(d)})$ be tuples of probability vectors such that $\bigcap_{k=1}^d{\rm supp}\,p^{(k)}\neq\emptyset$ and $\bigcap_{k=1}^d{\rm supp}\,q^{(k)}\neq\emptyset$. If the inequalities in \eqref{eq:majorizationconditions} hold, then we have the following: 
		\begin{enumerate}[(a)]
			\item There exist a stochastic map $T$ and a tuple of probability vectors $(r^{(1)},\ldots,r^{(d)})$, satisfying $\bigcap_{k=1}^d{\rm supp}\,r^{(k)}\neq\emptyset$, such that
			\begin{equation}
				T(r^{(k)}\otimes p^{(k)})=r^{(k)}\otimes q^{(k)}\quad\text{for }k\in[d], 
			\end{equation}
			and we may choose, for $n$ sufficiently large, 
			\begin{equation}
                \label{eq:rform}
				r^{(k)}=\frac{1}{n+1}\bigoplus_{\ell=0}^n\left(q^{(k)}\right)^{\otimes \ell}\otimes\left(p^{(k)}\right)^{\otimes(n-\ell)}\quad\text{ for }k\in[d]. 
			\end{equation}
			\item For $n$ sufficiently large there exists a stochastic map $T_n$ such that 
			\begin{equation}
				T_n(p^{(k)})^{\otimes n}=(q^{(k)})^{\otimes n}\quad\text{for }k\in[d]. 
			\end{equation}
		\end{enumerate}
		Conversely, if either one of these holds, then the inequalities in \eqref{eq:majorizationconditions} hold non-strictly. 
	\end{theorem}
	
	Just as in the case with equal support in \cite{farooq2024}, we have an asymptotic version of majorization. Note that now, as opposed to the exact version of majorization in the Theorem above, we have to additionally require that $P$ is power universal, which is also necessary for the approximate version in the equal-support case in \cite{farooq2024}.  
	\begin{theorem}\label{thm:genasymptocatalytic}
		Let $P=(p^{(1)},\ldots,p^{(d)})$, $Q=(q^{(1)},\ldots,q^{(d)})$ be tuples of probability vectors such that $\bigcap_{k=1}^d{\rm supp}\,p^{(k)}\neq\emptyset$ and $\bigcap_{k=1}^d{\rm supp}\,q^{(k)}\neq\emptyset$, and $P$ is power universal. The following are equivalent: 
		\begin{enumerate}[(i)]
			\item For all $\Phi_{\underline{\alpha},C}\in\mathcal{H}^{\rm n.d.}(S^d,\R_+^{\rm op})$ that are continuous (see \eqref{eq:continuity}) it holds that 
            \begin{equation}
                \Phi_{\underline{\alpha},C}(P)\leq \Phi_{\underline{\alpha},C}(Q). 
            \end{equation}
			\item For every $\varepsilon>0$ there exist two tuples of probability vectors $(q_\varepsilon^{(1)},\ldots,q_\varepsilon^{(d)})$ and $(r_\varepsilon^{(1)},\ldots,r_\varepsilon^{(d)})$, satisfying $\bigcap_{k=1}^d{\rm supp}\,q_\varepsilon^{(k)}\neq\emptyset$ and $\bigcap_{k=1}^d{\rm supp}\,r_\varepsilon^{(k)}\neq\emptyset$, and a stochastic map $T_\varepsilon$ such that
			\begin{equation}
				\|q^{(k)}-q_\varepsilon^{(k)}\|_1\leq\varepsilon\quad\text{for }k\in[d], 
			\end{equation}
			and 
			\begin{equation}
				T_\varepsilon(r_\varepsilon^{(k)}\otimes p^{(k)})=r_\varepsilon^{(k)}\otimes q_\varepsilon^{(k)}\quad\text{for }k\in[d]. 
			\end{equation}
			Also, we may choose, for $n$ sufficiently large
			\begin{equation}
				r_\varepsilon^{(k)}=\frac{1}{n+1}\bigoplus_{\ell=0}^n\left(q_\varepsilon^{(k)}\right)^{\otimes\ell}\otimes\left(p^{(k)}\right)^{\otimes(n-\ell)}\quad\text{ for }k\in[d]. 
			\end{equation}
			\item For every $\varepsilon>0$ there exist a tuple of probability vectors $(q_\varepsilon^{(1)},\ldots,q_\varepsilon^{(d)})$, satisfying\\$\bigcap_{k=1}^d{\rm supp}\,q_\varepsilon^{(k)}\neq\emptyset$, such that
			\begin{equation}
				\|q^{(k)}-q_\varepsilon^{(k)}\|_1\leq\varepsilon\quad\text{for }k\in[d], 
			\end{equation}
			and for $n$ sufficiently large a stochastic map $T_{\varepsilon,n}$ such that 
			\begin{equation}
				T_{\varepsilon,n}(p^{(k)})^{\otimes n}=(q_\varepsilon^{(k)})^{\otimes n}\quad\text{for }k\in[d]. 
			\end{equation}
		\end{enumerate}
	\end{theorem}
	\begin{proof}
		(i) $\Rightarrow$ (ii) and (iii): Fix any $k\in[d]$. Define $W:=(w,\ldots,w)$, where $w:=(1/n,\ldots,1/n)\in\R^n$, and 
		\begin{equation}
			Q_\varepsilon:=\left(1-\frac{\varepsilon}{2}\right)Q+\frac{\varepsilon}{2}W, 
		\end{equation}
		where $+$ means entry-wise addition of two matrices. Write $Q_\varepsilon=(q_\varepsilon^{(1)},\ldots,q_\varepsilon^{(d)})$. Let $\Phi_{\underline{\alpha},C}\in\mathcal{H}^{\rm n.d.}(S^d,\R_+^{\rm op})$ be continuous. Then 
		\begin{align}
			\Phi_{\underline{\alpha},C}&(Q_\varepsilon)=\Phi_{\underline{\alpha},C}\left(\left(1-\frac{\varepsilon}{2}\right)Q+\frac{\varepsilon}{2}W\right)\geq \Phi_{\underline{\alpha},C}\left(\left(1-\frac{\varepsilon}{2}\right)Q\boxplus\frac{\varepsilon}{2}W\right)\\
			&=\left(1-\frac{\varepsilon}{2}\right)\Phi_{\underline{\alpha},C}(Q)+\frac{\varepsilon}{2}\Phi_{\underline{\alpha},C}(W)=\Phi_{\underline{\alpha},C}(Q)+\frac{\varepsilon}{2}[1-\Phi_{\underline{\alpha},C}(Q)]\geq \Phi_{\underline{\alpha},C}(Q)\geq \Phi_{\underline{\alpha},C}(P), 
		\end{align}
		where in the first inequality we used monotonicity, in the equality that follows we used additivity and the fact that $\Phi_{\underline{\alpha},C}(cA)=c\Phi_{\underline{\alpha},C}(A)$ for any $c>0$ and $A\in \V^d$, in the next equality we used that $\Phi_{\underline{\alpha},C}(W)=1$ since all the columns in $W$ are the same and sum to $1$, in the penultimate inequality we used that $\Phi_{\underline{\alpha},C}(Q)\leq1$, and in the final inequality we used the assumption in (i). If $\Phi_{\underline{\alpha},C}(Q)=1$, then the final inequality is strict since $\Phi_{\underline{\alpha},C}(P)<1$ by Lemma \ref{lem:poweruniversalconditions}. If $\Phi_{\underline{\alpha},C}(Q)<1$, then the penultimate inequality is strict. Thus, in both cases $\Phi_{\underline{\alpha},C}(P)<\Phi_{\underline{\alpha},C}(Q_\varepsilon)$.
		
		Let $\Phi_{\underline{\alpha},C}\in\mathcal{H}^{\rm n.d.}(S^d,\R_+^{\rm op})$ be discontinuous. Since $Q_\varepsilon$ has no $0$ entries, $\Phi_{\underline{\alpha},C}(Q_\varepsilon)=\Phi_{\underline{\alpha},D}(Q_\varepsilon)$, where $D\subseteq C$ consists of all $c\in C$ that satisfy $\alpha_c>0$. If $\Phi_{\underline{\alpha},D}$ is one of the degenerate homomorphisms (i.e. the $1$-norm of one of the columns), then 
        \begin{equation}
			\Phi_{\underline{\alpha},C}(P)<1=\Phi_{\underline{\alpha},D}(Q_\varepsilon)=\Phi_{\underline{\alpha},C}(Q_\varepsilon), 
		\end{equation}
        where the inequality follows from Lemma \ref{lem:poweruniversalconditions}. If $\Phi_{\underline{\alpha},D}$ is nondegenerate, then we find 
		\begin{equation}
			\Phi_{\underline{\alpha},C}(P)\leq \Phi_{\underline{\alpha},D}(P)<\Phi_{\underline{\alpha},D}(Q_\varepsilon)=\Phi_{\underline{\alpha},C}(Q_\varepsilon), 
		\end{equation}
		where the first inequality is due to the fact that $\Phi_{\underline{\alpha},C}$ sums over all or a subset of the rows that $\Phi_{\underline{\alpha},D}$ sums over, and the second inequality is due to the fact that $\Phi_{\underline{\alpha},D}$ is continuous and nondegenerate. 
		
		We conclude that $\Phi_{\underline{\alpha},C}(P)<\Phi_{\underline{\alpha},C}(Q_\varepsilon)$ for all $\Phi_{\underline{\alpha},C}\in\mathcal{H}^{\rm n.d.}(S^d,\R_+^{\rm op})$, hence Theorem \ref{thm:majorization} yields (ii) and (iii). Note that one easily sees that $\bigcap_{k=1}^d{\rm supp}\,q_\varepsilon^{(k)}\neq\emptyset$ and the fact that $\bigcap_{k=1}^d{\rm supp}\,r_\varepsilon^{(k)}\neq\emptyset$ in (ii) follows from the possible choice of the $r_\varepsilon^{(k)}$ in \eqref{eq:rform}. 
		
		(ii) or (iii) $\Rightarrow$ (i): The fact that $\bigcap_{k=1}^d{\rm supp}\,q_\varepsilon^{(k)}\neq\emptyset$, respectively $\bigcap_{k=1}^d{\rm supp}\,r_\varepsilon^{(k)}\neq\emptyset$, implies that $(q_\varepsilon^{(1)},\ldots,q_\varepsilon^{(d)})\in S^d$, respectively $(r_\varepsilon^{(1)},\ldots,r_\varepsilon^{(d)})\in S^d$. Then, for any continuous $\Phi\in\mathcal{H}^{\rm n.d.}(S^d,\R_+^{\rm op})$ we find for any $\varepsilon>0$, using monotonicity and multiplicativity, that $\Phi(P)\leq \Phi(Q_\varepsilon)$, where $Q_\varepsilon:=(q_\varepsilon^{(1)},\ldots,q_\varepsilon^{(d)})$. For any $\varepsilon>0$, write $Q_\varepsilon=Q_\varepsilon^0\boxplus Q_\varepsilon^1$, where $Q_\varepsilon^0$ consists of the first $n_Q$ rows of $Q_\varepsilon$ and $n_Q$ is the number of rows in $Q$ (if $Q_\varepsilon$ has less than $n_Q$ rows, then we add rows of zeros to form $Q_\varepsilon^0$). Note that $\lim_{\varepsilon\rightarrow0}Q_\varepsilon^0=Q$ and the sum of every column in $Q_\varepsilon^1$ is at most $\varepsilon$. Then 
        \begin{equation}
            \Phi(Q_\varepsilon)+1=\Phi(Q_\varepsilon\boxplus(1\,\cdots\,1))=\Phi(Q_\varepsilon^0)+\underbrace{\Phi(Q_\varepsilon^1\boxplus(1\,\cdots\,1))}_{\leq1+\varepsilon}, 
        \end{equation}
        where in the inequality we used monotonicity and the fact that the rows of $Q_\varepsilon^1\boxplus(1\,\cdots\,1)$ sum to at most $1+\varepsilon$. Hence $\Phi(Q_\varepsilon)\leq\Phi(Q_\varepsilon^0)+\varepsilon$. This implies that $\Phi(P)\leq\Phi(Q_\varepsilon^0)+\varepsilon$. Taking the limit $\varepsilon\rightarrow0$ in this inequality, since $\Phi$ is a continuous function when restricted to matrices with a fixed number of rows, we deduce that $\Phi(P)\leq\Phi(Q_\varepsilon)$, completing the proof.     
	\end{proof}

    \begin{remark}
    \label{rmk:minimalrestrictions}
        As mentioned in the Introduction, we could drop the requirement \eqref{eq:baselinecond}, namely that the intersection of the supports of the columns in a matrix is non-empty, and require instead that the pairwise intersection of supports is non-empty. This would yield a semiring with even less restrictions than the one studied in this section. However, in that case, the homomorphisms and derivations appearing in the conditions of the Vergleichsstellensatz (Theorem \ref{thm:Fritz2022}) will turn out to be only the bivariate ones, namely those associated with the R\'enyi divergences between each pair of columns of the matrix. To get an idea why this is so, consider the $d=3$ case. A matrix of the form 
        \begin{equation}
            P:=\left(\begin{array}{ccc}
            *&*&0\\
            *&0&*\\
            0&*&*
            \end{array}\right), 
        \end{equation}
        where the asterisks represent positive entries, has pairwise non-empty intersections of the supports, but does not have non-empty intersection over all the supports. Any homomorphism $\Phi_{\underline{\alpha},C}$ with a character $C$ containing more than two columns will evaluate to zero on the matrix above, since the set of rows $I(P,C)$ defined in \eqref{eq:Iset} is empty. Hence $\Phi_{\underline{\alpha},C}$ has a non-trivial kernel, and such homomorphisms are excluded from the set of homomorphisms appearing in the conditions of Theorem \ref{thm:Fritz2022}. Therefore, the only non-degenerate monotone homomorphisms with trivial kernel associated with this semiring are $\Phi_{\underline{\alpha},\{k,k'\}}$ with $k,k'\in[d]$, $k\neq k'$, $\underline{\alpha}\in A_+$ having at most two non-zero components $\alpha_k,\alpha_{k'}\geq0$, and $\alpha_{l}=0$ for all $l\neq k,k'$. These homomorphisms correspond to the R\'enyi divergences between any pair of columns, with parameter $\alpha\in[0,1)$; see also Figure \ref{fig:emptytriangle} for a depiction.

        Additionally, the conditions for a matrix to be power universal are even stricter in this semiring than in $S^d$. Namely, a matrix having any common support among all columns is now no longer power universal, hence we can no longer use this semiring to find majorization conditions for such a matrix.  
    \end{remark}

\begin{figure}
\begin{center}
\begin{overpic}[scale=0.43,unit=1mm]{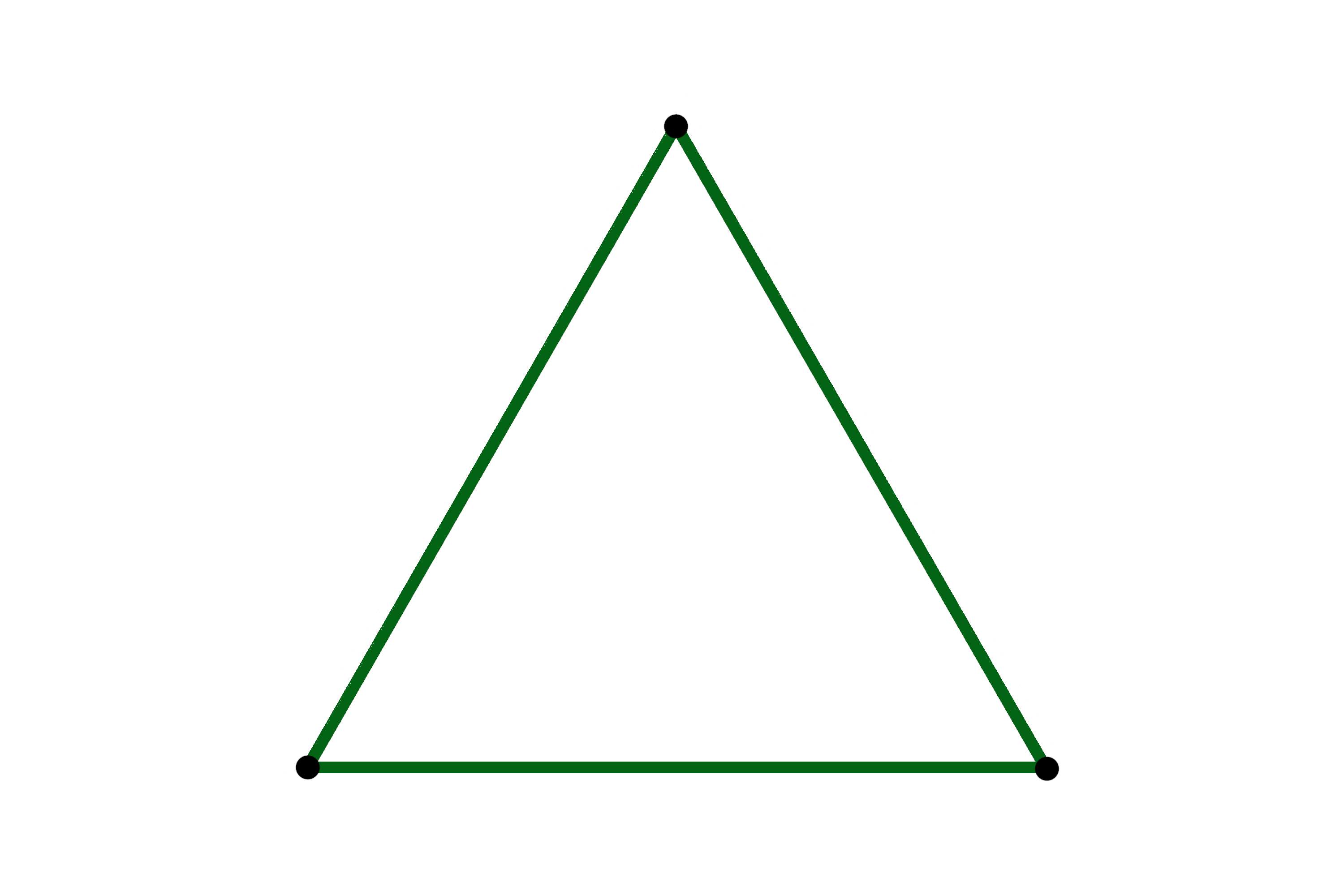}
\put(12,5){\begin{large}
$\Phi^{(1)}$
\end{large}}
\put(88,5){\begin{large}
$\Phi^{(2)}$
\end{large}}
\put(50,66){\begin{large}
$\Phi^{(3)}$
\end{large}}
\put(41,13){\begin{large}
\darker{$\Phi_{\alpha,1-\alpha,0,\{1,2\}}$}
\end{large}}
\put(70,35){\begin{large}
\darker{$\Phi_{0,\beta,1-\beta,\{2,3\}}$}
\end{large}}
\put(11,35){\begin{large}
\darker{$\Phi_{\gamma,0,1-\gamma,\{1,3\}}$}
\end{large}}
\put(21,6){
\darker{$\Phi_{e_1,\{1,2\}}$}
}
\put(10,13){
\darker{$\Phi_{e_1,\{1,3\}}$}
}
\put(74,6){
\darker{$\Phi_{e_2,\{1,2\}}$}
}
\put(83,13){
\darker{$\Phi_{e_2,\{2,3\}}$}
}
\put(38,62){
\darker{$\Phi_{e_3,\{1,3\}}$}
}
\put(55,62){
\darker{$\Phi_{e_3,\{2,3\}}$}
}
\end{overpic}
\caption{\label{fig:emptytriangle} A depiction of the monotone homomorphisms of the semiring discussed in Remark \ref{rmk:minimalrestrictions} in the case $d=3$. Only bivariate R\'{e}nyi divergences involving only two columns of the matrices are present. On the line segment connecting $e_1$ and $e_2$, e.g.,\ the homomorphisms correspond to $D_\alpha\big(p^{(1)}\big\|p^{(2)}\big)$ and $D_\alpha\big(p^{(2)}\big\|p^{(1)}\big)$ for $0\leq\alpha\leq 1/2$. Note that at the corner points we have again two nondegenerate (bivariate) homomorphisms in addition to the degenerate $\Phi^{(1)}$, $\Phi^{(2)}$, and $\Phi^{(3)}$.}
\end{center}
\end{figure}
	
	\section{One dominating column case}\label{sec:Asymmetric}
	
	In this section we will consider tuples of probability vectors where the support of one vector contains the support of all other vectors. Such tuples have applications in quantum thermodynamics, as we will see later. Similarly as for $S^d$, we will use the Vergleichsstellensatz given by Theorem \ref{thm:Fritz2022} to derive conditions for large-sample and catalytic majorization. 
	
	We define the \emph{dominating column matrix majorization semiring} $S^d_{\rm d.c.}$ as the subset of all (equivalence classes of) matrices in $S^d$ where the support of the last column contains all others: 
	\begin{equation}
		S^d_{\rm d.c.}:=\left\{\,\left[\left(p^{(1)},\ldots,p^{(d)}\right)\right]\in S^d\;\middle|\;{\rm supp}\,p^{(k)}\subseteq {\rm supp}\,p^{(d)}\quad\forall k\in[d-1]\,\right\}. 
	\end{equation}
 The set of individual matrices (i.e.,\ lifting the projection into equivalence classes) generating $S^d_{\rm d.c.}$ is denoted by $\mc V^d_{\rm d.c.}$. The subset of matrices in $\mc V^d_{\rm d.c.}$ with columns of length $n\in\N$ is denoted by $(\mc V_{\rm d.c.})^d_n$. In subsection \ref{subsec:poweruniversaldc}, we identify the power universals of $S^d_{\rm d.c.}$ thus showing that $S^d_{\rm d.c.}$ is of polynomial growth.
		
\subsection{Monotone homomorphisms}

Since $S^d_{\rm d.c.}$ is a subset of $S^d$, all monotone homomorphisms of $S^d$, when restricted to $S^d_{\rm d.c.}$, are also monotone homomorphisms of $S^d_{\rm d.c.}$. Thus, especially functions in $\mathcal{H}^{\rm n.d.}(S^d_{\rm d.c.},\R_+^{\rm op})$, which is defined in exactly the same way as in \eqref{eq:temperateopposite}, are non-degenerate monotone homomorphisms $S^d_{\rm d.c.}\to\R_+^{\rm op}$. Note that in this case some homomorphisms that are different on $S^d$ coincide on $S^d_{\rm d.c.}$, since column $d$ always has a non-zero entry and therefore $\Phi_{\underline{\alpha},C}=\Phi_{\underline{\alpha},C\cup\{d\}}$ for any $\underline{\alpha}\in A_+$ and ${\rm supp}\,\underline{\alpha}\subseteq C\subseteq[d]$. This also means that the homomorphisms $\Phi_{e_k,\{k,d\}}$ for $k\in[d-1]$, which are nondegenerate on $S^d$, are degenerate on $S^d_{\rm d.c.}$. 

Let $\alpha>1$ and $c\in[d-1]$. Define $\varphi_{\alpha,c}:\R_+\times\R_{>0}\to\R$ by 
\begin{equation}
    \varphi_{\alpha,c}(x_c,x_d):=x_c^\alpha x_d^{1-\alpha}. 
\end{equation}
It follows from the proof of Proposition 13 in \cite{farooq2024} that $\varphi_{\alpha,c}$ is convex on $\R_{>0}^2$. Using the fact that $\alpha>0$ and the continuity of $\varphi_{\alpha,c}$, we find that $\varphi_{\alpha,c}$ is convex on its entire domain $\R_+\times\R_{>0}$. In the following, we will abuse notation by writing $\varphi_{\alpha,c}(x)$ for $x$ in a higher dimensional space $\R^d$ instead of $x\in \R_+\times\R_{>0}$, where it is understood that we ignore the coordinates not equal to $c$ or $d$. 

We now define the homomorphism $\Phi_{\alpha,c}:S^d\to\R_+$ (where we assume WLOG that $p^{(d)}_i>0$ for all $i=1,\ldots,n$) by 
\begin{equation}
\label{eq:homdcR+}
    \Phi_{\alpha,c}(P):=\sum_{i=1}^n\varphi_{\alpha,c}\left(p_i^{(c)},p_i^{(d)}\right),\quad\quad P=(p^{(1)},\ldots,p^{(d)})\in (\V_{\rm d.c.})_{n}^d. 
\end{equation}
We denote all these homomorphisms $\Phi_{\alpha,c}$ that are non-degenerate by $\mathcal{H}^{\rm n.d.}(S^d_{\rm d.c.},\R_+)$, i.e.,
\begin{equation}
\mathcal{H}^{\rm n.d.}(S^d_{\rm d.c.},\R_+)=\big\{\Phi_{\alpha,c}\,\big|\,1<\alpha<\infty,\ c=1,\ldots,d-1\big\}.
\end{equation}
	
	\begin{proposition}
		All homomorphisms in $\mathcal{H}^{\rm n.d.}(S^d_{\rm d.c.},\R_+)$ are monotone. 
	\end{proposition}
	
	\begin{proof}
		Let $\alpha>1$ and $c\in[d-1]$. Consider $P\in S^d_{\rm d.c.}$ with rows $p_1,\ldots,p_n\in\R_+^d$ and a stochastic $m\times n$-matrix $T$, and write $Q:=TP\in S^d_{\rm d.c.}$ with rows $q_1,\ldots,q_m\in\R_+^d$, and $T_i:=\sum_{j=1}^n T_{i,j}>0$. Using the convexity of $\varphi_{\alpha,c}$ and the fact that $\left(p_j\right)_d>0$ for all $j\in[n]$, we find 
		\begin{align}
			\Phi_{\alpha,c}(Q)&=\sum_{i=1}^m\varphi_{\alpha,c}(q_i)=\sum_{i=1}^m\varphi_{\alpha,c}\left(\sum_{j=1}^n T_{i,j}p_j\right)=\sum_{i=1}^m T_i\varphi_{\alpha,c}\left(\sum_{j=1}^n \frac{T_{i,j}}{T_i}p_j\right)\\
			&\leq\sum_{i=1}^m T_i\sum_{j=1}^n \frac{T_{i,j}}{T_i}\varphi_{\alpha,c}(p_j)=\sum_{j=1}^n \varphi_{\alpha,c}(p_j)\underbrace{\sum_{i=1}^m T_{i,j}}_{=1}=\Phi_{\alpha,c}(P), 
		\end{align}
		which completes the proof. 		
	\end{proof}

Let us also define, for all $P=\big(p^{(1)},\ldots,p^{(d)}\big)\in\mc V^d_{\rm d.c.}$ with $p^{(k)}=\big(p^{(k)}_1,\ldots,p^{(k)}_n\big)$ for $k=1,\ldots,d$ (where we assume WLOG that $p^{(d)}_i>0$ for all $i=1,\ldots,n$) and $c=1,\ldots,d-1$,
\begin{equation}
\Phi^\T_c(P):=\max_{1\leq i\leq n}\frac{p^{(c)}_i}{p^{(d)}_i}.
\end{equation}
We also define
\begin{equation}
\mathcal{H}^{\rm n.d.}(S^d_{\rm d.c.},\T\R_+):=\big\{\Phi^\T_{\beta,c}:=\big(\Phi^\T_c(\cdot)\big)^\beta\,\big|\,\beta>0\big\}.
\end{equation}
The functions in $\mathcal{H}^{\rm n.d.}(S^d_{\rm d.c.},\T\R_+\}$ are easily seen to be homomorphisms $S^d_{\rm d.c.}\to\T\R_+$. The conditions for large-sample or catalytic majorization of $P$ over $Q$ require in particular that $\Phi^\T_{\beta,c}(P)>\Phi^\T_{\beta,c}(Q)$ for all $\Phi^\T_{\beta,c}\in\mathcal{H}^{\rm n.d.}(S^d_{\rm d.c.},\T\R_+)$. Notice that this is equivalent to $\Phi^\T_c(P)>\Phi^\T_c(Q)$ for all $c=1,\ldots,d-1$, and that $\Phi^\T_c(P)=D_\infty\big(p^{(c)}\|p^{(d)}\big)$, where $D_\infty$ is the max-divergence. We next show that the homomorphisms in $\mathcal{H}^{\rm n.d.}(S^d_{\rm d.c.},\T\R_+)$ are monotone.

\begin{proposition}
All homomorphisms in $\mathcal{H}^{\rm n.d.}(S^d_{\rm d.c.},\T\R_+)$ are monotone.
\end{proposition}

\begin{proof}
Let $\beta>0$ and $c\in[d-1]$. Note that, denoting the rows of $P\in\mc V^d_{\rm d.c.}$ by $p_1,\ldots,p_n$, we have
\begin{equation}
\Phi^\T_{\beta,c}(P)=\sup_{1\leq j\leq n}\fii^\T_{\beta,c}(p_j), 
\end{equation}
where $\fii^\T_{\beta,c}(x_1,\ldots,x_d)=(x_c/x_d)^\beta$ for all $x_1,\ldots,x_{d-1}\geq0$ and $x_d>0$. The sublevel sets
\begin{equation}
\left\{(x_1,\ldots,x_d)\in\R_+^{d-1}\times\R_{>0}\,\middle|\,\left(\frac{x_c}{x_d}\right)^\beta\leq a\right\},\qquad a>0,
\end{equation}
are immediately seen to be convex, implying that $\fii^\T_{c,\beta}$ is quasi-convex, i.e.,\ for all $x,y\in\R_+^{d-1}\times\R_{>0}$ and $t\in(0,1)$,
\begin{equation}
\fii^\T_{c,\beta}\big(tx+(1-t)y\big)\leq\max\{\fii^\T_{c,\beta}(x),\fii^\T_{c,\beta}(y)\}.
\end{equation}

Let $P,Q\in\mc V^d_{\rm d.c.}$, $P$ with rows $p_1,\ldots,p_n$ and $Q$ with rows $q_1,\ldots,q_m$. Assume that $P\rgeq Q$ and let $T=(T_{i,j})_{i,j}$ be a stochastic matrix such that $Q=TP$. We again denote $T_i=\sum_{j=1}^n T_{i,j}$ which we may assume to be positive WLOG. Noting that $\fii^\T_{c,\beta}(cx)=\fii^\T_{c,\beta}(x)$ for all $c>0$ and $x\in\R_+^{d-1}\times\R_{>0}$ and using the quasi-convexity of $\fii^\T_{c,\beta}$, we now have
\begin{align}
\Phi^\T_{c,\beta}(Q)&=\max_{1\leq i\leq m}\fii^\T_{c,\beta}(q_i)=\max_{1\leq i\leq m}\fii^\T_{c,\beta}\left(\sum_{j=1}^n T_{i,j}p_j\right)\\
&=\max_{1\leq i\leq m}\fii^\T_{c,\beta}\left(\sum_{j=1}^n\frac{T_{i,j}}{T_i}p_j\right)\leq\max_{1\leq i\leq m}\max_{1\leq j\leq n}\fii^\T_{c,\beta}(p_j)\\
&=\max_{1\leq j\leq n}\fii^\T_{c,\beta}(p_j)=\Phi^\T_{c,\beta}(P),
\end{align}
proving the claim.
\end{proof}
	
	\begin{proposition}
		The non-degenerate monotone homomorphisms in $S_{\rm d.c.}^d$ are exactly those in the union of $\mathcal{H}^{\rm n.d.}(S^d_{\rm d.c.},\R_+)$, $\mathcal{H}^{\rm n.d.}(S^d_{\rm d.c.},\R_+^{\rm op})$, and $\mathcal{H}^{\rm n.d.}(S^d_{\rm d.c.},\T\R_+)$. There are no nondegenerate monotone homomorphisms $S^d_{\rm d.c.}\to\T\R_+^{\rm op}$.
	\end{proposition}
	
	\begin{proof}
Throughout this proof, we consider an arbitrary matrix $P=\big(p^{(1)},\ldots,p^{(d)}\big)\in\mc V^d_{\rm d.c.}$ with rows $p^{(k)}=\big(p^{(k)}_1,\ldots,p^{(k)}_n\big)$ for $k=1,\ldots,d$. We also denote the rows of $P$ by $p_1,\ldots,p_n$.

We first consider the case of a monotone homomorphism $\Phi:S^d_{\rm d.c.}\to\mathbb{K}\in\{\R_+,\R_+^{\rm op}\}$. As in the minimally restricted case $S^d$, also now $\Phi$ is associated with $\underline{\alpha}\in A_+\cup A_-\setminus\{e_1,\ldots,e_d\}$. Lemmas \ref{lem:1or2} and \ref{clm:nontrivialB} can be proven in exactly the same way (except now we only consider $B\subsetneq[d]$ such that $d\in B$) and we will use these results in the following. Hence, the notion of $\Phi$-nontrivial sets can again be defined and $\Phi$ also has a character $C\subseteq[d]$. Through the same proof as the one for Proposition \ref{pro:monhommin}, we know that
\begin{equation}
\Phi(P)=\sum_{i:\,C\subseteq{\rm supp}\,p_i}\prod_{k\in{\rm supp}\,p_i}\big(p_i^{(k)}\big)^{\alpha_k}
\end{equation}
and that ${\rm supp}\,\underline{\alpha}=\{k\,|\,\alpha_k\neq0\}\subseteq C$. As for the case $\mathbb{K}=\R_+^{\rm op}$, we already know that $\underline{\alpha}\in A_+$, so this case is dealt with and the set of nondegenerate monotone homomorphisms coincides with $\mathcal{H}^{\rm n.d.}(S^d_{\rm d.c.},\R_+^{\rm op})$. Thus we may concentrate on the case $\mathbb{K}=\R_+$.

Assume that $\mathbb{K}=\R_+$. From the results of \cite{farooq2024}, we know that $\underline{\alpha}\in A_-$. Let us first concentrate on the subset of those $\underline{\alpha}=(\alpha_1,\ldots,\alpha_d)$ where $\alpha_1>1$ and $\alpha_2,\ldots,\alpha_d\leq 0$. We show that $\alpha_2=\cdots=\alpha_{d-1}=0$. Consider
\begin{equation}
P_t:=\left(\begin{array}{ccccc}
1/2&1-t&1/2&\cdots&1/2\\
1/2&t&1/2&\cdots&1/2
\end{array}\right),\qquad 0\leq t\leq\frac{1}{2}.
\end{equation}
Assume $\alpha_2\neq0$, hence $\alpha_2<0$. Then, by Lemma \ref{lem:1or2}(ii), $2\in C$. Whenever $0\leq s\leq t\leq 1/2$, we have $P_s\rgeq P_t$. Thus, $\Phi(P_\varepsilon)\leq\Phi(P_0)$, i.e.,
\begin{equation}
\label{eq:onlytwocolumns}
2^{-\alpha_1-\alpha_3-\cdots-\alpha_d}\left[(1-\varepsilon)^{\alpha_2}+\varepsilon^{\alpha_2}\right]\leq2^{-\alpha_1-\alpha_3-\cdots-\alpha_d}
\end{equation}
for all $\varepsilon\in(0,1/2]$. However, the left-hand side above diverges as $\varepsilon\to 0$, so $\alpha_2=0$. Similarly, $\alpha_3=\cdots=\alpha_{d-1}=0$. The same logic does not work for $\alpha_d$ because
\begin{equation}
\left(\begin{array}{cccc}
1/2&\cdots&1/2&1\\
1/2&\cdots&1/2&0
\end{array}\right)\notin\mc V^d_{\rm d.c.}, 
\end{equation}
hence we are not able to consider matrices similar to the $P_t$ for $0\leq t\leq1/2$ as above, but now with the last column equal to $(1-t,t)$. 

We know that $1,d\in C$ and we now show that there are no other columns in $C$. Assume that $C$ contains another column $k\in C\setminus\{1,d\}$. WLOG we assume that $k=2$. Then, \eqref{eq:onlytwocolumns} holds with $\alpha_1+\alpha_d=1$ and $\alpha_3=\cdots=\alpha_{d-1}=0$, yielding $1\leq1/2$, a contradiction. Hence $C=\{1,d\}$. Therefore, in the part of $A_-$ where $\alpha_1>1$, the only non-degenerate homomorphisms $S^d_{\rm d.c.}\to\R_+$ are all $\Phi_{\alpha,1}$ for $\alpha_d>1$. In the same way, we can deal with the parts of $A_-$ where $\alpha_c>1$, $c=2,\ldots,d-1$ to show that the only homomorphisms are all $\Phi_{\alpha,c}$, for $\alpha_c>1$, $c=2,\ldots,d-1$. 

It remains to be shown what the homomorphisms are for the part of $A_-$ where $\alpha_d>1$. Assume that $\alpha_d\geq1$. 
Considering matrices
\begin{equation}
Q_t:=\left(\begin{array}{cccc}
1-t&1/2&\cdots&1/2\\
t&1/2&\cdots&1/2
\end{array}\right),\qquad 0\leq t\leq\frac{1}{2},
\end{equation}
we find that $Q_s\rgeq Q_t$ for $0\leq s\leq t\leq 1/2$, so $\Phi(Q_\varepsilon)\leq\Phi(Q_0)$ for all $\varepsilon\in(0,1/2]$. However, for $\alpha_1<0$, this again leads to a contradiction, so $\alpha_1=0$. Similarly, $\alpha_2=\cdots=\alpha_{d-1}=0$. Thus, $\alpha_d=1$ since $\alpha_1+\cdots+\alpha_d=1$. However, this case falls under the set $\mathcal{H}^{\rm n.d.}(S^d_{\rm d.c.},\R_+^{\rm op})$ which is already dealt with. Thus, the set of nondegenerate monotone homomorphisms $S^d_{\rm d.c.}\to\R_+$ is exactly $\mathcal{H}^{\rm n.d.}(S^d_{\rm d.c.},\R_+)$.

We now consider the case $\mathbb{K}=\T\R_+$. From the results of \cite{farooq2024}, we know that there is $\underline{\beta}=(\beta_1,\ldots,\beta_d)\in B_-$ such that, when the columns of $P$ have full support,
\begin{equation}
\Phi(P)=\max_{1\leq i\leq n}\prod_{k=1}^d \big(p^{(k)}_i\big)^{\beta_k}.
\end{equation}
Proceeding similarly as above for $\mathbb{K}=\R_+$ and considering the matrices $P_t$ and $Q_t$, one can again show that $\beta_d<0$ and that there is $c\in\{1,\ldots,d-1\}$ such that $\beta_c>0$ and $\beta_k=0$ for $k\in[d-1]\setminus\{c\}$. Thus, $\beta:=\beta_c=-\beta_d$.

  

Note that we may assume $P$ to have columns that sum up to $1$ because, in the general case, 
\begin{equation}
\Phi(P)=\Phi\left(\|p^{(1)}\|_1\,\cdots\,\|p^{(d)}\|_1\right)\Phi(P^\prime) 
\end{equation}
for a matrix $P'$ with columns that sum to $1$, and $\Phi\left(\|p^{(1)}\|_1\,\cdots\,\|p^{(d)}\|_1\right)$ is known since $\Phi$ is evaluated on a matrix with common support. Thus, we assume for now that $\|p^{(k)}\|=1$ for $k=1,\ldots,d$. By monotonity, this implies that $\Phi(P)\geq1$. 
  
        Let $p_1,\ldots,p_n\in\R_+^d$ denote the rows of $P$. Let $\delta>0$, which we will fix later. If a row $p_i$ contains a zero, but $(p_i)_k>0$, then we denote by $\widetilde{p_i}$ the row $p_i$ with every $0$ replaced by a $1$. If a row $p_i$ contains a zero, and $(p_i)_k=0$ is one of them, then we denote by $\widetilde{p_i}^\delta$ the row $p_i$ with $\delta$ on position $k$ and every other $0$ replaced by a $1$. Then, notice that 
		\begin{equation}
			\label{eq:PhiPdecomposedtropical}
			P\boxplus\bigboxplus_{k\in{\rm supp}\,p_i\neq[d]}\widetilde{p_i}\boxplus\bigboxplus_{k\notin{\rm supp}\,p_i\neq[d]}\widetilde{p_i}^\delta=\bigboxplus_{{\rm supp}\,p_i=[d]}p_i\boxplus\bigboxplus_{k\in{\rm supp}\,p_i\neq[d]}\left(\begin{array}{c}
				\widetilde{p_i}\\
				p_i
			\end{array}\right)\boxplus\bigboxplus_{k\notin{\rm supp}\,p_i\neq[d]}\left(\begin{array}{c}
				\widetilde{p_i}^\delta\\
				p_i
			\end{array}\right). 
		\end{equation}

Let us note that Lemma \ref{lem:1or2} still holds (with the obvious modification that $d\in B$), especially $\Phi[B]=1$ whenever $d\in B\subseteq[d]$. Using this, we have for the terms in the penultimate $\boxplus$-sum of the RHS \eqref{eq:PhiPdecomposedtropical} (i.e.,\ when $i$ is such that $k\in{\rm supp}\,p_i$ or, simply, $i\in{\rm supp}\,p^{(k)}$ and ${\rm supp}\,p_i\neq [d]$)
\begin{equation}
\Phi\big(\widetilde{p_i}\boxplus p_i\big)=\Phi\left(\widetilde{p_i}\boxtimes\left(
\begin{array}{ccc}
1&\cdots&1\\
&{\rm supp}\,p_i&
\end{array}\right)\right)=\Phi(\widetilde{p_i})\Phi[{\rm supp}\,p_i]=\Phi(\widetilde{p_i})=\left(\frac{p^{(k)}_i}{p^{(d)}_i}\right)^\beta.
\end{equation}
Similarly, for the terms in the final $\boxplus$-sum of the RHS of \eqref{eq:PhiPdecomposedtropical}, we have
\begin{equation}
\Phi(\widetilde{p_i}^\delta\boxplus p_i)=\left(\frac{\delta}{p^{(d)}_i}\right)^\beta.
\end{equation}
Then \eqref{eq:PhiPdecomposedtropical} implies 
		\begin{align}
			\label{eq:Phitropicaldominatingcolumnalmost}
   \max&\left(\{\Phi(P)\}\cup\left\{\left(\frac{p^{(k)}_i}{p^{(d)}_i}\right)^\beta\,\middle|\,i:\ k\in{\rm supp}\,p_i\neq[d]\right\}\cup\left\{\left(\frac{\delta}{p^{(d)}_i}\right)^\beta\,\middle|\,i:\ k\notin{\rm supp}\,p_i\neq[d]\right\}\right)\\
   &=\max\left(\left\{\left(\frac{p^{(k)}_i}{p^{(d)}_i}\right)^\beta\,\middle|\,i:\ k\in{\rm supp}\,p_i\right\}\cup\left\{\left(\frac{\delta}{p^{(d)}_i}\right)^\beta\,\middle|\,i:\ k\notin{\rm supp}\,p_i\right\}\right).
		\end{align}
        By making $\delta$ suitably small, we can make sure that both maximizations above do not have an element of the form $\delta^\beta\big(p^{(d)}_i\big)^{-\beta}$, for some row $i$, as largest element. 
		
		For any $\varepsilon>0$, we have 
		\begin{equation}
			P\boxplus\left(\begin{array}{ccc}
				\varepsilon&\cdots&\varepsilon\\
				\vdots&\ddots&\vdots\\
				\varepsilon&\cdots&\varepsilon
			\end{array}\right)\rgeq P+\left(\begin{array}{ccc}
				\varepsilon&\cdots&\varepsilon\\
				\vdots&\ddots&\vdots\\
				\varepsilon&\cdots&\varepsilon
			\end{array}\right), 
		\end{equation}
		where $+$ denotes the entry-wise addition of matrices. This implies 
		\begin{equation}
			\Phi(P)\geq\max_{i\in[n]}\left(\frac{p^{(k)}_i+\varepsilon}{p^{(d)}_i+\varepsilon}\right)^\beta, 
		\end{equation}
        where we used the assumption that $\Phi(P)\geq1$. By continuity of the functions $\varepsilon\mapsto\frac{p^{(k)}_i+\varepsilon}{p^{(d)}_i+\varepsilon}$, $i\in[n]$, there exists $\varepsilon_{\rm max}$ such that the above maximum is obtained at the same $i_0$ for all $\varepsilon\leq\varepsilon_{\rm max}$. Then, taking $\varepsilon\rightarrow0$ on both sides of the inequality above, we deduce 
		\begin{equation}
			\Phi(P)\geq\left(\frac{p^{(k)}_{i_0}}{p^{(d)}_{i_0}}\right)^\beta\quad\forall\,i\in[n]. 
		\end{equation}
		This together with \eqref{eq:Phitropicaldominatingcolumnalmost} yields 
        \begin{equation}
            \Phi(P)=\max_{i\in[n]}\left(\frac{p^{(k)}_i}{p^{(d)}_i}\right)^\beta. 
        \end{equation}
		
		Finally, we consider the case $\mathbb{K}=\T\R_+^{\rm op}$. Exactly as in the proof of Proposition 13 in \cite{farooq2024} (see Equations (89), (90)), we can show that $\Phi(P)$ only depends on the supports of the columns of $P$. More precisely, if $P,Q\in S_{\rm d.c.}^d$ satisfy ${\rm supp}\,p^{(k)}={\rm supp}\,q^{(k)}$ for all $k\in[d]$, then $\Phi(P)=\Phi(Q)$. It then follows that for any $Q\in S_{\rm d.c.}^d$, we have $\Phi(Q)\leq\Phi(\|q^{(1)}\|_1\,\cdots\,\|q^{(d)}\|_1)=\Phi(1\,\cdots\,1)=1$, since $Q\rgeq(\|q^{(1)}\|_1\,\cdots\,\|q^{(d)}\|_1)$ and the latter one-row matrix has the same support structure as $(1\,\cdots\,1)$. $P$ has at least one row not containing any $0$. Assume WLOG it is the first row. Then $P$ majorizes and is majorized by 
		\begin{equation}
			\left(\begin{array}{ccc}
				p_1^{(1)}/2&\cdots&p_1^{(d)}/2\\
				p_1^{(1)}/2&\cdots&p_1^{(d)}/2\\
				p_2^{(1)}&\cdots&p_2^{(d)}\\
				\vdots&\vdots&\vdots\\
				p_n^{(1)}&\cdots&p_n^{(d)}
			\end{array}\right), 
		\end{equation}
		hence 
        \begin{equation}
            \Phi(P)=\max\left(\underbrace{\Phi\left(p_1^{(1)}/2\;\cdots\;p_1^{(d)}/2\right)}_{=1},\underbrace{\Phi\left(\begin{array}{ccc}
				p_1^{(1)}/2&\cdots&p_1^{(d)}/2\\
				p_2^{(1)}&\cdots&p_2^{(d)}\\
				\vdots&\vdots&\vdots\\
				p_n^{(1)}&\cdots&p_n^{(d)}
			\end{array}\right)}_{\leq1}\right)=1. 
        \end{equation}
        Therefore, there are no non-degenerate monotone homomorphisms for $\mathbb{K}=\T\R_+^{\rm op}$. 
	\end{proof}

\begin{remark}
Note that, compared to the minimally restricted semiring $S^d$, some elements in $\mathcal{H}^{\rm n.d.}(S^d_{\rm d.c.},\R_+^{\rm op})$ coincide. Since the final column dominates all the other columns in all matrices in $\mc V^d_{\rm d.c.}$, it immediately follows that, whenever $\underline{\alpha}=(\alpha_1,\ldots,\alpha_{d-1},0)\in A_+$ and ${\rm supp}\,\underline{\alpha}\subseteq C\subseteq[d]$, we have $\Phi_{\underline{\alpha},C}=\Phi_{\underline{\alpha},C\cup\{d\}}$. See also Figure \ref{fig:DominatingColumn} for a depiction of the monotone homomorphisms in the case $d=3$. The figure also shows the derivations which we study in the next subsection.
\end{remark}

	\subsection{Monotone derivations}
	

    Recall that the degenerate monotone homomorphisms are all $\Phi^{(k)}$ for $k\in[d]$. Note that we only need to consider monotone derivations of these homomorphisms up to interchangeability. Similarly as for the minimal restrictions semiring, interchangeability allows us to show a result analogous to Lemma \ref{lem:interchangeability}, which we will use to characterize all monotone derivations, up to interchangeability. 
	
	\begin{proposition}
		\label{pro:derivationsdominatingcolumn}
		For each $k\in[d-1]$, any monotone derivation $\Delta$ of the degenerate homomorphism $\Phi^{(k)}$, up to interchangeability, satisfies $\Delta(P)=\gamma D_1(p^{(k)}\|p^{(d)})$ for all $P\in S^d_{\rm d.c}$, for some $\gamma\geq0$. The only monotone derivation $\Delta$ of the degenerate homomorphism $\Phi^{(d)}$, up to interchangeability, vanishes, i.e.,\ $\Delta(P)=0$ for all $P\in S^d_{\rm d.c}$. 
	\end{proposition}

    Let $\Delta:S^d_{\rm d.c}\to\mathbb{\R}$ be a monotone $\Phi^{(k)}$-derivation for some $k\in[d]$. For any $B\subsetneq[d]$, $d\in B$, we write 
    \begin{equation}
        \Delta[B]:=\Delta\left(\begin{array}{ccc}
                1&\cdots&1\\
                &B&
            \end{array}\right). 
    \end{equation}
Before going on to the proof, note that, whenever $B\subseteq[d]$ is such that $d\in B$, then $\Delta[B]=0$. The proof is essentially the same as that of Lemma \ref{lem:derivationszero} and will not be repeated here.
    

\begin{proof}[Proof of Proposition \ref{pro:derivationsdominatingcolumn}]
Let $\Delta:S_{\rm d.c.}^d\to\mathbb{\R}$ be a monotone $\Phi^{(k)}$-derivation for some $k\in[d]$. For any $x\in\R_+^d$ with $d\in{\rm supp}\,x$ we find using the above note and the Leibniz rule for $\Delta$,
\begin{align}
\Delta\big(\tilde{x}^T\boxplus x^T\big)&=\Delta\left(\tilde{x}^T\boxtimes\left(\begin{array}{ccc}
1&\cdots&1\\
&{\rm supp}\,x&
\end{array}\right)\right)\\
&=\Delta(\tilde{x}^T)\Phi^{(k)}[{\rm supp}\,x]+\Phi^{(k)}(\tilde{x}^T)\underbrace{\Delta[{\rm supp}\,x]}_{=0}=\left\{\begin{array}{ll}
2\Delta(\tilde{x}^T)&{\rm if}\ k\in{\rm supp}\,x\\
\Delta(\tilde{x}^T)&{\rm otherwise}
\end{array}\right.. 
\end{align}
Through a similar calculation as in \eqref{eq:Phihajotus}, we now find that
\begin{equation}
\Delta(P)=\sum_{i\in{\rm supp}\,p^{(k)}}\Delta(\widetilde{p_i}) 
\end{equation}
for any $P=(p^{(1)},\ldots,p^{(d)})\in S_{\rm d.c.}^d$ with rows $p_i$. Again, $\widetilde{p_i}$ is the row $p_i$ where the vanishing entries are replaced with 1. It follows that $\Delta(P)$ can be seen as $\Delta$ evaluated on a matrix with non-vanishing entries, more specifically the matrix resulting from taking only the rows in $P$ corresponding to ${\rm supp}\,p^{(k)}$ and replacing all vanishing entries with $1$. Hence, we can use the results in the proof of Proposition 14 in \cite{farooq2024} for the equal-support case. Combining Equations (111) and (117) in \cite{farooq2024}, we find 
\begin{equation}
\Delta(P)=\sum_{\ell\neq k}\gamma_\ell\sum_{i\in{\rm supp}\,p^{(k)}}p_i^{(k)}\log\left(\frac{p_i^{(k)}}{\widetilde{p_i^{(\ell)}}}\right) 
\end{equation}
for some constants $\gamma_\ell\geq0$, $\ell\in[d]\setminus\{k\}$. 
        
        Next, we show that $\gamma_\ell=0$ if $\ell\neq d$. Let $\ell\in[d-1]\setminus\{k\}$. Consider the following matrices in $S_{\rm d.c.}^d$ 
        \begin{align}
            P:=\left(\begin{array}{ccccccc}
                1/4&\cdots&1/4&1&1/4&\cdots&1/4\\
                3/4&\cdots&3/4&0&3/4&\cdots&3/4
            \end{array}\right),\\
            Q:=\left(\begin{array}{ccccccc}
                1/2&\cdots&1/2&1&1/2&\cdots&1/2\\
                1/2&\cdots&1/2&0&1/2&\cdots&1/2
            \end{array}\right),
        \end{align}
        where the column containing a $0$ is the $\ell$-th column. Notice that $P\rgeq Q$. A simple calculation shows that 
        \begin{equation}
            \Delta(P)=\gamma_\ell\left(2-\frac{3}{4}\log3\right)\leq\gamma_\ell=\Delta(Q), 
        \end{equation}
        with strict inequality if $\gamma_\ell>0$. This would imply that $\Delta$ is not monotone. Hence, we find that $\gamma_\ell=0$ for all $\ell\in[d-1]\setminus\{k\}$. Therefore, if $k\neq d$, 
        \begin{equation}
            \Delta(P)=\gamma_d\sum_{i\in{\rm supp}\,p^{(k)}}p_i^{(k)}\log\left(\frac{p_i^{(k)}}{\widetilde{p_i^{(d)}}}\right)=\gamma_d D_1\left(p^{(k)}\|p^{(d)}\right), 
        \end{equation}
        where we used that ${\rm supp}\,p^{(k)}\subseteq {\rm supp}\,p^{(d)}$ to identify the sum as the Kullback-Leibler divergence $D_1$. Note that the monotonicity of $\Delta$ follows from the data-processing inequality satisfied by $D_1$. For the case $k=d$, we have $\gamma_\ell=0$ for all $\ell\neq d$, hence $\Delta(P)=0$ for every $P\in S_{\rm d.c.}^d$. 		
	\end{proof}

\subsection{Power universal elements}\label{subsec:poweruniversaldc}

The following Proposition characterizes the power universal elements in $S_{\rm d.c.}^d$. 

\begin{proposition}\label{pro:poweruniversaldcgeneral}
    Let $U=(u^{(1)},\ldots,u^{(d)})\in S_{\rm d.c.}^d$ with all $u^{(k)}$ for $k\in[d]$, probability distributions. Then the following are equivalent: 
    \begin{enumerate}[(i)]
        \item $U$ is power universal. 
        \item It holds that ${\rm supp}\,u^{(k)}\subsetneq{\rm supp}\,u^{(d)}$ for all $k\in[d-1]$, and ${\rm supp}\,u^{(k)}\nsubseteq{\rm supp}\,u^{(k^\prime)}$ for all $k,k^\prime\in[d-1]$, $k\neq k^\prime$. 
        \item Any nondegenerate homomorphism $\Phi\in\mathcal{H}^{\rm n.d.}(S^d_{\rm d.c.},\R_+)\cup\mathcal{H}^{\rm n.d.}(S^d_{\rm d.c.},\R_+^{\rm op})\cup\mathcal{H}^{\rm n.d.}(S^d_{\rm d.c.},\T\R_+)$ satisfies $\Phi(U)\neq1$. 
    \end{enumerate}
\end{proposition}

\begin{proof}
    (iii) $\Rightarrow$ (ii): Note that for a monotone homomorphism $\Phi:S^d_{\rm d.c.}\to\mathbb{\R^{\rm op}}$, $\Phi(U)\neq1$ implies $\Phi(U)<1$. Since $\Phi_{e_d,\{k,d\}}(U)<1$ for $k\in[d-1]$, there is a row with a vanishing entry at $k$ and a positive value at $d$. Similarly, since $\Phi_{e_k,\{k,k^\prime,d\}}(U)<1$ for $k,k^\prime\in[d-1]$, $k\neq k^\prime$, there is a row with a vanishing entry at $k^\prime$ and a positive value at $k$. Hence, the conditions in (ii) are satisfied.  
    
    (ii) $\Rightarrow$ (i): From (ii), it follows that for each $k=2,\ldots,d-1$, $U$ has a row with positive entries at positions $1$ and $d$, and $0$ at position $k$. Then $U^{\boxtimes (d-2)}$ will contain the $\boxtimes$-product of these $d-2$ rows, which is a row with positive entries at positions $1$ and $d$, and $0$ everywhere else. Similarly, for the other columns $k^\prime=2,\ldots,d-2$, $U^{\boxtimes (d-2)}$ will contain a row with positive entries at positions $k^\prime$ and $d$ and $0$ everywhere else. Note that $U^{\boxtimes (d-1)}=U^{\boxtimes (d-2)}\boxtimes U$ will also contain rows of these specific types. From (ii) it also follows that for each $k=1,\ldots,d-1$, $U$ contains a row with a positive entry at position $d$ and $0$ at position $k$. Then $U^{\boxtimes (d-1)}$ will contain the $\boxtimes$-product of these $d-1$ rows, which is a row with a positive entry at position $d$ and $0$ everywhere else. One sees that it is possible to sum the rows in $U^{\boxtimes (d-1)}$ in such a way that 
    \begin{equation}
        \label{eq:almostidentity}
        U^{\boxtimes (d-1)}\rgeq\left(\begin{array}{cccc}
            a_1&\cdots&a_{d-1}&b_0\\
            1-a_1&\cdots&0&b_1\\
            0&\ddots&\vdots&\vdots\\
            0&\cdots&1-a_{d-1}&b_{d-1}\\
            0&\cdots&0&1-b			
        \end{array}\right):=\tilde{U}, 
    \end{equation}
    where $b:=\sum_{i=0}^{d-1}b_i$, and $a_k,b_i,b\in(0,1)$ for $k=1,\ldots,d-1$ and $i=0,\ldots,d-1$. 
    Define the stochastic map (matrix) $T$ mapping probability vectors on $\{0,1,\ldots,d\}^2$ into probability vectors on $\{0,1.\ldots,d\}$ through
    \begin{equation}
    (Tp)_m=p_{m,m}+\sum_{j=0}^{m-1}\big(p_{j,m}+p_{m,j}\big),\qquad m=0,\ldots,d,
    \end{equation}
    for all $p=(p_{j,k})_{i,j=0}^d$ where empty sums $\sum_{j=0}^{-1}(\cdots)$ vanish. Also define matrices $U_\ell$ through $U_1=\tilde{U}$ and $U_{\ell+1}=T(\tilde{U}\boxtimes U_\ell)$ for all $\ell=1,2,\ldots$. One easily sees that $\tilde{U}^{\boxtimes\ell}\rgeq U_\ell$ and a straightforward calculation combined with an induction argument shows that, for all $\ell\in\N$, there are $b^{(\ell)}_i$, $i=0,\ldots,d-1$ such that $b^{(\ell)}_0+\cdots+b^{(\ell)}_{d-1}=b^\ell$ and
    \begin{equation}\label{eq:majorizedbyU}
        U_\ell= \left(\begin{array}{cccc}
            a_1^\ell&\cdots&a_{d-1}^\ell&b_0^{(\ell)}\\
            1-a_1^\ell&\cdots&0&b_1^{(\ell)}\\
            0&\ddots&\vdots&\vdots\\
            0&\cdots&1-a_{d-1}^\ell&b_{d-1}^{(\ell)}\\
            0&\cdots&0&1-b^\ell		
        \end{array}\right).
    \end{equation}
    Hence, by choosing $\ell$ sufficiently large, we can make $a_k^\ell$ for $k=1,\ldots,d-1$ and $b_i^{(\ell)}$ for $i=0,\ldots, d-1$ arbitrarily small. 
    
Let $P=\big(p^{(1)},\ldots,p^{(d)}\big)\in \mc V_{\rm d.c.}^d$ with columns of unit $1$-norm, where $p^{(k)}=\big(p^{(k)}_1,\ldots,p^{(k)}_n\big)$. We may assume WLOG that $p^{(k)}_1>0$ for $k=1,\ldots,d$. Let us write $p_1:=\min_{1\leq k\leq d-1}p^{(k)}_1>0$, $t^{(0)}:=(1,0,\ldots,0)$ (with $n$ entries), and 
\begin{equation}
t^{(k)}_v:=\frac{1}{1-v}\big(p^{(k)}-vt^{(0)}\big)
\end{equation}
for $v\in(0,p_1)$ and $k=1,\ldots,d-1$. For any $v=(v_1,\ldots,v_{d-1})\in(0,p_1)^{d-1}$ and $w=(w_0,\ldots,w_{d-1})\in\R_{>0}^d$ such that $w_0+\cdots+w_{d-1}<1$, let us also define
\begin{equation}
t^{(d)}_{v,w}:=\frac{1}{1-w_0-\cdots-w_{d-1}}\big(p^{(d)}-w_0 t^{(0)}-w_1 t^{(1)}_{v_1}-\cdots-w_{d-1}t^{(d-1)}_{v_{d-1}}\big).
\end{equation}
A sufficient condition for $t^{(d)}_{v,w}$ to be a probability vector is
\begin{align}
p^{(d)}_1-w_0-\sum_{k=1}^{d-1}w_k\frac{p^{(k)}_1}{1-p_1}&\geq0,\\
p^{(d)}_i-\sum_{k=1}^{d-1}w_k\frac{p^{(k)}_i}{1-p_1}&\geq0,\qquad i=2,\ldots,n,
\end{align}
independent of $v_1,\ldots,v_{d-1}$. Thus, there is $w=(w_0,\ldots,w_{d-1})\in\R_{>0}^{d}$ such that $t^{(d)}_{v,w'}$ is a probability vector for all $v\in(0,p_1)^{d-1}$ and $w'\in\R_{>0}^d$ whose entries are upper-bounded by those of $w$. Suppose then that $\ell\in\N$ is large enough so that $a_k^\ell<p_1$ for $k=1,\ldots,d-1$ and $b^\ell<w_i$ (so that $b^{(\ell)}_i\leq b^\ell<w_i$) for $i=0,\ldots,d-1$ where $a_k$, $b_i$, and $b^{(\ell)}_i$ are as in \eqref{eq:almostidentity} and \eqref{eq:majorizedbyU}. Setting $v:=\big(a_1^\ell,\ldots,a_{d-1}^\ell\big)$ and $w':=\big(b_0^{(\ell)},\ldots,b_{d-1}^{(\ell)}\big)$ and $t^{(k)}:=t^{(k)}_{v}$ for $k=1,\ldots,d-1$ and $t^{(d)}:=t^{(d)}_{v,w'}$, we may set up the stochastic matrix $T=\big(t^{(0)},\ldots,t^{(d)}\big)$. It follows that we may write $TU_\ell=P$, i.e.,\ $U^{\boxtimes\ell(d-1)}\rgeq\tilde{U}^{\boxtimes\ell}\rgeq U_\ell\rgeq P$.

We now have shown that, for all $P\in\mc V^d_{\rm d.c.}$ with columns of unit $1$-norm, there is $\ell\in\N$ such that $U^{\boxtimes\ell}\rgeq P$. Let us assume next that $P,Q\in\mc V_{\rm d.c.}^d$ are such that $P\rgeq Q$. WLOG we may assume that the columns of $P$ and $Q$ have unit 1-norm. Thus, there is $\ell\in\N$ such that
\begin{equation}
Q\boxtimes U^{\boxtimes\ell}\rgeq(1\,\cdots\,1)\boxtimes U^{\boxtimes\ell}=U^{\boxtimes\ell}\rgeq P,
\end{equation}
showing that $U$ is a power universal.

(i) $\Rightarrow$ (iii): Let $P\in\mc V^d_{\rm d.c.}$ be such that all the columns of $P$ have unit 1-norm and $\Phi(P)\neq 1$ for all the nondegenerate monotone homomorphisms of $S^d_{\rm d.c.}$ into $\mb K\in\{\R_+,\R_+^{\rm op},\T\R_+,\T\R_+^{\rm op}\}$. We may choose, e.g.,
\begin{equation}
P=\left(\begin{array}{cccc}
1/2&\cdots&1/2&1/d\\
1/2&\cdots&0&1/d\\
\vdots&\ddots&\vdots&\vdots\\
0&\cdots&1/2&1/d
\end{array}\right), 
\end{equation}
for which $\Phi(P)>1$ when $\Phi\in\mc H^{\rm n.d.}(S^d_{\rm d.c.},\R_+)\cup\mc H^{\rm d.c.}(S^d_{\rm d.c.},\T\R_+)$ and $\Phi(P)<1$ otherwise. According to (i), there is $\ell\in\N$ such that $U^{\boxtimes\ell}\rgeq P$, so that $\Phi(U)^\ell\geq\Phi(P)>1$ for $\Phi\in\mc H^{\rm n.d.}(S^d_{\rm d.c.},\R_+)\cup\mc H^{\rm d.c.}(S^d_{\rm d.c.},\T\R_+)$ and $\Phi(U)^\ell\leq\Phi(P)<1$ otherwise. All in all, $\Phi(U)\neq1$ for all the monotone homomorphisms.

\end{proof}

\begin{remark}
    Note that for the implication (iii) $\Rightarrow$ (ii) above, we do not need all non-degenerate monotone homomorphisms to be unequal to $1$. It is sufficient that this holds for the homomorphisms $\Phi_{e_d,\{k,d\}}$ for $k\in[d-1]$, and $\Phi_{e_k,\{k,k^\prime,d\}}$ for $k,k^\prime\in[d-1]$, $k\neq k^\prime$. 
\end{remark}

\begin{remark}
Note that, according to item (ii) of Proposition \ref{pro:poweruniversaldcgeneral} and item (iii) of Proposition \ref{lem:poweruniversalconditions}, $U=\big(u^{(1)},\ldots,u^{(d)}\big)$ is a power universal of $S^d_{\rm d.c.}$ if and only if $\big(u^{(1)},\ldots,u^{(d-1)}\big)$ is a power universal of $S^{d-1}$ (n.b.,\ {\it not} of $S^{d-1}_{\rm d.c.}$) and ${\rm supp}\,u^{(k)}\subsetneq{\rm supp}\,u^{(d)}$ for $k=1,\ldots,d-1$. 

For $d=2$, $U=\big(u^{(1)},u^{(2)}\big)$ is a power universal of $S^d_{\rm d.c.}$ if and only if ${\rm supp}\,u^{(1)}\subsetneq{\rm supp}\,u^{(2)}$. For $d=3$, the general form of a power universal, up to the $\approx$-equivalence, is
\begin{equation}
U=\left(\begin{array}{ccc}
*&*&*\\
*&0&*\\
0&*&*
\end{array}\right), 
\end{equation}
where the asterisks represent positive entries (or vectors with positive entries).
\end{remark}

\subsection{Results: applying the Vergleichsstellensatz}

Applying the Vergleichsstellensatz in the form of Theorem \ref{thm:Fritz2022}, we can use the results from the previous sections to find the conditions for large sample and catalytic majorization in the semiring $S^d_{\rm d.c.}$, as given in the following Theorem. Note that just like the analogous result for the semiring $S^d$, the strict inequalities in \eqref{eq:majorizationconditionsdominatingcolumn1} and \eqref{eq:majorizationconditionsdominatingcolumn2} already imply that $P$ is power universal by Proposition \ref{pro:poweruniversaldcgeneral}, hence we do not need to require this explicitly in the conditions of the Theorem. 

The conditions on the monotone homomorphisms and derivations for large-sample and catalytic majorization of a matrix $P\in \V^d$ over $Q\in \V^d$ can be expressed as 
\begin{align}
    \label{eq:majorizationconditionsdominatingcolumn1}
      &\Phi_{\underline{\alpha},C}(P)<\Phi_{\underline{\alpha},C}(Q)&\forall\; \Phi_{\underline{\alpha},C}\in\mathcal{H}^{\rm n.d.}(S^d_{\rm d.c.},\R_+^{\rm op}),\\
    \label{eq:majorizationconditionsdominatingcolumn2}
    &\Phi_{\alpha,c}(P)>\Phi_{\alpha,c}(Q)&\forall\; \Phi_{\alpha,c}\in\mathcal{H}^{\rm n.d.}(S^d_{\rm d.c.},\R_+),\\
    \label{eq:majorizationconditionsdominatingcolumn3}
    &D_1(p^{(k)}\|p^{(d)})>D_1(q^{(k)}\|q^{(d)}),\,D_\infty(p^{(k)}\|p^{(d)})>D_\infty(q^{(k)}\|q^{(d)})&\forall k=1,\ldots,d-1. 
\end{align}
Recall that $D_1$ is the Kullback-Leibler divergence and $D_\infty$ the max-divergence between two columns. The homomorphisms $\Phi_{\underline{\alpha},C}\in\mathcal{H}^{\rm n.d.}(S^d_{\rm d.c.},\R_+^{\rm op})$ are exactly as in \eqref{eq:PhiR+op}, where we again exclude all degenerate homomorphisms $\Phi^{(k)}$, $k\in[d]$ (i.e. those with either $\underline{\alpha}=e_k$ and $C=\{k\}$ for some $k\in[d]$, or $\underline{\alpha}=e_k$ and $C=\{k,d\}$ for some $k\in[d-1]$). The homomorphisms $\Phi_{\alpha,c}\in\mathcal{H}^{\rm n.d.}(S^d_{\rm d.c.},\R_+)$ were defined in \eqref{eq:homdcR+} and can be conveniently written as 
\begin{equation}
    \Phi_{\alpha,c}(P)=\sum_{i=1}^n\big(p^{(c)}\big)^\alpha\big(p^{(d)}\big)^{1-\alpha}, 
\end{equation}
where $c\in[d-1]$ and $\alpha>1$. 

\begin{theorem}\label{thm:majorizationdominatingcolumn}
		Let $P=(p^{(1)},\ldots,p^{(d)})$, $Q=(q^{(1)},\ldots,q^{(d)})$ be tuples of probability vectors such that $\bigcap_{k=1}^d{\rm supp}\,p^{(k)}\neq\emptyset$ and ${\rm supp}\,p^{(k)}\subseteq{\rm supp}\,p^{(d)}$ for $k\in[d-1]$, and similarly for $Q$. If the inequalities in \eqref{eq:majorizationconditionsdominatingcolumn1}, \eqref{eq:majorizationconditionsdominatingcolumn2} and \eqref{eq:majorizationconditionsdominatingcolumn3} hold, then we have the following: 
		\begin{enumerate}[(a)]
			\item There exist a stochastic map $T$ and a tuple of probability vectors $(r^{(1)},\ldots,r^{(d)})$, satisfying $\bigcap_{k=1}^d{\rm supp}\,r^{(k)}\neq\emptyset$ and ${\rm supp}\,r^{(k)}\subseteq{\rm supp}\,r^{(d)}$ for $k\in[d-1]$, such that
			\begin{equation}
				T(r^{(k)}\otimes p^{(k)})=r^{(k)}\otimes q^{(k)}\quad\text{for }k\in[d], 
			\end{equation}
			and we may choose, for $n$ sufficiently large, 
			\begin{equation}
				r^{(k)}=\frac{1}{n+1}\bigoplus_{\ell=0}^n\left(q^{(k)}\right)^{\otimes \ell}\otimes\left(p^{(k)}\right)^{\otimes(n-\ell)}\quad\text{ for }k\in[d]. 
			\end{equation}
			\item For $n$ sufficiently large there exists a stochastic map $T_n$ such that 
			\begin{equation}
				T_n(p^{(k)})^{\otimes n}=(q^{(k)})^{\otimes n}\quad\text{for }k\in[d]. 
			\end{equation}
		\end{enumerate}
		Conversely, if either one of these holds, then the inequalities \eqref{eq:majorizationconditionsdominatingcolumn1}, \eqref{eq:majorizationconditionsdominatingcolumn2} and \eqref{eq:majorizationconditionsdominatingcolumn3} hold non-strictly. 
	\end{theorem}

    \begin{remark}
        We may derive an asymptotic version of this result, analogous to Theorem \ref{thm:genasymptocatalytic}. However, we would have to require the matrix $P$ to be power universal, which according to the characterization in (ii) of Proposition \ref{pro:poweruniversaldcgeneral} is quite restrictive. Nevertheless, for the case $d=2$, the situation is manageable, since there are essentially three interesting cases: either both columns have the same support and are different, or both columns have part of the support outside of the support of the other column, or one column has support strictly contained in the other. Majorization results (both for the exact and asymptotic setting) for the first case were derived in \cite{farooq2024}, and the second case is covered by the results in subsection \ref{subsec:applyingvssnorestrictions}. The third case is covered by Theorem \ref{thm:majorizationdominatingcolumn} (for the exact setting) above and by Theorem \ref{thm:asymptocatalytic} (for the asymptotic setting) in the next section, where we will derive an analogue of Theorem \ref{thm:genasymptocatalytic} for the case $d=2$. 
        
        Note that for $d>2$ there are many possibilities for the relationships of the supports of the columns, and each would require a different semiring to be studied. We will discuss this issue further in our concluding remarks. 
    \end{remark}

\subsection{The two-column case: dichotomies}\label{sec:2column}

Here we specify to the case $d=2$. In this context, any pair $\big(p^{(1)},p^{(2)}\big)$, where $p^{(1)}$ and $p^{(2)}$ are probability vectors (i.e.,\ $\big\|p^{(1)}\big\|_1=1=\big\|p^{(2)}\big\|_1$) and satisfy ${\rm supp}\,p^{(1)}\subseteq{\rm supp}\,p^{(2)}$, is called a \emph{dichotomy}. This special case will turn out to be particularly interesting in the theory of thermal majorization in the next section. We derive sufficient conditions for exact large-sample and catalytic majorization in this setting and sufficient and necessary conditions for approximate large-sample and catalytic majorization with asymptotically vanishing error in Corollary \ref{cor:dichotomyexact} and Theorem \ref{thm:asymptocatalytic}. We describe the conditions in the form of inequalities involving the traditional R\'{e}nyi divergences. We also complete the picture of asymptotic majorization in large samples and catalytically in Corollary \ref{cor:asymptocatalytic} to obtain a full characterization without having to make any support restrictions.

For all dichotomies $\big(p^{(1)},p^{(2)}\big)$, we have 
\begin{equation}
\frac{1}{\alpha-1}\log{\Phi_{(\alpha,1-\alpha),\{1,2\}}\big(p^{(1)}\|p^{(2)}\big)}=D_\alpha\big(p^{(1)}\big\|p^{(2)}\big), 
\end{equation}
when $\alpha\in[0,1)$, corresponding to the homomorphisms in $\mathcal{H}^{\rm n.d.}(S^d_{\rm d.c.},\R_+^{\rm op})$, and 
\begin{equation}
\frac{1}{\alpha-1}\log{\Phi_{\alpha,1}\big(p^{(1)}\|p^{(2)}\big)}=D_\alpha\big(p^{(1)}\big\|p^{(2)}\big), 
\end{equation}
when $\alpha\in(1,\infty)$, corresponding to the homomorphisms in $\mathcal{H}^{\rm n.d.}(S^d_{\rm d.c.},\R_+)$. 
We have the two degenerate homomorphisms $\Phi^{(1)}=\Phi_{(1,0),\{1\}}=\Phi_{(1,0),\{1,2\}}$ and $\Phi^{(2)}=\Phi_{(0,1),\{2\}}$. Up to a positive multiplier (and interchangeability), we have only one monotone derivation at $\Phi^{(1)}$, which is simply the Kullback-Leibler divergence, and none at $\Phi^{(2)}$. We also have essentially only one nondegenerate tropical homomorphism $\Phi^\T_1$ and
\begin{equation}
\log{\Phi^\T_1\big(p^{(1)},p^{(2)}\big)}=D_\infty\big(p^{(1)}\big\|p^{(2)}\big)
\end{equation}
for all dichotomies $\big(p^{(1)},p^{(2)}\big)$. Thus, we may unify the nondegenerate monotone homomorphisms and derivations of $S^2_{\rm d.c.}$ within the family $D_\alpha$, $0\leq\alpha\leq\infty$, of R\'{e}nyi divergences; see Figure \ref{fig:2d} for a depiction. The monotonicity of the homomorphisms and derivation is now encapsulated by the data processing inequality
\begin{equation}
D_\alpha\big(p^{(1)}\big\|p^{(2)}\big)\geq D_\alpha\big(Tp^{(1)}\big\|Tp^{(2)}\big),\qquad\forall\alpha\in[0,\infty],
\end{equation}
for all dichotomies $\big(p^{(1)},p^{(2)}\big)$ and stochastic matrices $T$.


We reformulate Theorem \ref{thm:majorizationdominatingcolumn} specified to the case of dichotomies and in terms of inequalities of R\'enyi divergences. 

\begin{corollary}\label{cor:dichotomyexact}
Let $\big(p^{(1)},p^{(2)}\big)$ and $\big(q^{(1)},q^{(2)}\big)$ be dichotomies. If
\begin{equation}\label{eq:RenyiOrdered}
D_\alpha\big(p^{(1)}\big\|p^{(2)}\big)>D_\alpha\big(q^{(1)}\big\|q^{(2)}\big),\qquad\forall\alpha\in[0,\infty],
\end{equation}
then we have the following:
\begin{itemize}
\item[(a)] There is a stochastic matrix $T$ and a dichotomy $\big(r^{(1)},r^{(2)}\big)$ such that
\begin{equation}
T\big(r^{(k)}\otimes p^{(k)}\big)=r^{(k)}\otimes q^{(k)},\qquad k=1,2.
\end{equation}
There is $n\in\N$ such that we may choose
\begin{equation}\label{eq:BinomialCatalyst}
r^{(k)}=\frac{1}{n+1}\bigoplus_{\ell=0}^n \big(q^{(k)}\big)^{\otimes\ell}\otimes\big(p^{(k)}\big)^{\otimes(n-\ell)},\quad k=1,2.
\end{equation}
\item[(b)] For any $n\in\N$ sufficiently large, there is a stochastic matrix $T_n$ such that
\begin{equation}
T_n\big(p^{(k)}\big)^{\otimes n}=\big(q^{(k)}\big)^{\otimes n},\qquad k=1,2.
\end{equation}
\end{itemize}
Conversely, if item (a) or (b) above holds, then \eqref{eq:RenyiOrdered} holds with non-strict inequalities.
\end{corollary}

\begin{figure}
\begin{center}
\begin{overpic}[scale=0.65,unit=1mm]{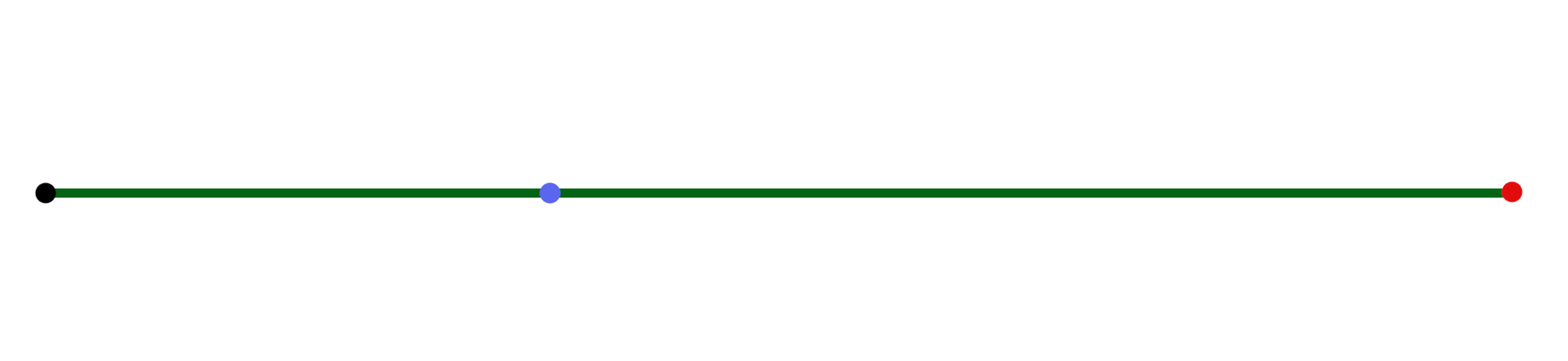}
\put(0,9){\begin{large}
$\Phi^{(2)}$
\end{large}}
\put(49,9){\begin{large}
$\Phi^{(1)}$
\end{large}}
\put(0,20){\begin{Large}
\darker{$D_0$}
\end{Large}}
\put(0,28){
\darker{$\Phi_{e_2,\{1,2\}}$}
}
\put(18,9){\begin{large}
\darker{$0<\alpha<1$}
\end{large}}
\put(18,20){\begin{Large}
\darker{$D_\alpha$}
\end{Large}}
\put(18,28){
\darker{$\Phi_{(\alpha,1-\alpha),\{1,2\}}$}
}
\put(90,9){\begin{large}
\darker{$1<\alpha<\infty$}
\end{large}}
\put(90,20){\begin{Large}
\darker{$D_\alpha$}
\end{Large}}
\put(90,28){
\darker{$\Phi_{(\alpha,1-\alpha),\{1,2\}}$}
}
\put(49,20){\begin{Large}
\sininen{$D_1$}
\end{Large}}
\put(135,9){\begin{large}
\punainen{$\alpha=\infty$}
\end{large}}
\put(140,20){\begin{Large}
\punainen{$D_\infty$}
\end{Large}}
\put(140,28){
\punainen{$\Phi^\T_1$}
}
\end{overpic}
\caption{\label{fig:2d} The monotone homomorphisms and derivations depicted for $S^2_{\rm d.c.}$: The line is the compactification of the half line $[0,\infty]$. Nondegenerate temperate homomorphisms are in green in the area $\alpha\in[0,1)\cup(1,\infty)$. These correspond to the R\'{e}nyi divergences $D_\alpha$. At points $\alpha=0$ and $\alpha=1$ we also have the degenerate homomorphisms $\Phi^{(2)}$ and $\Phi^{(1)}$, the latter of which is associated with the non-zero derivation $D_1$, the Kullback-Leibler divergence. At the infinity (in red) we have the tropical homomorphism corresponding to the max-divergence $D_\infty$.}
\end{center}
\end{figure}

We may also derive an asymptotic version of the above result using largely the same methods as in the proof of Theorem \ref{thm:genasymptocatalytic}. However, since the second probability vector dominates the first one in a dichotomy, we may remove the vanishing error completely on the second probability vector as stated at the end of the following result.

\begin{theorem}\label{thm:asymptocatalytic}
Consider dichotomies $\big(p^{(1)},p^{(2)}\big)$ and $\big(q^{(1)},q^{(2)}\big)$ 
and assume ${\rm supp}\,p^{(1)}\subsetneq{\rm supp}\,p^{(2)}$. The following conditions are equivalent:
\begin{itemize}
\item[(i)] For all $\alpha\in(0,\infty]$,
\begin{equation}\label{eq:0alphainfty}
D_\alpha\big(p^{(1)}\big\|p^{(2)}\big)\geq D_\alpha\big(q^{(1)}\big\|q^{(2)}\big).
\end{equation}
\item[(ii)] For all $\varepsilon>0$, there is a stochastic matrix $T_\varepsilon$ and dichotomies $\big(q_\varepsilon^{(1)},q_\varepsilon^{(2)}\big)$ and $\big(r_\varepsilon^{(1)},r_\varepsilon^{(2)}\big)$ such that
\begin{equation}
\big\|q^{(k)}-q^{(k)}_\varepsilon\big\|_1\leq\varepsilon,\qquad k=1,2,
\end{equation}
and
\begin{equation}
T_\varepsilon\big(r_\varepsilon^{(k)}\otimes p^{(k)}\big)=r_\varepsilon^{(k)}\otimes q_\varepsilon^{(k)},\qquad k=1,2.
\end{equation}
Again, we may choose $r_\varepsilon^{(k)}$ as in \eqref{eq:BinomialCatalyst} (with $q_\varepsilon^{(k)}$ substituted for $q^{(k)}$) for sufficiently large $n\in\N$.
\item[(iii)] For all $\varepsilon>0$, there is a dichotomy $\big(q_\varepsilon^{(1)},q_\varepsilon^{(2)}\big)$ 
such that
\begin{equation}
\big\|q^{(k)}-q_\varepsilon^{(k)}\big\|_1\leq\varepsilon,\qquad k=1,2,
\end{equation}
and for $n\in\N$ sufficiently large there is a stochastic matrix $T_{\varepsilon,n}$ such that
\begin{equation}
T_{\varepsilon,n}\big(p^{(k)}\big)^{\otimes n}=\big(q_\varepsilon^{(k)}\big)^{\otimes n},\qquad k=1,2.
\end{equation}
\end{itemize}
In statements (ii) and (iii) above, we may assume that $q_\varepsilon^{(2)}=q^{(2)}$ for all $\varepsilon>0$.
\end{theorem}

\begin{proof}
Let us assume condition (i). Let us choose $w\in\mc P_n$ which dominates $q^{(2)}$ (and, hence, automatically $q^{(1)}$). Naturally, we may choose $w=q^{(2)}$; this choice yields the final claim of the statement of the theorem, as one clearly subsequently sees. Let us denote, for brevity, $\big(p^{(1)},p^{(2)}\big)=:P$, $\big(q^{(1)},q^{(2)}\big)=:Q$, and $W:=(w,w)$ and define, for all $\varepsilon\in(0,2]$,
\begin{equation}
q^{(k)}_\varepsilon:=\left(1-\frac{\varepsilon}{2}\right)q^{(1)}+\frac{\varepsilon}{2}w,\qquad k=1,2,
\end{equation}
and
\begin{equation}
Q_\varepsilon:=\big(q^{(1)}_\varepsilon,q^{(2)}_\varepsilon\big)=\left(1-\frac{\varepsilon}{2}\right)Q+\frac{\varepsilon}{2}W.
\end{equation}
Clearly, $Q_\varepsilon$ is a dichotomy with mutually supporting columns for all $\varepsilon\in(0,2]$. Moreover, $\big\|q^{(k)}-q^{(k)}_\varepsilon\big\|_1\leq\varepsilon$, as one easily verifies. We may show that $D_\alpha\big(p^{(1)}\big\|p^{(2)}\big)>D_\alpha\big(q^{(1)}_\varepsilon\big\|q^{(2)}_\varepsilon\big)$ for all $\alpha\in[0,1)$ in the same way as in the proof of Theorem \ref{thm:genasymptocatalytic}.

Using the joint convexity of the Kullback-Leibler divergence, we have, for all $\varepsilon\in(0,2]$,
\begin{equation}
D_1\big(q^{(1)}_\varepsilon\big\|q^{(2)}_\varepsilon\big)\leq \left(1-\frac{\varepsilon}{2}\right)D_1\big(q^{(1)}\big\|q^{(2)}\big)+\frac{\varepsilon}{2}\underbrace{D_1(w\|w)}_{=0}\leq D_1\big(q^{(1)}\big\|q^{(2)}\big)\leq D_1\big(p^{(1)}\big\|p^{(2)}\big).
\end{equation}
If $D_1\big(q^{(1)}\big\|q^{(2)}\big)=0$, then $D_1\big(q^{(1)}\big\|q^{(2)}\big)=0<D_1\big(p^{(1)}\big\|p^{(2)}\big)$ since $p^{(1)}\neq p^{(2)}$. If, on the other hand, $D_1\big(q^{(1)}\big\|q^{(2)}\big)>0$, then the penultimate inequality in the above calculation is strict. Thus, we have
\begin{equation}
D_1\big(q^{(1)}_\varepsilon\big\|q^{(2)}_\varepsilon\big)<D_1\big(p^{(1)}\big\|p^{(2)}\big),\qquad\forall\varepsilon\in(0,2].
\end{equation}

Let now $\alpha\in(1,\infty)$ and denote $\Phi_\alpha:=\Phi_{(\alpha,1-\alpha),\{1,2\}}$. Similarly as in the proof of Theorem \ref{thm:genasymptocatalytic}, we may evaluate, for all $\varepsilon\in(0,2]$,
\begin{align}
D_\alpha\big(q^{(1)}_\varepsilon\big\|q^{(2)}_\varepsilon\big)&\leq\frac{1}{\alpha-1}\log{\bigg(Q_\alpha(Q)-\frac{\varepsilon}{2}\underbrace{\big(\Phi_\alpha(Q)-1\big)}_{\geq0}\bigg)}\\
&\leq\frac{1}{\alpha-1}\log{\Phi_\alpha(Q)}=D_\alpha\big(q^{(1)}\big\|q^{(2)}\big)\leq D_\alpha\big(p^{(1)}\big\|p^{(2)}\big).
\end{align}
If $\Phi_\alpha(Q)=1$, then $\Phi_\alpha(Q)=1<\Phi_\alpha(P)$ and the final inequality above is strict. If $\Phi_\alpha(Q)>1$, then the penultimate inequality above is strict. Thus,
\begin{equation}
D_\alpha\big(q^{(1)}_\varepsilon\big\|q^{(2)}_\varepsilon\big)<D_\alpha\big(p^{(1)}\big\|p^{(2)}\big),\qquad\forall\varepsilon\in(0,2].
\end{equation}

Let us then look at the tropical homomorphism $\Phi^\T_1$. Using monotonicity, we have, for all $\varepsilon\in(0,2]$,
\begin{align}
\Phi^\T_1(Q_\varepsilon)&=\Phi^\T_1\left(\left(1-\frac{\varepsilon}{2}\right)Q+\frac{\varepsilon}{2}W\right)\leq\Phi^\T_1\left(\left(1-\frac{\varepsilon}{2}\right)Q\boxplus\frac{\varepsilon}{2}W\right)\\
&=\max\big\{\Phi^\T_1(Q),1\big\}=\Phi^\T_1(Q)\leq\Phi^\T_1(P).
\end{align}
Especially, $\Phi^\T_1(Q_\varepsilon)\leq\Phi^\T_1(Q)$, i.e.,
\begin{equation}
\max_{1\leq i\leq n}\frac{\left(1-\frac{\varepsilon}{2}\right)q^{(1)}_i+\frac{\varepsilon}{2}w_i}{\left(1-\frac{\varepsilon}{2}\right)q^{(2)}_i+\frac{\varepsilon}{2}w_i}\leq\max_{1\leq i\leq n}\frac{q^{(1)}_i}{q^{(2)}_i}
\end{equation}
for all $\varepsilon\in(0,2]$. When $0<\varepsilon\leq\varepsilon'\leq 2$, then $Q_\varepsilon\rgeq Q_{\varepsilon'}$, implying that the LHS above is non-increasing in $\varepsilon$ and coincides with the RHS when $\varepsilon\to0$. Thus, if we have equality with some $\varepsilon_0\in(0,2]$, then we have equality for all $\varepsilon\in[0,\varepsilon_0]$. Let us assume this. Since the LHS is a pointwise maximum of a finite set of real-analytic functions, the maximum is reached with the same index $i$ for all $\varepsilon\in[0,\varepsilon_0]$. The real-analyticity of the LHS now implies that the LHS is constant for all $\varepsilon\in[0,2]$, i.e.,\ $\Phi^\T_1(Q)=1=\Phi^\T_1(W)$. Thus, if $\Phi^\T_1(Q_\varepsilon)=\Phi^\T_1(Q)$ for some $\varepsilon\in(0,2]$, this is the case for all $\varepsilon\in[0,2]$ and $\Phi^\T_1(Q)=1<\Phi^\T_1(P)$. On the other hand, if $\Phi^\T_1(Q)>1$ then $\Phi^\T(Q_\varepsilon)<\Phi^\T(Q)$, so that $\Phi^\T_1(Q_\varepsilon)<\Phi^\T_1(Q)\leq\Phi^\T_1(P)$. All in all,
\begin{equation}
D_\infty\big(q^{(1)}_\varepsilon\big\|q^{(2)}_\varepsilon\big)<D_\infty\big(p^{(1)}\big\|p^{(2)}\big),\qquad\forall\varepsilon\in(0,2].
\end{equation}
Thus, the conditions of Corollary \ref{cor:dichotomyexact} are satisfied for $Q_\varepsilon$ and $P$, and conditions (ii) and (iii) of the claim follow.

Let us then assume that (ii) or (iii) holds. The proof that $D_\alpha\big(p^{(1)}\big\|p^{(2)}\big)\geq D_\alpha\big(q^{(1)}\big\|q^{(2)}\big)$ for $\alpha\in(0,1)$ follows from the same argument as in the end of the proof of Theorem \ref{thm:genasymptocatalytic}. To show the inequality for $\alpha\in(1,\infty)$, we use almost the same argument: consider any $D_\alpha$ with $\alpha\in(1,\infty)$, which corresponds to the monotone homomorphism $\Phi_{\alpha,1}$. Then we use similar notation and reasoning as in the proof of Theorem \ref{thm:genasymptocatalytic}, except that $\Phi=\Phi_{\alpha,1}$ and the direction of the inequalities is reversed, and we note that 
\begin{equation}
    \Phi_{\alpha,1}(Q_\varepsilon^1\boxplus(1\,\cdots\,1))\geq(1+\varepsilon)^{1-\alpha}, 
\end{equation}
where we used monotonicity of $\Phi_{\alpha,1}$ and that the sum of each column in $Q_\varepsilon^1\boxplus(1\,\cdots\,1)$ is between $1$ and $1+\varepsilon$. Finally, the inequalities $D_\alpha\big(p^{(1)}\big\|p^{(2)}\big)\geq D_\alpha\big(q^{(1)}\big\|q^{(2)}\big)$ for $\alpha\in\{1,\infty\}$ follow from the same inequalities for $\alpha\in(0,1)\cup(1,\infty)$ and the continuity of $D_\alpha$ in $\alpha$. 
\end{proof}

Combining the above result with Theorem 22 in \cite{farooq2024} gives immediately the following corollary.

\begin{corollary}\label{cor:asymptocatalytic}
Denote by $\mc P_n$ the set of probability vectors with $n$ components and consider probability vectors $p^{(2)},q^{(2)}\in\mc P_n$ of full support and some additional $p^{(1)},q^{(1)}\in\mc P_n$. Also assume that $p^{(1)}\neq p^{(2)}$. The following conditions are equivalent:
\begin{itemize}
\item[(i)] For all $\alpha\geq 1/2$,
\begin{equation}\label{eq:jointcond}
D_\alpha\big(p^{(1)}\big\|p^{(2)}\big)\geq D_\alpha\big(q^{(1)}\big\|q^{(2)}\big),\qquad D_\alpha\big(p^{(2)}\big\|p^{(1)}\big)\geq D_\alpha\big(q^{(2)}\big\|q^{(1)}\big).
\end{equation}
\item[(ii)] For all $\varepsilon>0$, there is a stochastic matrix $T_\varepsilon$ and dichotomies $\big(q_\varepsilon^{(1)},q_\varepsilon^{(2)}\big)$ and $\big(r_\varepsilon^{(1)},r_\varepsilon^{(2)}\big)$ such that
\begin{equation}
\big\|q^{(k)}-q^{(k)}_\varepsilon\big\|_1\leq\varepsilon,\qquad k=1,2,
\end{equation}
and
\begin{equation}
T_\varepsilon\big(r_\varepsilon^{(k)}\otimes p^{(k)}\big)=r_\varepsilon^{(k)}\otimes q_\varepsilon^{(k)},\qquad k=1,2.
\end{equation}
Again, we may choose $r_\varepsilon^{(k)}$ as in \eqref{eq:BinomialCatalyst} (with $q_\varepsilon^{(k)}$ substituted for $q^{(k)}$) for sufficiently large $n\in\N$.
\item[(iii)] For all $\varepsilon>0$, there is a dichotomy $\big(q_\varepsilon^{(1)},q_\varepsilon^{(2)}\big)$ with probability vectors in $\mc P_n$
such that
\begin{equation}
\big\|q^{(k)}-q_\varepsilon^{(k)}\big\|_1\leq\varepsilon,\qquad k=1,2,
\end{equation}
and for $n\in\N$ sufficiently large there is a stochastic matrix $T_{\varepsilon,n}$ such that
\begin{equation}
T_{\varepsilon,n}\big(p^{(k)}\big)^{\otimes n}=\big(q_\varepsilon^{(k)}\big)^{\otimes n},\qquad k=1,2.
\end{equation}
\end{itemize}
In statements (i) and (ii), we may assume that $q_\varepsilon^{(2)}=q^{(2)}$ for all $\varepsilon>0$.
\end{corollary}

\begin{proof}
Let us first assume that ${\rm supp}\,p^{(1)}={\rm supp}\,p^{(2)}=\{1,\ldots,n\}$ and ${\rm supp}\,q^{(1)}={\rm supp}\,q^{(2)}=\{1,\ldots,n\}$. In this case, the desired result is given by Theorem 22 in \cite{farooq2024}. Strictly speaking, only the equivalence of statements (i) and (ii) is shown there, but adding (iii) can be done in the same way that we have already employed earlier.

Let us now assume that ${\rm supp}\,p^{(1)}\subsetneq{\rm supp}\,p^{(2)}$ or ${\rm supp}\,q^{(1)}\subsetneq{\rm supp}\,q^{(2)}$. Let us first show that the latter case is out of the question if the first one does not hold. So, assume that ${\rm supp}\,p^{(1)}={\rm supp}\,p^{(2)}$ while ${\rm supp}\,q^{(1)}\subsetneq{\rm supp}\,q{(2)}$. If we assume (i), we obtain from the second set of inequalities that
\begin{equation}
1=\sum_{i\in{\rm supp}\,p^{(1)}}p^{(2)}_i\leq\sum_{i\in{\rm supp}\,q^{(1)}}q^{(2)}_i<1
\end{equation}
as $\alpha\to 1$, a contradiction. Thus, we only have to consider the case where ${\rm supp}\,p^{(1)}\subsetneq{\rm supp}\,p^{(2)}$. In this case, the second set of the inequalities in \eqref{eq:jointcond} are irrelevant when $\alpha\geq1$. Indeed, $D_\alpha\big(p^{(2)}\big\|p^{(1)}\big)=\infty$ when $\alpha\geq1$. Thus we may ignore this set of inequalities. For $0<\alpha\leq1/2$, we have $D_\alpha\big(p^{(1)}\big\|p^{(2)}\big)=\frac{1-\alpha}{\alpha}D_{1-\alpha}\big(p^{(2)}\big\|p^{(1)}\big)$. This means that, if (i) holds, then $D_\alpha\big(p^{(1)}\big\|p^{(2)}\big)\geq D_\alpha\big(q^{(1)}\big\|q^{(2)}\big)$ for all $\alpha>0$. Naturally this also holds for $\alpha=0$ as a pointwise limit of the above. Finally, condition (i) is equivalent with \eqref{eq:0alphainfty}, and we are in the situation of Theorem \ref{thm:asymptocatalytic} providing the proof for the remainder of the claim.
\end{proof}

\section{Some applications}

In this section, we discuss some applications and corollaries of the results we have derived in this work. The first immediate result deals with optimal rates. Our basic result can be formulated in the abstract setting of Theorem \ref{thm:Fritz2022}, but it immediately translates into results dealing with the optimal rates of transforming a given statistical experiment into another. Our second field of immediate applications is quantum thermodynamics and the characterization of asymptotic catalytic thermal majorization of quantum states commuting with the thermal (Gibbs) state. Together with the results of \cite{farooq2024}, we give a coherent proof for the statements made in \cite{Brandao2014}.

\subsection{Achievable rates}\label{rem:rates}
We now discuss the rate of transforming i.i.d.\ copies of an input to a given output. The rate is characterized by {\it divergences} derived from the monotone homomorphisms and derivations. In the case of the preceding subsection dealing with dichotomies with a dominating second distribution, these divergences are simply the R\'{e}nyi divergences $D_\alpha$ for $\alpha\in[0,\infty]$. In the general setting of Theorem \ref{thm:Fritz2022} with a preordered semiring $S$ and the degenerate homomorphism $\|\cdot\|:S\to\R_{>0}^d\cup\{0\}$, these `divergences' are defined as follows: We fix a power universal $u\in S$. For any nondegenerate homomorphism $\Phi:S\to\mb K$, $\mb K\in\{\R_+,\R_+^{\rm op},\T\R_+,\T\R_+^{\rm op}\}$, with trivial kernel, the associated divergence is
\begin{equation}
D(x)=\frac{\log{\Phi(x)}}{\log{\Phi(u)}},\qquad\forall x\in S\setminus\{0\}.
\end{equation}
For any $k\in\{1,\ldots,d\}$ and any derivation $\Delta$ at $\|\cdot\|_{(k)}$ normalized according to $\Delta(u)=1$, the corresponding divergence is
\begin{equation}
D(x)=\frac{\Delta(x)}{\|x\|_{(k)}},\qquad\forall x\in S\setminus\{0\}.
\end{equation}
We note that, whenever $x\succeq y$, then $D(x)\geq D(y)$ for all of these divergences. Moreover the `additivity' $D(xy)=D(x)+D(y)$ holds for any $x,y\in S\setminus\{0\}$. In the case of dichotomies with a dominating second column, the choice
\begin{equation}
u=\left(\begin{array}{cc}
1&1/2\\
0&1/2
\end{array}\right),
\end{equation}
gives us the familiar R\'{e}nyi divergences as defined in \eqref{eq:RenyiDivergences}. According to Proposition 8.5 of \cite{fritz2022}, the set $\mc D$ of these divergences is a compact Hausdorff space with respect to the coarsest topology where the comparison maps
\begin{equation}
\mc D\ni D\mapsto D(x)-D(y)\in\R,\qquad\forall x,y\in S\setminus\{0\},\ x\sim y,
\end{equation}
are continuous. Note that the conditions (i) and (ii) of Theorem \ref{thm:Fritz2022} can now be concisely reformulated as $D(x)>D(y)$ for all $D\in\mc D$.

Let us pick $x,y\in S\setminus\{0\}$ such that $x\sim y$ and $x$ is a power universal. Note that, since $x$ is a power universal and $1\succeq 1$, then there is $k\in\N$ such that $x^k\succeq 1$. Thus, $\|x\|^k=\|1\|=(1,\ldots,1)$ so that 
$\|x\|=\|1\|$ and $x\sim 1$. Hence, $1\sim x\sim y$. We say that $r\geq0$ is {\it $(x,y)$-achievable} if there exists a sequence $(m_n)_{n=1}^\infty\in\N^\N$ of natural numbers with $r\leq\liminf_{n\to\infty}m_n/n$ and $x^n\succeq y^{m_n}$ for any $n\in\N$ large enough. The {\it optimal achievable rate} is then
\begin{equation}
r(x\to y):=\sup\{r\geq0\,|\,r\ {\rm is}\ (x,y){\rm -achievable}\}.
\end{equation}

\begin{corollary}\label{cor:ratemin}
With the above assumptions and definitions,
\begin{equation}\label{eq:ratemin}
r(x\to y)=\min_{D\in\mc D}\frac{D(x)}{D(y)}.
\end{equation}
\end{corollary}

Note that, according to the compactness result referenced to above, the minimum on the right-hand side exists. Indeed, $D\mapsto D(x)=D(x)-D(1)$ (recall that $x\sim 1$) is continuous and so is $D\mapsto D(y)$.

\begin{proof}
We prove \eqref{eq:ratemin} by first showing that any achievable rate is upper bounded by $D(x)/D(y)$ for any additive and monotone divergence $D$ (i.e.,\ not only for $D\in\mc D$). This proof is completely standard and does not use the earlier results. This shows that $r(x\to y)\leq\min_{D\in\mc D}D(x)/D(y)$. Then we show that $r(x\to y)\geq\min_{D\in\mc D}D(x)/D(y)$ with a simple application of Theorem \ref{thm:Fritz2022}.

Suppose that $r\geq0$ is $(x,y)$-achievable, i.e.,\ there is a sequence $(m_n)_{n=1}^\infty\in\N^\N$ such that $r\leq\liminf_{n\to\infty}m_n/n$ and $x^n\succeq y^{m_n}$ for any $n\geq n_0$ where $n_0\in\N$ is some sufficiently large number. We fix, for now, a general additive and monotone divergence $D:S\setminus\{0\}\to\R$, i.e,\ $D(xy)=D(x)+D(y)$ for all $x,y\in S$ and $x\succeq y$ $\Rightarrow$ $D(x)\geq D(y)$. From the additivity and monotonicity it immediately follows that $m_n/n\leq D(x)/D(y)$ for all $n\geq n_0$. From the definition of $(x,y)$-achievability it follows that, for all $\varepsilon>0$, there is $n_\varepsilon\in\N$ which we can choose to be larger than $n_0$ such that $r-\varepsilon\leq \inf_{n\geq n_\varepsilon} m_n/n$. But as $n_\varepsilon\geq n_0$, this means that $r-\varepsilon\leq D(x)/D(y)$. As this holds for all $\varepsilon>0$, we have $r\leq D(x)/D(y)$. Taking the supremum over achievable rates and then the minimum over $D\in\mc D$, we get $r(x\to y)\leq\min_{D\in\mc D}D(x)/D(y)$.

Let us then assume that $r=p/q$ with $p,q\in\N$ and that $r<\min_{D\in\mc D}D(x)/D(y)$. We now have
\begin{equation}
D(x^q)=qD(x)=\frac{p}{r}D(x)>pD(x)\max_{D'\in\mc D}\frac{D'(y)}{D'(x)}\geq pD(y)=D(y^p),\qquad\forall D\in\mc D.
\end{equation}
Thus, according to (the reformulation of) Theorem \ref{thm:Fritz2022}, we know that $x^{qn}\succeq y^{pn}$ for any $n\in\N$ large enough. This means that $r=p/q$ is $(x,y)$-achievable, so that $r\leq r(x\to y)$. Since this holds for any rational $r<\min_{D\in\mc D}D(x)/D(y)$, we finally have $r(x\to y)\geq\min_{D\in\mc D}D(x)/D(y)$. Combining our observations, we arrive at \eqref{eq:ratemin}.
\end{proof}

Suppose now that $P=\big(p^{(1)},p^{(2)}\big)$ and $Q=\big(q^{(1)},q^{(2)}\big)$ are dichotomies such that ${\rm supp}\,p^{(1)}\subsetneq{\rm supp}\,p^{(2)}$ (i.e.,\ $P$ is a power universal of the $d=2$ dominating column semiring). The statement of \eqref{eq:ratemin} now takes the form
\begin{equation}
r(P\to Q)=\min_{\alpha\in[0,\infty]}\frac{D_\alpha\big(p^{(1)}\big\|p^{(2)}\big)}{D_\alpha\big(q^{(1)}\big\|q^{(2)}\big)}.
\end{equation}
Naturally, similar results also hold in the higher $d$ case, but these look quite complex. Let us see however what the optimal rate is in the case of dichotomies with full support studied as a special case in \cite{farooq2024}: Given dichotomies $P=\big(p^{(1)},p^{(2)}\big)$ and $Q=\big(q^{(1)},q^{(2)}\big)$ where ${\rm supp}\,p^{(1)}={\rm supp}\,p^{(2)}$, ${\rm supp}\,q^{(1)}={\rm supp}\,q^{(2)}$, and $p^{(1)}\neq p^{(2)}$, we have
\begin{equation}
r(P\to Q)=\min\left\{\frac{D_\alpha\big(p^{(1)}\big\|p^{(2)}\big)}{D_\alpha\big(q^{(1)}\big\|q^{(2)}\big)},\frac{D_\alpha\big(p^{(2)}\big\|p^{(1)}\big)}{D_\alpha\big(q^{(2)}\big\|q^{(1)}\big)}\,\middle|\,\alpha\in[1/2,\infty]\right\}.
\end{equation}
We note that in this case the minimization has to be performed on a larger set. In the $d=2$ case of the minimal restrictions semiring discussed in Section \ref{sec:findings}, we notice that the minimization task is smaller: If $P=\big(p^{(1)},p^{(2)}\big)$ and $Q=\big(q^{(1)},q^{(2)}\big)$ are pairs of probability distributions such that ${\rm supp}\,p^{(1)}\cap{\rm supp}\,p^{(2)}\neq\emptyset\neq{\rm supp}\,q^{(1)}\cap{\rm supp}\,q^{(2)}$, ${\rm supp}\,p^{(1)}\nsubseteq{\rm supp}\,p^{(2)}$, and ${\rm supp}\,p^{(2)}\nsubseteq{\rm supp}\,p^{(1)}$ (i.e.,\ $P$ is a power universal of the minimal restrictions semiring),
\begin{equation}
r(P\to Q)=\min\left\{\frac{D_\alpha\big(p^{(1)}\big\|p^{(2)}\big)}{D_\alpha\big(q^{(1)}\big\|q^{(2)}\big)},\frac{D_\alpha\big(p^{(2)}\big\|p^{(1)}\big)}{D_\alpha\big(q^{(2)}\big\|q^{(1)}\big)}\,\middle|\,\alpha\in[0,1/2]\right\}.
\end{equation}

\subsection{Thermal majorization}\label{sec:thermal}

Let us consider a finite-dimensional quantum thermodynamic system. We assume that the system's steady state is described by the Gibbs state $\gamma_\beta$ at inverse temperature $\beta>0$, i.e.,
\begin{equation}
\gamma_\beta=\frac{1}{Z}e^{-\beta H}, 
\end{equation}
where $Z=\tr{e^{-\beta H}}$ and $H$ is the Hamiltonian of the system. Let $\{h_i\}_{i=0}^{d-1}$ be the orthonormal basis of the system's Hilbert space that diagonalizes $H$, i.e.,\ $H=\sum_{i=0}^{d-1}E_i|h_i\>\<h_i|$ where $E_i$ are the eigenvalues of $H$. This means
\begin{equation}
\gamma_\beta=\sum_{i=0}^{d-1} \lambda^\beta_i |h_i\>\<h_i|, 
\end{equation}
where $\lambda^\beta_i=e^{-\beta E_i}/Z$. We define the probability vector $\lambda^\beta:=(\lambda^\beta_0,\ldots,\lambda^\beta_{d-1})$. From now on, we assume that $E_i>0$ for $i=0,\ldots,d-1$; we are free to assume this as the zero energy can be shifted if necessary.

We are interested now in the question when a state of this thermodynamic system can be transformed into another with a Gibbs-preserving channel. Recall that a quantum channel (completely positive trace-preserving linear map) $\Phi$ whose input and output systems coincide with this thermodynamic system is Gibbs-preserving if $\Phi(\gamma_\beta)=\gamma_\beta$. When $\Phi(\rho)=\sigma$ where $\Phi$ is a Gibbs-preserving channel, we say that $\rho$ {\it thermally majorizes} $\sigma$ and denote $\rho\rgeq_\beta\sigma$. Thermal majorization of $\rho$ over $\sigma$ in large samples means that there is $n\in\N$ such that $\rho^{\otimes n}\rgeq_\beta\sigma^{\otimes n}$, i.e.,\ there is a channel $\Phi_n$ on $n$ copies of the system such that $\Phi_n(\rho^{\otimes n})=\sigma^{\otimes n}$ and $\Phi_n(\gamma_\beta^{\otimes n})=\gamma_\beta^{\otimes n}$. Catalytic thermal majorization means that there are states $\tau^{(1)},\tau^{(2)}$ on possibly another system (e.g.,\ a heat bath), where we may interpret $\tau^{(2)}$ as the Gibbs state of the other system, and a channel $\Phi$ such that $\Phi(\rho\otimes\tau^{(1)})=\sigma\otimes\tau^{(1)}$ and $\Phi(\gamma_\beta\otimes\tau^{(2)})=\gamma_\beta\otimes\tau^{(2)}$. We are particularly interested in asymptotic catalytic thermal majorization: for any $\varepsilon>0$, there are catalysts $\tau_\varepsilon^{(1)},\tau_\varepsilon^{(2)}$ on some system (which may depend on $\varepsilon$), a state $\sigma_\varepsilon$ on our system of interest such that $\|\sigma-\sigma_\varepsilon\|_1<\varepsilon$, and a channel $\Phi_\varepsilon$ such that $\Phi_\varepsilon(\rho\otimes\tau_\varepsilon^{(1)})=\sigma_\varepsilon\otimes\tau_\varepsilon^{(1)}$ and $\Phi_\varepsilon(\gamma_\beta\otimes\tau_\varepsilon^{(2)})=\gamma_\beta\otimes\tau_\varepsilon^{(2)}$.

We concentrate on the case where the states commute with the Gibbs state. If the energy spectrum is nondegenerate, i.e.,\ there are no repeating eigenvalues $E_i$, then $[\rho,\gamma_\beta]=0$ if and only if $\rho$ diagonalizes in the energy eigenbasis, i.e.,
\begin{equation}\label{eq:diagonalizable}
\rho=\sum_{i=0}^{d-1}\lambda_i |h_i\>\<h_i|.
\end{equation}
Even if the energy spectrum is degenerate, we will focus on states which are diagonalizable in the energy eigenbasis, i.e.,\ those $\rho$ as in \eqref{eq:diagonalizable} with $\lambda_i\geq0$, $\lambda_0+\cdots+\lambda_{d-1}=1$. To explicitly refer to the state, we will associate the above $\{h_i\}_i$-diagonal state $\rho$ with the probability vector $\lambda^\rho=(\lambda^\rho_0,\ldots,\lambda^\rho_{d-1})$ where
\begin{equation}
\lambda^\rho_i=\lambda_i=\<h_i|\rho\,h_i\>,\qquad i=0,\ldots,d-1.
\end{equation}


If $\rho$ and $\sigma$ are $\{h_i\}_i$-diagonal, then one can easily see that $\Phi(\rho)=\sigma$ for a channel $\Phi$ if and only if $T\lambda^\rho=\lambda^\sigma$ where $T=(T_{i,j})_{i,j=0}^{d-1}$ is a stochastic matrix given by $T_{i,j}=\<h_i|\Phi(|h_j\>\<h_j|)h_i\>$. This means that, for $\{h_i\}_i$-diagonal states $\rho$ and $\sigma$, we have $\rho\rgeq_\beta\sigma$ if and only if there is a stochastic matrix $T$ such that $T\lambda^\rho=\lambda^\sigma$ and $T\lambda^\beta=\lambda^\beta$, i.e.,
\begin{equation}
\big(\lambda^\rho,\lambda^\beta\big)\rgeq (\lambda^\sigma,\lambda^\beta).
\end{equation}
Thus, we arrive at matrix majorization in the case $d=2$ where the second column dominates the first one since we assume that each energy $E_i$ in the energy spectrum occurs with non-vanishing probability $e^{-\beta E_i}/Z$. Moreover, we have that $\rho$ thermally majorizes $\sigma$ in the asymptotic catalytic regime if and only if, for all $\varepsilon>0$, there are catalytic probability vectors $r_\varepsilon^{(1)},r_\varepsilon^{(2)}$ (of undetermined finite length), $\lambda_\varepsilon\in\mc P_d$ such that $\|\lambda^\sigma-\lambda_\varepsilon\|_1<\varepsilon$ and a stochastic matrix $T_\varepsilon$ such that $T_\varepsilon(\lambda^\rho\otimes r_\varepsilon^{(1)})=\lambda_\varepsilon\otimes r_\varepsilon^{(1)}$ and $T_\varepsilon(\lambda^\beta\otimes r_\varepsilon^{(2)})=\lambda^\beta\otimes r_\varepsilon^{(2)}$. Thus, asymptotic catalytic thermal majorization can be recast as asymptotic catalytic matrix majorization which we have exhaustively characterized in Section \ref{sec:2column}.

Consider now states $\rho$ and $\sigma$ both diagonalizable in the energy eigenbasis $\{h_i\}_i$. If $\rho$ and $\sigma$ are both of full rank (as is the Gibbs state $\gamma_\beta$), then, according to Corollary 24 of \cite{farooq2024}, $\rho$ asymptotically catalytically thermally majorizes $\sigma$ if and only if $D_\alpha(\lambda^\rho\|\lambda^\beta)\geq D_\alpha(\lambda^\sigma\|\lambda^\beta)$ and $D_\alpha(\lambda^\beta\|\lambda^\rho)\geq D_\alpha(\lambda^\beta\|\lambda^\sigma)$ for all $\alpha\geq 1/2$. Now we can also address the case of general rank. If $\rho$ is not of full rank, then, according to Proposition \ref{pro:poweruniversaldcgeneral}, $(\lambda^\rho,\lambda^\beta)$ is a power universal of $S^2_{\rm d.c.}$ and, according to Theorem \ref{thm:asymptocatalytic}, $\rho$ asymptotically catalytically thermally majorizes $\sigma$ if and only if only $D_\alpha(\lambda^\rho\|\lambda^\beta)\geq D_\alpha(\lambda^\sigma\|\lambda^\beta)$ for all $\alpha\geq0$. In this case, we have $D_\alpha(\lambda^\beta\|\lambda^\rho)=\infty$ for $\alpha\geq 1$. Because also
\begin{align}
D_\alpha(\lambda^\beta\|\lambda^\rho)&=\frac{\alpha}{1-\alpha}D_{1-\alpha}(\lambda^\rho\|\lambda^\beta)\geq \frac{\alpha}{1-\alpha}D_{1-\alpha}(\lambda^\sigma\|\lambda^\beta)=D_\alpha(\lambda^\beta\|\lambda^\sigma)
\end{align}
for all $\alpha\in (0,1)$, we see that the asymptotic catalytic thermal majorization of the state $\rho$ of not full rank over $\sigma$ is equivalent with exactly the same inequalities as in the full-rank case. The case where $\rho$ asymptotically catalytically thermally majorizes $\sigma$ while $\rho$ is of full support and $\sigma$ is not is impossible. We can see this as follows: Suppose that $\rho$ asymptotically catalytically thermally majorizes $\sigma$. Because $D_\alpha$ are monotones, we now have that $D_\alpha(\lambda^\rho\|\lambda^\beta)\geq D_\alpha(\lambda^\sigma\|\lambda^\beta)$ for all $\alpha\geq0$. As $\alpha\to0$, it follows that
\begin{equation}
\sum_{i\in{\rm supp}\,\lambda^\rho} e^{-\beta E_i}\leq \sum_{i\in{\rm supp}\,\lambda^\sigma} e^{-\beta E_i}.
\end{equation}
Because $E_i>0$ for all $i$, this means that $|{\rm supp}\,\lambda^\rho|\leq|{\rm supp}\,\lambda^\sigma|$, i.e.,\ the support of $\rho$ cannot be larger than the support of $\sigma$.

Note that, whenever $\rho$ diagonalizes in the energy eigenbasis, then $D(\lambda^\rho\|\lambda^\beta)=\tilde{D}_\alpha(\rho\|\gamma_\beta)$ where $\tilde{D}_\alpha$ is the `sandwiched' quantum R\'{e}nyi divergence, i.e.,
\begin{equation}\label{eq:sandwiched}
\tilde{D}_\alpha(\rho\|\tau)=\left\{\begin{array}{ll}
-\log{\tr{\sigma\,{\rm supp}\,\rho}}&\text{if}\ \alpha=0\ {\rm and}\ {\rm supp}\,\rho\cap{\rm supp}\,\sigma\neq\{0\}\\
\frac{1}{\alpha-1}\log{\tr{\big(\sigma^{\frac{1-\alpha}{2\alpha}}\rho\sigma^{\frac{1-\alpha}{2\alpha}}\big)^\alpha}}&\text{if}\ \alpha\in(0,1) \ {\rm and}\ {\rm supp}\,\rho\cap{\rm supp}\,\sigma\neq\{0\}\\
\tr{\rho(\log{\rho}-\log{\sigma})}&\text{if}\ \alpha=1\ {\rm and}\ {\rm supp}\,\rho\subseteq{\rm supp}\,\sigma\\
\frac{1}{\alpha-1}\log{\tr{\big(\sigma^{\frac{1-\alpha}{2\alpha}}\rho\sigma^{\frac{1-\alpha}{2\alpha}}\big)^\alpha}}&\text{if}\ \alpha\in(1,\infty)\ {\rm and}\ {\rm supp}\,\rho\subseteq{\rm supp}\,\sigma\\
\big\|\sigma^{-1/2}\rho\sigma^{-1/2}\big\|&\text{if}\ \alpha=\infty\ {\rm and}\ {\rm supp}\,\rho\subseteq{\rm supp}\,\sigma\\
\infty&{\rm otherwise}
\end{array}\right.,
\end{equation}
where $\|\cdot\|$ is the operator norm and we use ${\rm supp}\,\rho$ for both the supporting subspace of $\rho$ and for the orthogonal projection onto this subspace. Let us define the free energies $F^\pm_{\alpha,\beta}$ through
\begin{equation}
F^+_{\alpha,\beta}(\rho):=\tilde{D}_\alpha(\rho\|\gamma_\beta),\quad F^-_{\alpha,\beta}(\rho):=\tilde{D}_\alpha(\gamma_\beta\|\rho)
\end{equation}
for all $\alpha\geq1/2$. Thus, the above considerations yield the following result.
\begin{corollary}\label{cor:thermal}
Suppose that the eigen-energies $E_i>0$ for all $i$, and $\rho$ and $\sigma$ diagonalize in the energy eigenbasis. Then $\rho$ asymptotically catalytically thermally majorizes $\sigma$ if and only if
\begin{equation}
F^\pm_{\alpha,\beta}(\rho)\geq F^\pm_{\alpha,\beta}(\sigma),\qquad\forall\alpha\geq\frac{1}{2}.
\end{equation}
\end{corollary}
Note also that, when $\rho$ and $\sigma$ are as in Corollary \ref{cor:thermal} above, then the optimal rate at which i.i.d.\ copies of $\rho$ can be transformed into $\sigma$ with Gibbs-preserving channels is
\begin{equation}
r\left(\rho\overset{\beta}{\to}\sigma\right)=\min\left\{\frac{F^+_{\alpha,\beta}(\rho)}{F^+_{\alpha,\beta}(\sigma)},\frac{F^-_{\alpha,\beta}(\rho)}{F^-_{\alpha,\beta}(\sigma)}\,\middle|\,\alpha\in[1/2,\infty]\right\},
\end{equation}
a direct consequence of Corollary \ref{cor:ratemin}. As mentioned in the Introduction, although this result was previously claimed in \cite{Brandao2014}, our present proof avoids some technical issues in the original proof. We would also like to highlight the fact that, instead of using the result on vector majorization proven in \cite{klimesh2007inequalities}, our semiring methods allow us to directly study the problem from the matrix majorization view without having to embed the problem in the vector majorization form as was done in \cite{Brandao2014}. 

Our results also allow us to study multiple thermal majorization: for $d$-tuples $\big(\rho^{(1)},\ldots,\rho^{(d)}\big)$ and $\big(\sigma^{(1)},\ldots,\sigma^{(d)}\big)$ of states, we denote
\begin{equation}
\big(\rho^{(1)},\ldots,\rho^{(d)}\big)\rgeq_\beta \big(\sigma^{(1)},\ldots,\sigma^{(d)}\big)
\end{equation}
if there is a Gibbs-preserving channel $\Phi$ (i.e.,\ $\Phi(\gamma_\beta)=\gamma_\beta$) such that $\Phi\big(\rho^{(k)}\big)=\sigma^{(k)}$ for $k=1,\ldots,d$. This majorization can naturally also be studied in the large-sample, catalytic, and asymptotic regimes, the latter one arguably being the most physically relevant case. If we assume that all the states $\rho^{(k)}$ and $\sigma^{(k)}$ diagonalize in the energy eigenbasis, we are naturally again back to matrix majorization with the dominating column $\lambda^\beta$. We point out though that, in order to derive necessary and sufficient conditions for asymptotic thermal majorization in this case, $\big(\lambda^{\rho^{(1)}},\ldots,\lambda^{\rho^{(d)}},\lambda^\beta\big)$ has to be a power universal of the semiring $S^{d+1}_{\rm d.c.}$. Conditions for this are quite strict, as we have seen in Proposition \ref{pro:poweruniversaldcgeneral}. Thus, there remains plenty of unanswered questions in multiple thermal majorization.

\section{Conclusion and outlook}

We have identified sufficient conditions for matrix majorization in large samples (i.e.,\ having many copies of the matrices) as well as in the catalytic regime with varying support conditions. This is in contrast to the earlier work \cite{farooq2024}, where the columns within each matrix were assumed to have a common support. We have seen that the conditions can be written in the form of inequalities involving multivariate generalizations of the bivariate R\'{e}nyi divergences $D_{\underline{\alpha}}$ but, unlike in the common-support case, the support restrictions greatly affect the set of multivariate divergences to be considered. Moreover, additional conditions arise at the boundaries where the $\underline{\alpha}$-parameter space is cut. We have also derived sufficient conditions for asymptotic large-sample and catalytic matrix majorization in the case of different support conditions. Many of these conditions are also necessary. In particular, our results completely resolve the problem of asymptotic catalytic thermal majorization in the field of quantum thermodynamics also studied in \cite{Brandao2014}.

One issue with our results is that the sufficient conditions for the majorization of a matrix $P$ over another matrix $Q$, both with all columns summing to one, in the large-sample or catalytic setting, and both in the exact and asymptotic regime, is that $P$ has to be a power universal of the preordered semiring we are studying. In the case of exact large-sample or catalytic majorization (Theorems \ref{thm:majorization} and \ref{thm:majorizationdominatingcolumn}), the strict ordering of the values of the homomorphisms already imply that $P$ is power universal. However, in the case of asymptotic large-sample or catalytic majorization (Theorems \ref{thm:genasymptocatalytic} and \ref{thm:asymptocatalytic}), we have to require $P$ to be power universal. Being power universal is typically quite restrictive, as highlighted in Propositions \ref{lem:poweruniversalconditions} and \ref{pro:poweruniversaldcgeneral}. We suggest the following strategy to overcome this: If $P$ is not a power universal with respect to either of the semirings we are studying in this work, we should redefine the semiring so that all the matrices in this new semiring are required to have the same support restrictions satisfied by $P$. As an example, $P=\big(p^{(1)},\ldots,p^{(d)}\big)$ might be such that $p^{(d)}$ dominates all other columns and ${\rm supp}\,p^{(1)}=\ldots={\rm supp}\,p^{(d-1)}$. We then consider the semiring consisting of all matrices satisfying the same support conditions. If we also remove possible repetitions of the same columns, $P$ is now a power universal of the new semiring and we obtain tighter conditions for majorization of $P$ over other matrices. More generally, given any $P$, we specify for any ordered pair of columns $(k,k^\prime)$, $k,k^\prime\in[d]$, whether either ${\rm supp}\,p^{(k)}\subseteq{\rm supp}\,p^{(k^\prime)}$, or ${\rm supp}\,p^{(k)}\supseteq{\rm supp}\,p^{(k^\prime)}$, or ${\rm supp}\,p^{(k)}={\rm supp}\,p^{(k^\prime)}$, or there is no condition. Then we define a new semiring $S_{\rm res}^d$ consisting of all matrices with the same support conditions. Extrapolating the results we found for the conditions for any $U\in S_{\rm res}^d$, with all columns summing to one, to be power universal in Propositions \ref{lem:poweruniversalconditions} and \ref{pro:poweruniversaldcgeneral}, and Lemma 12 in \cite{farooq2024} (the case of minimal restrictions, one dominating column, and equal supports, respectively), we expect that each condition on the supports of two columns will give rise to a corresponding condition for $U$, according to the following table. 

\vspace{3mm}
\begin{center}
    \begin{tabular}{|c|c|}
        \hline
        Condition on $P\in S_{\rm res}^d$&Condition on power universal $U\in S_{\rm res}^d$\\
        \hline
        ${\rm supp}\,p^{(k)}\subseteq{\rm supp}\,p^{(k^\prime)}$&${\rm supp}\,u^{(k)}\subsetneq{\rm supp}\,u^{(k^\prime)}$\\
        ${\rm supp}\,p^{(k)}\supseteq{\rm supp}\,p^{(k^\prime)}$&${\rm supp}\,u^{(k)}\supsetneq{\rm supp}\,u^{(k^\prime)}$\\
        ${\rm supp}\,p^{(k)}={\rm supp}\,p^{(k^\prime)}$&$u^{(k)}\neq u^{(k^\prime)}$\\
        no condition&${\rm supp}\,u^{(k)}\nsubseteq{\rm supp}\,u^{(k^\prime)}$ and ${\rm supp}\,u^{(k^\prime)}\nsubseteq{\rm supp}\,u^{(k)}$\\
        \hline
    \end{tabular}
\end{center}
\vspace{3mm}

In both semirings studied in this paper, we were able to show that for any matrix $U$ with all columns summing to one, 
\begin{equation}\label{eq:poweruniversalconjecture}
    U\text{ is power universal}\quad\Longleftrightarrow\quad\Phi(U)\neq1\quad\forall\,\Phi\text{ nondegenerate homomorphism.} 
\end{equation}
In proving this, we used conditions on $U$ similar to the ones in the table above. Hence, we conjecture that \eqref{eq:poweruniversalconjecture} also holds for a semiring $S_{\rm res}^d$ with other suppport restrictions. Going even further, \eqref{eq:poweruniversalconjecture} might be true for more general semirings (i.e. not just the matrix semirings considered here). 

As explained above, given any matrix $P$ that does not have repeated columns, we can tailor the semiring $S_{\rm res}^d$ such that $P$ is power universal in that semiring. We expect that the homomorphisms and derivations corresponding to $S_{\rm res}^d$ are similar to the ones found in this paper, precisely those for parameters $\underline{\alpha}$ or $\underline{\beta}$ that are well-defined, i.e. those where there does not appear a factor $\left(q^{(k)}_i\right)^\alpha$ or $\left(q^{(k)}_i\right)^\beta$ with $\alpha,\beta<0$ where $q^{(k)}_i$ vanishes for at least one matrix $Q$ in the semiring. In this way, we would be able to find the conditions for large-sample and catalytic majorization of $P$ over $Q$, both for the exact and asymptotic case, for $P$ with any kind of support structure, not just the ones that are power universal in the semirings discussed here or in \cite{farooq2024}. 



From a practical point of view, the notion of asymptotic matrix majorization in the large-sample or catalytic regime, as characterized in Theorems \ref{thm:genasymptocatalytic} and \ref{thm:asymptocatalytic}, is the most relevant case. After all, it is often enough to know that large-sample or catalytic majorization only holds approximately, possibly with vanishing error. Our current methods are geared towards identifying exact majorization in large samples or catalytically and often yield extra conditions which are not needed in the asymptotic picture, as we have seen. Finding conditions directly for the desired asymptotic results without having to take a scenic route through the theory and results for exact majorization would obviously be of high interest.

There is ever-increasing interest also in multivariate quantum divergences \cite{mosonyi2024geometric,nuradha2024multivariate}. However, although we may define various multivariate quantum divergences and often even show that they have desirable properties, e.g.,\ monotonicity under operating with the same quantum channel on all the states (data processing inequality) and additivity under tensor products, characterization of all the quantum divergences with these (and possibly other) properties seems to be a daunting task, even in the bivariate case \cite{Tomamichel2016Book}. Despite these difficulties, we may still define semirings of $d$-tuples $\vec{\rho}=\big(\rho^{(1)},\ldots,\rho^{(d)}\big)$ of quantum states with varying conditions for the supports ${\rm supp}\,\rho^{(k)}$, where the addition and multiplication are defined through the componentwise direct sum and tensor product respectively. These semirings can be preordered by denoting
\begin{equation}
\big(\rho^{(1)},\ldots,\rho^{(d)}\big)\rgeq \big(\sigma^{(1)},\ldots,\sigma^{(d)}\big)
\end{equation}
when there is a quantum channel $\Phi$ such that $\Phi\big(\rho^{(k)}\big)=\sigma^{(k)}$ for $k=1,\ldots,d$. We may still be able to identify relevant subsets of monotone homomorphisms and derivations on these semirings yielding conditions for large-sample quantum majorization. 

Despite the issues we face in the case of quantum majorization, various types of multivariate quantum divergences have been presented in literature. For a (non-extreme) probability vector $\underline{\alpha}\in\R_+^d$ and any quantum state tuple $\vec{\rho}$, one can minimize the weighted sum $\sum_{k=1}^d \alpha_k\delta_2\big(G,\rho^{(k)}\big)$ over all positive semi-definite operators $G$ where $\delta_2(\rho,\sigma)^2=\sum_j\big(\log{\lambda_j(\rho^{-1}\sigma)}\big)^2$ is the squared Euclidean distance between $\rho$ and $\sigma$ and $\lambda_j(\tau)$ are the singular values of $\tau$. The unique minimizer $G_{\underline{\alpha}}(\vec{\rho})$ of this problem is called the {\it (weighted) matrix mean}. According to \cite{Bhatia_Karandikar_2012}, whenever $\Phi$ is a quantum channel (a completely positive trace-preserving linear map), then
\begin{equation}
G_{\underline{\alpha}}\big(\Phi\big(\rho^{(1)}\big),\ldots,\Phi\big(\rho^{(d)}\big)\big)\geq G_{\underline{\alpha}}\big(\rho^{(1)},\ldots,\rho^{(d)}\big).
\end{equation}
Using this, it easily follows that $\Phi_{\underline{\alpha}}$,
\begin{equation}
\Phi_{\underline{\alpha}}(\vec{\rho})=\tr{G_{\underline{\alpha}}(\vec{\rho})},
\end{equation}
is a monotone homomorphism into $\R_+^{\rm op}$; the additivity and multiplicativity properties are easily proven. The corresponding divergences in case $d=2$ have a closed form introduced by Kubo and Ando \cite{Kubo_Ando_79}. Also particular barycentric constructions can be used to create multivariate divergences from bivariate ones \cite{mosonyi2024geometric}. Furthermore the matrix means and the bivariate {\it $\alpha$-$z$ quantum divergences} $D_{\alpha,z}$ (see, e.g.,\ \cite{Audenaert_Datta_2015}) have been used to define further multivariate quantum divergences \cite{bunth2023,Furuya_et_al_2023,mosonyi2024geometric}: Given $\alpha$ and $z$ such that
\begin{equation}\label{eq:alphazedcond}
0<\alpha<1,\ \max\{\alpha,1-\alpha\}\leq z\quad{\rm or}\quad\alpha>1,\ \max\{\alpha/2,\alpha-1\}\leq z\leq\alpha, 
\end{equation}
and a (non-extreme) probability vector $\underline{\alpha}$, one can define the multivariate quantum divergence $D_{\underline{\alpha},\alpha,z}$,
\begin{equation}
D_{\underline{\alpha},\alpha,z}(\vec{\rho})=D_{\alpha,z}\big(\rho^{(1)}\big\|G_{\underline{\alpha}}(\vec{\rho})\big).
\end{equation}
In the classical (commuting) case, $D_{\underline{\alpha},\alpha,z}$ reduces to the classical multivariate divergence corresponding to the homomorphism $\Phi_{\underline{\alpha'}}$ where $\alpha'_1=\alpha+(1-\alpha)\alpha_1$ and $\alpha'_k=(1-\alpha)\alpha_k$ for $k=2,\ldots,d$. In the quantum case, the map $\log{\big((\alpha-1)D_{\underline{\alpha},\alpha,z}(\cdot)\big)}$ is a monotone $\R_+^{\rm op}$-homomorphism when $\alpha<1$ and a monotone $\R_+$-homomorphism when $\alpha>1$. This is just one particular way of constructing multivariate quantities from bivariate quantities and matrix means.

The case of varying support conditions in quantum majorization is a much more complex issue than in the classical setting. The supports of quantum states are subspaces of the underlying Hilbert space which we can identify with the orthogonal projections onto those subspaces. The projection lattice of a Hilbert space is a far more complex structure than what we encounter in the classical setting. Yet, we can make some observations. If, e.g.,\ we consider majorization of state dichotomies of common support, we clearly arrive at conditions involving quantum divergences extending the notion of the classical R\'{e}nyi divergences $\big(p,q\big)\mapsto D_\alpha\big(p\big\|q\big)$ and $\big(p,q\big)\mapsto D_\alpha\big(q\big\|p\big)$ for $\alpha\in[1/2,\infty]$. However, if we allow ${\rm supp}\,\rho\subsetneq{\rm supp}\,\sigma$ in a state tuple we arrive at quantum extensions of $\big(p,q\big)\mapsto D_\alpha\big(p\big\|q\big)$ for $\alpha\in[0,\infty]$. Especially at $\alpha=0$, we have the quantum min-divergence $D_{\rm min}$,
\begin{equation}
D_{\rm min}\big(\rho\big\|\sigma\big)=-\log{\tr{\big({\rm supp}\,\rho\big)\sigma}},
\end{equation}
where ${\rm supp}\,\rho$ is identified with the projection onto the supporting subspace of $\rho$. This exemplifies a special case where varying support restrictions affect the conditions for large-sample quantum majorization. Another case is where we only require the supports of a quantum dichotomy to have a non-trivial intersection but neither of the supports needs to be included in the other. In this case, we naturally arrive at quantum extensions of the classical divergences $\big(p,q\big)\mapsto D_\alpha\big(p\big\|q\big)$ and $\big(p,q\big)\mapsto D_\alpha\big(q\big\|p\big)$ for $\alpha\in[0,1/2]$. At $\alpha=1$, we now also have the inversed min-divergence $D'_{\rm min}$,
\begin{equation}
D'_{\rm min}\big(\rho\big\|\sigma\big)=D_{\rm min}\big(\rho\big\|\sigma\big)=-\log{\tr{\big({\rm supp}\,\sigma\big)\rho}}.
\end{equation}
We would like to point out that, in this latter case, we can identify the minimal and maximal divergences \cite{matsumoto2018,Tomamichel2016Book}. These are essentially the sandwiched and Kubo-Ando divergences respectively. However, our relaxed support conditions introduce some technicalities which must be properly addressed. Yet, the following partly conjectural discussion on this special case should not be too far from the truth.

Let us define the set $\mc D$ defined similarly as in Section \ref{rem:rates} using the monotone homomorphisms and derivations of this (minimal restrictions) quantum semiring. As the normalizing power universal we shall use the qutrit dichotomy $u=(\rho_0,\sigma_0)$,
\begin{equation}
\rho_0:=\frac{1}{2}(|0\>\<0|+|1\>\<1|),\quad\sigma_0:=\frac{1}{2}(|0\>\<0|+|2\>\<2|).
\end{equation}
For any $\alpha\in[0,1]$, let us denote by $\mc D_\alpha$ the set of those $D\in\mc D$ such that $D$ reduces (up to normalization) to the R\'{e}nyi divergence $D_\alpha$ on commuting dichotomies. For any pair $\big(\rho,\sigma\big)$ of states, let us define the {\it shorted operator} $\rho_{\ll\sigma}$ \cite{Anderson_Trapp_1975},
\begin{equation}
\rho_{\ll\sigma}:=Q\rho Q-Q\rho\big(Q^\perp\rho Q^\perp\big)^{-1}\rho Q
\end{equation}
where $Q:={\rm supp}\,\sigma$ and the inverse is realized on the support of the operator. We now define $D^{\rm min}_\alpha,D^{\rm max}_\alpha\in\mc D_\alpha$ for all $\alpha\in[0,1]$ according to
\begin{align}
D^{\rm min}_\alpha\big(\rho\big\|\sigma\big)=-\log{\tr{\Big(\sigma^{1/2}\rho^{\frac{\alpha}{1-\alpha}}\sigma^{1/2}\Big)^{1-\alpha}}}
\end{align}
when $0\leq\alpha\leq 1/2$ and, for $1/2\leq\alpha\leq 1$, $D_\alpha^{\rm min}$ is defined as above but with arguments reversed and
\begin{align}
D^{\rm max}_\alpha\big(\rho\big\|\sigma\big)=-\log{\tr{\sigma^{1/2}\Big(\sigma^{-1/2}\rho_{\ll\sigma}\sigma^{-1/2}\Big)^\alpha\sigma^{1/2}}}.
\end{align}
Note that, up to normalization, $D^{\rm min}_\alpha$ are the sandwiched quantum R\'{e}nyi divergences of (\ref{eq:sandwiched}) (symmetrized around $\alpha=1/2$) and $D^{\rm max}$ are the Kubo-Ando divergences. We now (likely) have, for all $D\in\mc D_\alpha$ and $\alpha\in[0,1]$,
\begin{equation}
D^{\rm min}_\alpha\big(\rho\big\|\sigma\big)\leq D\big(\rho\big\|\sigma\big)\leq D^{\rm max}_\alpha\big(\rho\big\|\sigma\big).
\end{equation}
Using Theorem \ref{thm:Fritz2022} and this (still partly conjectural) observation, we obtain a sufficient (but not likely very tight) condition for large sample quantum majorization: Consider pairs $\big(\rho,\sigma\big)$ and $\big(\rho',\sigma'\big)$ of states where
\begin{equation}
{\rm supp}\,\rho\,\wedge\,{\rm supp}\,\sigma\neq\{0\}\neq{\rm supp}\,\rho'\,\wedge\,{\rm supp}\,\sigma',
\end{equation}
${\rm supp}\,\rho\subsetneq{\rm supp}\,\sigma$, and ${\rm supp}\,\sigma\subsetneq{\rm supp}\,\rho$ (we conjecture that this makes $\big(\rho,\sigma\big)$ a power universal of the semiring). If
\begin{equation}
D^{\rm min}_\alpha\big(\rho\big\|\sigma\big)>D^{\rm max}_\alpha\big(\rho'\big\|\sigma'\big)\qquad\forall\alpha\in[0,1],
\end{equation}
then, for any sufficiently large $n\in\N$, there is a quantum channel $\Phi_n$ such that
\begin{equation}
\Phi_n\Big(\rho^{\otimes n}\Big)=\big(\rho'\big)^{\otimes n},\quad\Phi_n\Big(\sigma^{\otimes n}\Big)=\big(\sigma'\big)^{\otimes n}.
\end{equation}

\section*{Acknowledgements}
We would like to thank the anonymous reviewers for useful comments and suggestions on our first draft of this paper. This project is supported by the National Research Foundation, Singapore through the National Quantum Office, hosted in A*STAR, under its Centre for Quantum Technologies Funding Initiative (S24Q2d0009). 

\bibliographystyle{ultimate}
\bibliography{bibliography}
	
\end{document}